\documentclass[11pt,namelimits,sumlimits]{amsart}
\usepackage{amssymb,amsmath}
\usepackage[mathscr]{eucal}
\usepackage{comment}

   \usepackage{color}

\usepackage{psfrag}
\usepackage{enumerate} 

\usepackage{graphicx}
\usepackage{amsmath,amsthm}
\usepackage{graphics}
\usepackage{color}
\usepackage{epsfig}
\usepackage{fullpage}
\usepackage{amssymb,amsmath}
\usepackage[mathscr]{eucal}

\usepackage[usenames,dvipsnames]{pstricks}
\usepackage{epsfig}
\usepackage{pst-grad} 
\usepackage{pst-plot} 

\newtheorem{theorem}[equation]{Theorem}
\newtheorem{lemma}[equation]{Lemma}
\newtheorem{prop}[equation]{Proposition}
\newtheorem{corollary}[equation]{Corollary}

\newtheorem{definition}[equation]{Definition}

\theoremstyle{remark}
\newtheorem{remark}[equation]{Remark}
\newtheorem{notation}[equation]{Notation}

\newtheorem{assumption}[equation]{Assumption}

\numberwithin{equation}{section}

\newcommand{\Real}{\mathbb R}
\newcommand{\R}{\mathbb R}
\newcommand{\Sph}{\mathbb{S}}

\newcommand{\calF}{\mathcal{F}}
\newcommand{\calK}{\mathcal{K}}
\newcommand{\func}[1]{\ensuremath{\mathrm{#1} \:} }

\newcommand{\supp}[0]{\func{supp}}

\newcommand{\sech}[0]{\func{sech}}

\newcommand{\dist}[0]{\mathrm{dist}}

\newcommand{\xX}[0]{\mathbf{x}}
\newcommand{\yY}[0]{\mathbf{y}}

\newcommand{\Bv}[0]{\mathbf{v}}
\newcommand{\Be}[0]{\mathbf e}
\newcommand{\Bvp}[0]{\mathbf{v}}

\newcommand{\Pe}[0]{{\mathbf p}_{ \tau_0[e]}}
\newcommand{\Pde}[0]{{\mathbf p}_{\tau_d[e]}}
\newcommand{\Pdo}[0]{  {{\mathbf p}_\tau}}

\newcommand{\Pim}[0]{\hat{\mathbf p}_\tau}
\newcommand{\Pimd}[0]{\hat{\mathbf p}_{\tau_d[e]}}
\newcommand{\bunder}{\underline{b}}
\newcommand{\Ctilde}{\widetilde C}
\newcommand{\utau}{\underline{\tau}}
\newcommand{\ud}{\underline{d}}

\newcommand{\taue}[0]{\tau_d [e]}

\newcommand{\Sm}[0]{S^-}
\newcommand{\Smext}[0]{{\widetilde S}^-}
\newcommand{\Sp}[0]{S^+}
\newcommand{\Spext}[0]{{\widetilde S}^+}

\newcommand{\mathin}[0]{{\mathrm{in}}}
\newcommand{\mathout}[0]{{\mathrm{out}}}

\newcommand{\pe}[0]{[p,e]}
\newcommand{\ppe}[0]{[p^+[e],e]}
\newcommand{\pme}[0]{[p^-[e],e]}
\newcommand{\pen}[0]{[p,e,m]}
\newcommand{\penp}[0]{[p,e,m]}
\newcommand{\penm}[0]{[p,e,m]}

\newcommand{\Gtdtl}[0]{\Gamma(\tilde d,\tilde \ell)}

\newcommand{\epdl}[0]{[e; \tilde d,\tilde \ell]}

\newcommand{\zetabold}[0]{{\boldsymbol\zeta}}
\newcommand{\dbold}{{\mathbf{d}}}

\newcommand{\dzeta}[0]{d,\boldsymbol \zeta}
\newcommand{\td}[0]{{\tau_d[e]}}
\newcommand{\tz}[0]{{\tau_0[e]}}
\newcommand{\RRR}[0]{\mathsf R}
\newcommand{\TTT}[0]{\mathsf T}
\newcommand{\UUU}[0]{\mathsf U}

\newcommand{\hYtdz}[0]{ Y_{d,{\boldsymbol \zeta}}}
\newcommand{\Htdz}[0]{H_{d,{\boldsymbol \zeta}}}

\newcommand{\Ntdz}[0]{N_{d,\boldsymbol \zeta}}
\newcommand{\Ss}[0]{\mathbb S}
\newcommand{\delt}[0]{\epsilon}

\newcommand{\Rn}[0]{\mathbb R^{n+1}}
\newcommand{\Ssn}[0]{\mathbb S^{n-1}}

\newcommand{\bt}[0]{\Theta}

\newcommand{\sout}[0]{{s_{\mathrm{out}}}}
\newcommand{\ssin}[0]{{s_{\mathrm{in}}}}
\newcommand{\rout}[0]{r_{\mathrm{out}}}
\newcommand{\rin}[0]{r_{\mathrm{in}}}
\newcommand{\urout}[0]{{\underline r_{\mathrm{out}}}}
\newcommand{\urin}[0]{{\underline r_{\mathrm{in}}}}
\newcommand{\ur}[0]{\underline r}
\newcommand{\maxT}[0]{T}
\newcommand{\maxTG}[0]{T_\Gamma}
\newcommand{\tsd}[0]{t_d}

\newcommand{\Cout}[0]{C^{\mathrm{out}}}
\newcommand{\Cin}[0]{C^{\mathrm{in}}}

\newcommand{\ovr}[0]{r}
\newcommand{\ovk}[0]{k}

\newcommand{\RH}[0]{\mathbf P[e]}

\newcommand{\hf}[0]{{\hat f}}
\newcommand{\hF}[0]{{\hat F}}

\newcommand{\hd}[0]{\hat d}
\newcommand{\dz}{{d,\boldsymbol \zeta}}

\newcommand{\Lc}[0]{\Lambda^{\mathrm{close}}}
\newcommand{\Lf}[0]{\Lambda^{\mathrm{far}}}

\newcommand{\de}[0]{\mathfrak t_0[e]}

\newcommand{\Cn}{\widetilde \omega}
\newcommand{\CCn}{\frac{\widetilde \omega_{n-1}}{\widetilde \omega_{n}^{\frac 12}}}
\newcommand{\signep}{\mathrm{sgn}[p,e]}

\newcommand{\VS}{V_S}
\newcommand{\VSp}{V_S^+}
\newcommand{\VSm}{V_S^-}
\newcommand{\VN}{V_{\Lambda}}
\newcommand{\loc}{_{\mathrm{loc}}}

\begin{document}

\title[CMC]{Complete Constant Mean Curvature Hypersurfaces in Euclidean space of dimension four or higher}
\author[C.~Breiner]{Christine~Breiner}

\address{Department of Mathematics, Fordham University, Bronx, NY  10458}
\email{cbreiner@fordham.edu}

\author[N.~Kapouleas]{Nikolaos~Kapouleas}
\address{Department of Mathematics, Brown University, Providence,
RI 02912} \email{nicos@math.brown.edu}


\date{\today}


\begin{abstract}
In this article we provide a general construction when $n\ge3$ for immersed 
in Euclidean $(n+1)$-space, complete, smooth, constant mean curvature hypersurfaces of finite topological type 
(in short CMC $n$-hypersurfaces). 
More precisely our construction converts certain graphs in Euclidean $(n+1)$-space to CMC $n$-hypersurfaces with asymptotically Delaunay ends 
in two steps: 
First appropriate small perturbations of the given graph have their vertices replaced by round spherical regions and their edges and rays by Delaunay pieces 
so that a family of initial smooth hypersurfaces is constructed. 
One of the initial hypersurfaces is then perturbed to produce the desired CMC $n$-hypersurface which depends on the given family of perturbations of the graph 
and a small in absolute value parameter $\utau$. 
This construction is very general because of the abundance of graphs which satisfy the required conditions 
and because it does not rely on symmetry requirements. 
For any given $k\ge2$ and $n\ge3$ it allows us to realize infinitely many topological types as CMC $n$-hypersurfaces 
in $\Rn$ with $k$ ends. 
Moreover for each case there is a plethora of examples reflecting the abundance of the available graphs. 
This is in sharp contrast with the known examples which in the best of our knowledge 
are all (generalized) cylindrical obtained by ODE methods and are compact or with two ends. 
Furthermore we construct embedded examples when $k\ge3$ where the number of possible topological types for each $k$ is finite but tends to $\infty$ as $k\to\infty$. 

MSC 53A05, 53C21.
\end{abstract}

\maketitle

\nopagebreak

\section{Introduction}
\label{S:intro}

\subsection*{The general framework}
$\phantom{ab}$
\nopagebreak

Constant Mean Curvature (CMC) (hyper)surfaces in a Riemannian manifold 
can be described variationally as critical points of 
the induced intrinsic volume (or area in dimension two) functional, 
subject to an enclosed volume constraint.  
Alternatively they can be described as soap films (or fluid interfaces) 
in equilibrium under only the forces of surface tension and uniform 
enclosed pressure. 
In both cases the geometric condition is that the mean curvature $H$ 
of the hypersurface is constant as the name suggests. 

Of particular interest are the complete CMC (hyper)surfaces of finite topological type 
smoothly immersed in Euclidean spaces and in particular in Euclidean three-space.
The only classically known such examples were the round spheres, the cylinders, 
and more generally the rotationally invariant surfaces discovered by Delaunay in 1841 \cite{Delaunay}.
Two major results were proved in the 1950's characterizing the 
round two-spheres as the only closed CMC surface in Euclidean three-space,  
under the assumption of embeddedness (by Alexandrov \cite{Alexandrov}),   
or the assumption of zero genus (by Hopf \cite{Hopf}). 
These results and their methods of proof had a profound influence in Mathematics.  
They also led to the celebrated conjecture (or question according to some) 
by Hopf on whether the only immersed closed CMC surfaces in Euclidean three-space are round spheres. 
In 1986 Wente disproved the Hopf conjecture by constructing genus one closed immersed examples \cite{Wente}.

At that stage the only examples of finite topological type in Euclidean three-space 
were the classical ones and the Wente tori. 
Following a general gluing methodology developed by Schoen \cite{schoen} and N.K. \cite{KapAnn},
and using the Delaunay surfaces as building blocks,
most of the possible finite topological types were realized as immersed (or Alexandrov embedded) 
CMC surfaces for the first time \cite{KapAnn,KapJDG}. 
\cite{KapJDG} in particular settled the Hopf question for high genus closed surfaces 
by providing examples of any genus $g\ge3$. 
In spite of its success the use of Delaunay pieces as building blocks has the limitation that it 
does not allow the construction of closed genus two CMC examples. 
In \cite{KapWente} a systematic and detailed refinement of the original gluing methodology
made it possible to construct genus two (actually any genus $g\ge2$) 
closed examples with the Wente tori as building blocks.
Since then, many other gluing problems have been successfully resolved by using this refined approach. 
These results include gluing constructions for special Lagrangian cones \cite{HaskKap,HaskKap2,HaskKap3} 
and various gluing constructions for minimal surfaces 
\cite{Yang,KapYang,KapSurvey,KapClay,KapJDGd,KapI,KapII}. 

It is worth pointing out that the constructions in \cite{schoen,KapAnn,KapJDG} 
are quite general in two ways: 
First, in that each construction is reduced to finding graphs satisfying some rather general conditions. 
There is an abundance of such graphs and so a plethora of examples can be produced. 
Second, in that no symmetry is required---although it can be imposed in special cases---and indeed most examples constructed do not satisfy any symmetries. 
These constructions can serve then as a prototype for general constructions in other geometric settings---see \cite{KapClay,KapSurvey,KapG}.

We briefly mention that much progress has been made in 
the case of embedded, or more generally Alexandrov embedded, complete CMC surfaces 
of finite genus $g$ with $k$ ends. 
Meeks \cite{MeeksCMC} proved that such (noncompact) surfaces have at least two ends and all their ends are cylindrically bounded. 
Motivated by \cite{MeeksCMC,KapAnn},
Korevaar, Kusner, and Solomon \cite{KKS} showed that each end converges
exponentially fast to a Delaunay surface and if there are only two ends then the surface is Delaunay. 
Further progress in this direction was made in \cite{KoKu1,KoKu2} 
and also in understanding the moduli space of these surfaces as for example in \cite{KuMaPo}. 
Moreover a significant success was that in some cases of genus zero, 
complete classification results were obtained with a satisfactory understanding of the surfaces involved \cite{GBKS,GBKSII,GBKKRS}. 

We briefly also mention that various constructions extended the results of \cite{KapAnn}:  
Gro{\ss}e-Brauckmann \cite{GB} used a conjugate surface construction to construct genus zero examples with $k$ ends 
under maximal ($k$-fold dihedral) symmetry, including examples with large neck size for the first time.
Various gluing constructions related to non-degeneracy \cite{MaPa,Ratzkin,MaPaPo,MaPaPoClay,RosCosin,JleliPacard} were developed in certain cases,  
which allowed some new examples, 
in particular examples with asymptotically cylindrical ends \cite{MaPaPo}, 
with noncatenoidal necks used as nodes instead of spheres \cite{MaPa}, 
and a modified construction (end-to-end gluing) of the closed CMC examples \cite{Ratzkin,JleliPacard}. 
Recently the construction and estimates in \cite{KapAnn} were refined in \cite{BKLD} 
by applying the improved methodology of \cite{KapWente}.  
This way a large class of embedded examples was produced.  
\cite{BKLD} served also as in intermediate step for developing the high-dimensional constructions presented in this article. 

Contrary to the case of Euclidean three-space very little is currently known in the case of higher-dimensional Euclidean spaces: 
Rotationally invariant CMC hypersurfaces analogous to the ones found by Delaunay have been constructed \cite{Kenmotsu}. 
In 1982 Hsiang \cite{Hsiang} demonstrated that the theorem of Hopf does not extend to higher dimensions by constructing immersed CMC hyperspheres that are not round. 
Jleli has studied moduli spaces \cite{JleliMS} and has developed an end-to-end gluing construction 
\cite{JleliE2E} which will provide new symmetric closed examples \cite{JleliCompact} when \cite{JleliToappear} appears. 
He also constructed examples bifurcating from the Delaunay-like ones \cite{Jleli_bifurcate}.

Finally we briefly mention that constructions of CMC hypersurfaces have also been carried out in compact ambient
manifolds under certain metric restrictions. 
Ye \cite{Ye} provided the first such example,
proving that there exists a foliation by CMC hyperspheres in a neighborhood
of a non-degenerate critical point of the scalar curvature. 
Pacard and Xu \cite{PacardXu} partially extended Ye's result.
Mazzeo and Pacard extended Ye's result to geodesic tubes \cite{MazzeoPacardTubes}. 
Further constructions of 
CMC surfaces (two-dimensional) condensing around geodesic intervals or rays 
were provided in \cite{ButscherMazzeo}, 
and for CMC hypersurfaces condensing around higher dimensional submanifolds 
in \cite{MahMazPa}.

\subsection*{Brief discussion of the results}
$\phantom{ab}$
\nopagebreak

In this article we extend the results of \cite{KapAnn} to higher dimensions, 
that is to the construction of CMC $n$-dimensional hypersurfaces in Euclidean $(n+1)$-space for $n\ge3$. 
Note that although the present proof and construction work for $n=2$ with small appropriate modifications, 
we restrict our attention to $n>2$ to simplify the presentation.  
For the same reason we restrict our attention to the construction of CMC hypersurfaces of finite topological type. 

Our constructions as in \cite{KapAnn,BKLD} are based on a suitable family of graphs $\calF$ which consists 
of small perturbations of a central graph $\Gamma$ (see \ref{FamilyDefinition}). 
Our graphs have vertices, edges, rays, and nonzero weights assigned to the edges and the rays (see \ref{D:graph}). 
$\Gamma$ is balanced in the sense that the resultant forces exerted on the vertices by the edges and rays vanish 
(see \ref{Vpdefn} and \ref{deltadefn}), and moreover its edges have even integer lengths. 
The other graphs in $\calF$ have approximately prescribed resultant forces (unbalancing condition) and 
prescribed small changes of the lengths of the edges (flexibility condition). 

Given $\calF$ and a small nonzero $\utau$ a family of initial immersions is constructed, 
where the image of each such immersion is built around a properly chosen $\Gamma'\in\calF$, 
and consists of unit spheres (with small geodesic balls removed) centered at the vertices of $\Gamma'$,  
and appropriately perturbed Delaunay pieces of parameter $\utau$ times the corresponding weight of $\Gamma'$. 
We have then the following. 

\begin{theorem}[Main Theorem]
Given a family of graphs $\mathcal F$, there exists $\maxT(\mathcal F) >0$ such that for all $0<|\utau| \leq \maxT$, 
there exists a $\Gamma' \in \mathcal F$ and an immersion built around $\Gamma'$ as outlined above 
which admits a small graphical perturbation which has mean curvature $H\equiv1$. 
Moreover the immersion is an embedding if the central graph $\Gamma$ satisfies certain conditions (see \ref{D:pre}) 
and $\utau>0$.
\end{theorem}

Note that the conditions in \ref{D:pre} are the expected ones, 
that is they ensure that the various pieces stay away from each other and that the Delaunay pieces are embedded. 
It is easy then to realize infinitely many topological types as immersed complete CMC surfaces with $k$ ends, 
where any $k\ge2$ can be given in advance. 
These constructions (when no symmetries are imposed) 
have $(k-1)(n+1)-\binom{n+1}2+\binom{n+1-k}2$ continuous parameters, 
reflecting thus the asymptotics of the $k$ Delaunay ends. 
Moreover there is further great variety in the immersions of a given number of ends and topological type reflected by the central graphs $\Gamma$ we can choose.  

We can restrict our attention to embedded examples. 
In this case we could find examples with $k\ge3$ and then we have only finitely  many topological 
types for each $k$, with the number of topological types for each $k$ tending to $\infty$ as $k\to\infty$. 

Finally we remark that in ongoing work we plan to extend these results to the compact case 
in the manner of \cite{KapJDG} extending \cite{KapAnn}.

\subsection*{Outline of the approach}
$\phantom{ab}$
\nopagebreak

The construction in this article is an extension to high dimensions of the constructions in \cite{KapAnn,BKLD} 
with \cite{BKLD} serving also as an intermediate step in the development of this article. 
The main difficulties and their resolution in extending to high dimensions are the following: 
\\
(1). A careful understanding of the geometry and analysis of the Delaunay hypersurfaces in 
high dimensions is needed, 
which to the best of our knowledge is new at least at this level of detail. 
In particular understanding their periods requires some work and is similar to 
work for special Legendrian submanifolds \cite{HaskKapCAGpq,HaskKap3}. 
\\
(2). The conformal covariance of the Laplacian in dimension two is not available anymore. 
Moreover the linearized operator in dimension two can be formulated with respect to 
a conformal metric $h=\frac12 |A|^2g$ which compactifies the catenoidal necks in the limit 
and actually converts the catenoidal necks of the Delaunay surfaces into 
spherical regions isometric to the actual spherical regions, 
introducing thus new symmetries which did not exist in the induced metric $g$; 
all of this is unavailable in high dimensions. 
We resolved this difficulty by 
understanding the linearized equation on the catenoidal necks 
using Fourier decompositions on the meridians and some $L^2$ estimates.  
This is a simpler version of the approach in the analysis of the linearized 
equation on the (complicated and only approximately rotationally invariant) 
necks in \cite{KapWente}. 
Note also that since we cannot compactify the necks we use appropriate weighted estimates. 
\\ 
(3). 
Since we do not use the end-to-end gluing idea which simplifies at the expense 
of limiting the scope of the construction, 
we still have to use the ideas of \cite{KapAnn}, modified for the high dimensions, 
to understand the linearized equation on the central---where 
the fusion with the Delaunay pieces occurs---spherical regions. 
We also use semi-localization, that is studying the linearized equation 
on the extended standard regions and combining the results. 
\\ 
(4). 
Because of the generality of the construction the whole scheme is quite involved. 
We tried to carefully organize the various steps so the whole structure of 
the proof is conceptually clear and easy to follow. 
\\ 
(5). We remark also that motivated by the geometric principle we achieve much 
faster decay away from the central spherical regions (compared to \cite{KapAnn}), 
by introducing simple dislocations between the central spherical regions and 
the Delaunay pieces attached. 
\\ 
(6). Finally we remark that instead of monitoring the use of the extended substitute
kernel at each step we chose to use a balancing formula \cite{KKS} 
on the final hypersurface to estimate the unbalancing error because this seems 
to provide better control. 

\subsection*{Organization of the presentation}
$\phantom{ab}$
\nopagebreak

Appendix \ref{DelSection} contains a thorough treatment of the essential information about the geometry of Delaunay surfaces. 
Appendix \ref{quadapp} provides standard background on the quadratic error estimates. 
Finally, in Appendix \ref{annuli} we study the Dirichlet problem on a flat annulus.

Section \ref{graphs} contains a description of the family of graphs which provides the structure for the immersion of the initial surfaces. 
We discuss the unbalancing and flexibility conditions and we associate to each graph in the family two parameters $(\tilde d, \tilde \ell)$ 
which give quantitative meaning to these conditions.  
In Section \ref{BuildingBlocks}, we describe the building blocks of the construction, 
spheres with balls removed and Delaunay pieces with perturbations near their boundaries. 
The Delaunay building blocks are not necessarily CMC near their boundaries; 
the estimates are controlled by the parameters describing the perturbation. 
We are careful to describe these building blocks independent of any reference to a family of graphs. 
The building blocks depend only on general parameters and not on the structure of a graph. 
In Section \ref{DelaunayLinear} we study the linear operator $\mathcal L_g$ on compact pieces of Delaunay surfaces. 
At this stage we choose a fixed large constant $\bunder \gg1$ and a small $\maxT>0$ depending on $\bunder$. 
For any $0<|\tau|<\maxT$, we consider regions on a Delaunay immersion with parameter $\tau$. 
The size of the regions considered depend upon $\tau$ and $\bunder$ and the choice of $\bunder, \tau$ 
along with our understanding of the geometry of Delaunay surfaces provide good geometric estimates. 
Again, 
the statements and proofs of this section do not reference or rely on a graph or family of graphs.

In Section \ref{InitialSurface} we construct a family of initial surfaces which depend upon a parameter $\utau$ 
and a pair of parameters $(d,\boldsymbol \zeta)$. 
We presume a given family of graphs $\mathcal F$. 
The parameter $\utau$ satisfies $0<|\utau| <\maxTG$ where $\maxTG$ depends upon $T$ and the graph $\Gamma$ but not on the structure of $\Gamma$.  
The parameters $(d,\boldsymbol \zeta)$ and $\utau$ determine $(\tilde d, \tilde \ell)$ and thus a graph in the family $\Gamma'$. 
We build the initial surface by positioning and fusing building blocks at designated locations given by the structure of $\Gamma'$. 
The parameters describing the building blocks are encoded in $\utau$, $(d,\boldsymbol \zeta)$ and the graph (but not the structure) of $\Gamma$. 

In Section \ref{GlobalSection} we study the linearized operator on the family of initial surfaces. 
We define the extended substitute kernel and solve the modified linear problem. 
Section \ref{GeometricPrinciple} contains the prescribing of substitute and extended substitute kernel. 
We prove the Main Theorem in Section \ref{MThm} using a fixed point theorem.

\subsection*{Preliminaries}

\begin{definition}\label{scalednorms}
For $k \in \mathbb N\cup \{0\}$, $\beta \in (0,1)$, a domain $\Omega$ in a Riemannian manifold, $u \in C^{k,\beta}\loc(\Omega)$, and $f,\rho:\Omega \to \Real^+$ we define the norm
\[
\|u:C^{k,\beta}(\Omega, \rho, g, f)\| :=\sup_{x \in \Omega}f(x)^{-1}\|u:C^{k,\beta}(B_x\cap \Omega, \rho^{-2}(x)g)\|.
\]
Here $B_x$ is a geodesic ball centered at $x$ with radius $1/10$ in the metric $\rho^{-2}(x)g$. 
For simplicity, when $\rho=1$ or $f=1$ we may omit them from the notation.
\end{definition}
Note from the definition that
\[
\|\nabla u:C^{k-1,\beta}(\Omega, \rho,g,\rho^{-1}f)\| \leq \|u:C^{k,\beta}(\Omega, \rho,g,f)\|
\]and
\[
\|u_1u_2:C^{k,\beta}(\Omega, \rho,g, f_1f_2)\| \leq C(k)\|u_1:C^{k,\beta}(\Omega, \rho,g,f_1)\| \, \|u_2:C^{k,\beta}(\Omega,\rho,g,f_2)\|.
\]
\begin{definition}
If $a,b>0$ and $c>1$, then we write
\[
a \sim_c b
\]if $a \leq cb$ and $b \leq ca$.
\end{definition}
Throughout this paper we make extensive use of cut-off functions, and thus we adopt the following notation:  Let $\Psi:\Real \to [0,1]$ be a smooth function such that
\begin{enumerate}
 \item $\Psi$ is non-decreasing
\item $\Psi \equiv 1$ on $[1,\infty)$ and $\Psi \equiv 0$ on $(-\infty, -1]$
\item $\Psi-1/2$ is an odd function.
\end{enumerate}
For $a,b \in \Real$ with $a \neq b$, let $\psi[a,b]:\Real \to [0,1]$ be defined by $\psi[a,b]=\Psi \circ L_{a,b}$ where $L_{a,b}:\Real \to \Real$ is a linear function with $L(a)=-3, L(b)=3$.
Then $\psi[a,b]$ has the following properties:
\begin{enumerate}
 \item $\psi[a,b]$ is weakly monotone.
\item $\psi[a,b]=1$ on a neighborhood of $b$ and $\psi[a,b]=0$ on a neighborhood of $a$.
\item $\psi[a,b]+\psi[b,a]=1$ on $\Real$.
\end{enumerate}

\begin{notation}
\label{NT}
For $X$ a subset of 
a Riemannian manifold $(M,g)$ we write 
$\dbold^{M,g}_X$ for the distance function from $X$ in $(M,g)$.  
For $\delta>0$ we define a tubular neighborhood of $X$ by
$$
D^{M,g}_X(\delta):=\left \{p\in M:\dbold^{M,g}_X(p)<\delta\right\}.  
$$
In both cases we may omit $M$ or $g$ if understood from the context and 
if $X$ is finite we may just enumerate its points. 
\end{notation}

\subsection*{Acknowledgments}
CB was supported in part by National Science Foundation grants DMS-1308420 and DMS-1609198. 
This material is also based upon work supported by the National Science Foundation under Grant No. DMS-1440140 
while CB was in residence at the Mathematical Sciences Research Institute in Berkeley, California, during the Spring 2016 semester.
NK would like to thank the Mathematics Department and the MRC at Stanford University
for providing a stimulating mathematical environment and generous financial support during Fall 2011, Winter 2012 and Spring 2016.
NK was also partially supported by NSF grants DMS-1105371 and DMS-1405537.  

\section{Finite Graphs} 
\label{graphs} 
The gluing construction carried out in this article uses round spheres and pieces of Delaunay surfaces 
to build initial hypersurfaces which are then perturbed to become CMC hypersurfaces. 
The parameters of the Delaunay pieces and the positioning of the spheres and the Delaunay pieces 
are naturally encoded by graphs. 
In this article for simplicity we restrict ourselves to finite graphs which we discuss in this section. 
The initial graph we use should satisfy all of the relations one expects for a singular CMC surface 
and thus we impose a balancing restriction on each vertex and a restriction on the length of each edge. 
We first define the kind of graphs we will be using:

\begin{definition}[Graphs] 
\label{D:graph} 
We define a finite graph in $\Rn$ for some $n>2$ to be a collection 
$\{V(\Gamma),E(\Gamma), R(\Gamma), \hat \tau\}$ such that 
\begin{enumerate}
\item $V(\Gamma) \subset \Rn$ is a finite collection of vertices. 
\item $E(\Gamma)$ is a finite collection of edges in $\Rn$, each with its two endpoints in $V(\Gamma)$.
\item $R(\Gamma)$ is a finite collection of rays in $\Rn$, each with its one endpoint in $V(\Gamma)$.
\item $\hat \tau: E(\Gamma) \cup R(\Gamma) \to \Real \backslash \{0\}$ is a function.
\end{enumerate}
\end{definition}

\begin{notation}
Given a finite graph $\Gamma$, the input of a function or vector valued function of $V(\Gamma), E(\Gamma), R(\Gamma)$ will be given by $[\cdot ]$. 
\end{notation}

\begin{definition}[Edge and Vertex Relations]
\label{vedef} 
Let $E_p$ denote the collection of edges and rays that have $p \in V(\Gamma)$ as an endpoint.  
We have then 
\[
\bigcup_{p \in V(\Gamma)}E_p = E(\Gamma) \cup R(\Gamma).\]
We also define the set of \emph{attachments} 
\begin{equation}
A(\Gamma) :=   \{    \pe \in V(\Gamma) \times \left(E(\Gamma) \cup R(\Gamma)\right) \, : \, e \in E_p     \}    .
\end{equation} 
Finally for each 
$[p,e]\in A(\Gamma)$ we denote the unit vector pointing away from $p$ and in the direction of $e$   
by $\mathbf{v}\pe$.  
\end{definition}

\begin{definition}
\label{Def:dlz}
For a graph $\Gamma$, let $L(\Gamma)$ denote the space of functions from $E(\Gamma)$ to $\Real$, let $D(\Gamma)$ denote the space of functions from $V(\Gamma)$ to $\Rn$, and let $Z(\Gamma)$ denote the space of functions from $A(\Gamma)$ to $\Rn$.
Equip each of these spaces with the maximum norm. 
\end{definition}
\begin{definition}\label{Vpdefn}
We define $\hd[\Gamma,\cdot] =\hd\in D(\Gamma)$ such that
\begin{equation}
\label{hd_gamma_def}
\hd[\Gamma,p] = \hd[p]: =  \left(\frac{\omega_{n-1}}n\right)\left(\frac{n+1}{\omega_n}\right)^{1/2}\sum_{e \in E_p} \hat \tau[e]\Bv\pe:= 
\CCn\sum_{e \in E_p} \hat \tau[ e]\Bv\pe 
\end{equation}
measures the deviation from balancing at the vertex $p$. 
Here $\Cn_{k-1}:= \frac{\omega_{k-1}}{k}$  
and $\omega_k$ denotes as usual the $k$-dimensional volume of $\mathbb S^k \subset \Real^{k+1}$.  

We let $l[\Gamma, \cdot]  = l \in L(\Gamma)$ such that for $e \in E(\Gamma)$, $2l[e]$ equals the length of $e$.  
\end{definition}
\begin{remark}\label{Cndefn}
The constant $\Cn_{k-1}$ will arise because of various normalizations throughout the argument. 
Absorbing it into the definition of $\hd$ will be convenient later. 
\end{remark}

Our construction will be based on a family of graphs that are perturbations of some fixed graph 
which we will call the central graph $\Gamma$ (see \ref{deltadefn}). 
The idea of the construction is to replace 
each edge or ray $e$ of $\Gamma$ by a Delaunay piece of parameter 
$\utau \hat \tau[e]$, where  $\utau$ is a sufficiently small global parameter.  
(See Section \ref{BuildingBlocks} for a description of the Delaunay pieces.) 
The construction of the initial surfaces 
requires appropriate small perturbations of $\Gamma$ 
depending on $\utau$ and on other parameters.  
The central graph $\Gamma$ will be the limit of the graphs employed as $\utau\to0$.  
In this limit our surfaces will tend to tangentially touching unit spheres.  
Correspondingly, 
the period of the Delaunay surfaces will tend to $2$. 
Therefore $\Gamma$ has to satisfy the condition that its edges have even integer length. 
Moreover the balancing conditions satisfied by CMC surfaces 
(see \ref{unbalancinglemma}, \eqref{forcevec}, \eqref{force})  
imply the vanishing of $\hd$ on $\Gamma$. 
These considerations motivate the following definition. 

\begin{definition}
\label{deltadefn}
Let $\Gamma$ be a finite graph. 
If $\hd[p] = 0$ for all $p \in V(\Gamma)$, we say $\Gamma$ is a \emph{balanced} graph. 
We call $\Gamma$ a \emph{central} graph if $\Gamma$ is balanced and $l[e] \in \mathbb N$ for all $e \in E(\Gamma)$. 
\end{definition}

Finally, we define central graphs that guarantee that our construction produces an \emph{embedded} CMC hypersurface: 
 \begin{definition}[Pre-embedded graphs] 
\label{D:pre} 
We say $\Gamma$ is \emph{pre-embedded} if it is a central graph with $\hat \tau :E(\Gamma)\cup R(\Gamma) \to \Real^+$ and 
 \begin{enumerate}
  \item For all $p \in V(\Gamma)$ and all $e_i \neq e_j \in E_p$, $\angle(\Bv[p,e_i] , \Bv[p,e_j]) \geq \pi/3$, where $\angle(\xX,\yY)$ 
measures the angle between the two vectors $\xX,\yY$.
\item For all $e,e' \in E(\Gamma) \cup R(\Gamma)$ that do not share any common endpoints, the Euclidean distance between $e,e'$ is greater than $2$.
\item For any two rays $e, e' \in R(\Gamma)$, $1-\Bv\pe \cdot \Bv[p',e'] > 0$. 
 \end{enumerate}
\end{definition}
For a pre-embedded $\Gamma$ and sufficiently small $\utau$, each of the initial surfaces constructed from one of the possible perturbations of $\Gamma$ is embedded. In the singular setting, when $\utau=0$, the angle condition between edges and rays about a fixed vertex allows for a singular surface with unit spheres touching tangentially. We do not require a strict inequality for this condition since the change in the period for small $\utau$ (on the order $|\utau|^{\frac 1{n-1}}$) dominates both the radius change and the changes we allow via unbalancing and dislocation (on the order $|\utau|$). The second item requires a strict inequality as the maximum radius of an embedded Delaunay surface is on the order $1- \utau\hat \tau + O(\utau^2)$ but we allow for the edges to move with order $\underline C |\utau \hat \tau|$ where $\underline C$ can be quite large.
The final condition also requires a strict inequality. Indeed if the central graph $\Gamma$ has two parallel rays pointing into the same half-plane, then the family of graphs on which we base our initial surfaces may include graphs with intersecting rays.

\subsection*{Deforming the graphs}
Given a central graph $\Gamma$, we will consider perturbations of this graph subject to parameters $\tilde d,\tilde \ell$. 
We need the perturbations to be smoothly dependent on the parameters and are thus interested in graphs $\Gamma$ 
which can be deformed in this way.

\begin{definition}[Isomorphic graphs]
\label{n1} 
We define two graphs as isomorphic if there exists a one-to-one correspondence between the vertices, edges, and rays, 
such that corresponding rays and edges emanate from the corresponding vertices.
For convenience we will often use the same letter to denote corresponding objects for isomorphic graphs.  
Using this correspondence, for $\tilde \Gamma$ isomorphic to $\Gamma$, 
we identify $D(\tilde \Gamma), L(\tilde \Gamma), Z(\tilde \Gamma)$ with $D(\Gamma), L(\Gamma), Z(\Gamma)$ respectively. 
\end{definition}

We proceed to define the function $\ell$, which quantifies the length change of each edge for a perturbation of $\Gamma$. 
\begin{definition}
Given a graph $\Gamma_1$ isomorphic to a central graph $\Gamma$, 
we define $\ell[\Gamma_1,\cdot] \in L(\Gamma)$ such that (following \ref{n1}) for all $e \in E(\Gamma) \approx E(\Gamma_1)$, 
\begin{equation}\label{first_ell_def}
\ell[\Gamma_1,e] := l[\Gamma_1,e] - l[\Gamma,e], 
\end{equation} 
and therefore the length of the edge of $\Gamma_1$ corresponding to $e\in E(\Gamma)$ is 
$$ 
2l[\Gamma_1,e] = 2 l[\Gamma,e] + 2\ell[\Gamma_1,e].  
$$
\end{definition}

\begin{definition}[Families of graphs] 
\label{FamilyDefinition}
We define a {\it family of graphs} 
$\mathcal{F}$ to be a collection of graphs 
parametrized by $(\tilde d,\tilde\ell)\in B_{\mathcal{F}}$  
such that the following hold: 
\begin{enumerate}
 \item $\Gamma := \Gamma(0,0)$ is a central graph in the sense of \ref{deltadefn} and 
$B_{\mathcal{F}}$ is a small ball about $(0,0)$ in $D(\Gamma) \times L(\Gamma)$.  
\item $\Gamma(\tilde d,\tilde\ell)$ is isomorphic to $\Gamma(0,0)$ and depends smoothly on $(\tilde d,\tilde\ell)$.
\item Following \ref{n1}, $ \hd[\Gamma(\tilde d,0) , \cdot]=\tilde d[\cdot] $  (unbalancing condition).
\item Following \ref{n1}, $\ell[\Gamma (\tilde d,\tilde\ell) , \cdot] =\tilde \ell[\cdot]$ (flexibility condition).
\item 
$ \hat \tau[\Gamma(\tilde d,0),.] = \hat \tau[\Gamma(\tilde d,\tilde\ell),.]$.
\end{enumerate}
\end{definition}

Note that by the above definition each $\Gamma(\tilde d,0)$ with $\tilde d \neq 0$ 
is a modification of the central graph that is unbalanced as prescribed by $\tilde d$ while the lengths
of the edges remain unchanged. 
Perturbing $\Gamma(\tilde d,0)$ to $\Gamma(\tilde d,\tilde \ell)$ is achieved by changing the lengths of the edges as prescribed by $\tilde \ell$. 
Note that by \ref{FamilyDefinition}.5 
$\hat\tau$ is unmodified under this perturbation. 
However, $\hd[\Gamma(\tilde d,0), \cdot]$ is not necessarily equal to $\hd[\Gamma(\tilde d, \tilde \ell), \cdot]$, 
as the edges may rotate to accommodate the changes in edge length.

\begin{definition}\label{Rnframe}
Throughout the paper, let $\{\Be_1, \dots, \Be_{n+1}\}$ denote the standard orthonormal basis of $\Real^{n+1}$.
\end{definition}
We now choose a frame associated to each edge in the graph $\Gamma$ and use this frame to determine a frame on each edge for any graph in $\calF$.

\begin{definition}
\label{gammaframe}
For $e \in E(\Gamma)$ we choose once for all one of its endpoints to call $p^+[e]$.  
We call then its other endpoint $p^-[e]$  
and we define $\mathrm{sgn}[p^{\pm} [e] ,e]:=\pm1$. 
For $e \in E(\Gamma) \cup R(\Gamma)$ 
we choose once and for all an ordered, positively oriented orthonormal frame 
$F_\Gamma [e]=\{\mathbf{v}_{1}[e], \dots, \mathbf{v}_{n+1}[e]\}$, 
such that $\Bv_1[e]=\Bv\pe $, 
where $p$ is the endpoint of $e$ if $e \in R(\Gamma)$ 
and $p=p^+[e]$ if $e \in E(\Gamma)$. 
We have therefore when $[p,e]\in A(\Gamma)$ and $e \in E(\Gamma)$ 
$$
\Bv[p^+[e],e] = \Bv_1[e]=-\Bv[p^-[e],e]
\quad\text{ and } \quad 
\signep= \Bv\pe  \cdot \Bv_1[e].
$$
\end{definition}

\begin{definition}\label{rotationdefn}
Given two unit vectors $\xX, \yY \in \Rn$ such that $\angle(\xX,\yY) < \pi/2$,  
let $\RRR[\xX, \yY]$ denote the unique rotation defined in the following manner.

\begin{itemize}
\item 
 If $\xX = \yY$, take $\RRR[\xX,\yY]$ to be the identity.
\item If $\xX \neq \yY$, set $\xX \cdot \yY = \cos a$ and $\mathbf v_y:= \frac{\yY-\xX\cos a }{\sin a}$. 
We define $\RRR[\xX,\yY]$ to be the rotation in the plane given by $\xX,\yY$ that rotates $\xX$ to $\yY$, 
that is in closed form 
\[
\RRR[\xX,\yY] = I + \sin a \left(\mathbf v_y \xX^T - \xX \mathbf v_y^T\right) + (1-\cos a)\left(\xX \xX^T+\mathbf v_y \mathbf v_y^T\right). \]
\end{itemize}
\end{definition}
\begin{lemma}\label{smoothrotation} The rotation $\RRR[\xX,\yY]$ depends smoothly on $\xX$ and $\yY$.
\end{lemma}
\begin{proof}

Simplifying the expression, using the definition of $\mathbf v_y$, we observe that for $\xX \neq\yY$,
\[
\RRR[\xX,\yY]=I +(\yY\xX^T-\xX\yY^T)+ (1-\cos a)\,\xX\xX^T+ \frac {1}{1+\cos a}\left(\yY\yY^T - \cos a\,( \yY\xX^T +\xX \yY^T) + \cos^2 a\, \xX\xX^T\right).
\] 
This expression is clearly smooth in $\xX,\yY$.
\end{proof}
 
For $\pe \in A(\Gamma)$ and $[p',e']$ the corresponding attachment on an isomorphic graph, let
\[
\angle(e,e'):= \arccos(\Bv\pe\cdot \Bv[p',e']).
\]

We use the rotation defined above to describe an orthonormal frame on the edges and rays of any graph in the family $\calF$. 
By the smooth dependence on $\tilde d,\tilde \ell$, and the presumed smallness of their norms, 
$\angle(e,e') <\pi/2$ for $e \in E(\Gamma) \cup R(\Gamma)$ and $e'$ a corresponding edge or ray on any graph in the family. 
It follows that the rotation we need will always be well-defined.

\begin{definition}
\label{FrameLemma}
For $\Gamma(\tilde d, \tilde \ell) \in \calF$ 
with $\calF$ as in  
\ref{FamilyDefinition}, 
given $e \in E(\Gamma) \cup R(\Gamma)$ we define an orthonormal frame $F_{\Gamma(\tilde d,\tilde\ell)}[e]=\{\Bvp_1\epdl, 
\dots, \Bvp_{n+1}\epdl\}$
uniquely by requiring the following:
\begin{enumerate}
\item $\Bvp_1\epdl=\Bv[\Gtdtl, p^+[e],e]$.
\item $\Bvp_i\epdl=\RRR[\Bv_1[e], \Bvp_1\epdl](\Bv_i[e])$ for $i=2, \dots, n+1$.
\end{enumerate}
\end{definition}

\begin{remark}
 $F_{\Gtdtl}[e]$ depends smoothly on $\tilde d,\tilde \ell$.
\end{remark}

\section{The Building Blocks}
\label{BuildingBlocks}
The initial hypersurfaces we construct will be built out of appropriately fused pieces of spheres and perturbed Delaunay hypersurfaces. 
The positioning of these pieces and the parameter of each Delaunay piece is determined by the graphs of $\calF$ and the parameters $\dz$. 
The building blocks however can be described independently of any reference to the graphs of $\calF$. 
To highlight this fact, we first develop the immersions of the building blocks to depend upon other general parameters not related to any graph. 
In Section \ref{InitialSurface} we use these immersions to produce a family of hypersurfaces from a family of graphs $\calF$, 
where each hypersurface will depend on the central graph $\Gamma$ of $\calF$ as well as the parameters $d, \zetabold$. 

\subsection*{Spherical building blocks}
Let $Y_0: \Real \times \Ssn \to \Ss^n \subset \Real^{n+1}$  
be as in \ref{Y0}. 
Immediately we see that
\[
g_0 = \sech^2 t(dt^2 + g_{\Ssn}), \qquad \: |A|^2 = n, \qquad \: H \equiv 1.
\]

\begin{definition}
\label{adef}
Let $\delta'$ be a small positive constant which we will choose in 
\ref{def:a2}
and define $a>0$ by 
$\tanh (a+1) =  \cos \left(\delta'\right)$. 
Note that $Y_0(\{a+1\}\times \Ssn) = \partial D_{(1,\mathbf 0)}^{\mathbb{S}^{n}} (\delta')   \subset \mathbb{S}^{n}$ (recall \ref{NT}). 
\end{definition}

We determine now sphere diffeomorphisms that will be used to guarantee that the immersion is well-defined.
First we define 
a rotation $\hat \RRR[F,F']$ which maps $F$ to $F'$ 
for a given orthonormal frame $F$ and a perturbation $F'$ of $F$. 

\begin{definition}
Let $F:=\{\xX_1, \dots, \xX_{n+1}\}, F':=\{\yY_1, \dots, \yY_{n+1}\}$ be two orthonormal frames of $\mathbb R^{n+1}$ with the same orientation. We define
$\hat \RRR[F, F']:\Rn \to \Rn$ to be the unique rotation such that
\begin{align*}
  \hat \RRR[ F,  F']( \xX_i) = \yY_i.
  \end{align*}

\end{definition}

We now define a map on $\mathbb S^n$ that consists of $m$ local frame transformations and smoothly transits to the identity map away from these transformations. In application, the first vector in each frame will describe the positioning of an edge on a graph $\Gtdtl \in \calF$.

\begin{definition}[Spherical Building Blocks]
\label{defn:sphere}
We assume given two sets of positively oriented ordered orthonormal frames 
$W=\{F_1 , F_2,\dots , F_m\}$ and $W' = \{F'_1, F'_2, \dots, F'_m\}$, 
where 
\begin{equation*} 
\begin{gathered} 
F_i=\{\xX_{1,i},\xX_{2,i}, \dots, \xX_{n+1,i}\}, 
\qquad 
F_i' =\{\yY_{1,i}, \dots, \yY_{n+1,i}\}, 
\\
\angle(\xX_{1,i}, \xX_{1,j})>16\delta' \quad \forall i \neq j,  
\qquad 
\angle (\xX_{1,i}, \yY_{1,i}) \leq (\delta')^2 \quad \forall i .
\end{gathered} 
\end{equation*} 
That is, the first vectors in each frame of $W$ are 
not close, 
while the first vectors in each pair of frames $F_i, F_i'$ are close.
We define then a family of diffeomorphisms $\hat Y[ W,  W']:\Ss^n\to \Ss^n \subset \Real^{n+1}$, 
smoothly dependent on $ W, W'$,  
by

\begin{equation*}
\hat Y[ W,W'](x):= 
\left\{\begin{array}{ll}x& \text{for } x \in \Ss^n\backslash \bigsqcup_i D^{\Ss^n}_{\xX_{1,i}}({4\delta'}),  
\\
\frac{ \psi_W(x) \, x + (1-\psi_W(x)) \, \hat \RRR[F_i, F_i'](x) }  
{ \,\left| \psi_W(x) \, x + (1-\psi_W(x)) \, \hat \RRR[F_i, F_i'](x) \right|\, }  
& \text{for } x \in D^{\Ss^n}_{\xX_{1,i}}({4\delta'})\backslash  D^{\Ss^n}_{\xX_{1,i}}({3\delta'}),
\\
 \hat \RRR[F_i,  F_i'](x)& \text{for } x \in   D^{\Ss^n}_{\xX_{1,i}}({3\delta'}),
\end{array}\right.
\end{equation*}
where $\psi_W:=\psi[3\delta',4\delta']\circ \dbold^{\Sph^n}_{\{\xX_{1,1},\, \xX_{1,2},\,  \dots,\,  \xX_{1,n+1}\}}$.  
\end{definition}

\noindent 

\subsection*{Delaunay building blocks}

We now describe a general immersion of an appropriately perturbed Delaunay piece. 
For a description of Delaunay immersions, see Section \ref{DelSection}. 
Throughout this subsection, let $a$ be the value defined in \ref{adef}, let $l \in \mathbb Z^+$, 
and let $\Pdo$ and $\Pim$ be as in \ref{dPim} so that 
$2\Pdo$ is the domain period and $2\Pim$ the translational period of a Delaunay hypersurface of parameter $\tau$. 
We presume throughout that $0<\maxT \ll 1$ is a constant chosen sufficiently small to guarantee 
that all immersions are smooth and well-defined and that all error estimates will hold as stated. 
Finally, we let $\underline C$ denote a possibly large constant that is independent of $\maxT$.

\begin{definition}
Let
 $\psi_{\mathrm{dislocation}^\pm}, \psi_{\mathrm{gluing}^\pm}: [a, 2\Pdo l -a] \to\Real$ be cutoff functions such that:
\begin{itemize}
\item $\psi_{\mathrm{dislocation}^+}=\psi[a+2,a+1]$,
\item $\psi_{\mathrm{dislocation}^-}=\psi[2\Pdo l-(a+2),2\Pdo l-(a+1)]$,
\item $\psi_{\mathrm{gluing}^+}=\psi[a+3,a+4]$,
\item $\psi_{\mathrm{gluing}^-}=\psi[2\Pdo l-(a+3),2\Pdo l-(a+4)]$.
\end{itemize}
\end{definition}
With these cutoff functions, we define the building blocks. 
Notice that $Y_0$ is the embedding of $\Ss^n$ defined in \eqref{Y0} and $Y_\tau$ is the Delaunay immersion defined in \eqref{DelImm}.
\begin{definition}
\label{defn:Yedge}
Given $\tau,l,a$ with $0<|\tau|\leq \maxT$ and $\boldsymbol\zeta^\pm\in \Real^{n+1}$ with $0\leq| \boldsymbol\zeta^\pm|\leq\underline C |\tau|$, we define two smooth immersions
$ Y_{\mathrm{edge}}[\tau,l,\boldsymbol \zeta^+,\boldsymbol \zeta^-]:[a, 2\Pdo l -a] \times  \Ssn\to\Rn$ and $ Y_{\mathrm{ray}}[\tau,\boldsymbol \zeta^+]:[a, \infty) \times  \Ssn\to\Rn$ such that, for $x=(t,\bt)$,
\begin{align*} Y_{\mathrm{edge}}[\tau,l,\boldsymbol \zeta^+,\boldsymbol \zeta^-](x)=& \psi_{\mathrm{dislocation}^+}(t) \cdot \left( Y_0(x)+\boldsymbol \zeta^+\right)\\
&+(1-\psi_{\mathrm{dislocation}^+}(t))(1-\psi_{\mathrm{gluing}^+}(t))Y_0(x)\\ 
&+\psi_{\mathrm{gluing}^+}(t) \cdot \psi_{\mathrm{gluing}^-}(t) \cdot Y_\tau(x)\\
&+(1-\psi_{\mathrm{dislocation}^-}(t))(1-\psi_{\mathrm{gluing}^-}(t))Y_0^-(x)  \\
&  + \psi_{\mathrm{dislocation}^-}(t) \cdot \left(Y_0^-(x)+ \boldsymbol \zeta^-\right)                                   
\end{align*}
\begin{align*}
 Y_{\mathrm{ray}}[\tau,\boldsymbol \zeta^+](x)= &\psi_{\mathrm{dislocation}^+}(t) (Y_0(x)+ {\boldsymbol \zeta}^+ )+(1-\psi_{\mathrm{dislocation}^+}(t))(1-\psi_{\mathrm{gluing}^+}(t))Y_0(x)\\
 &+\psi_{\mathrm{gluing}^+}(t) \cdot Y_\tau(x)
\end{align*}
where $Y_0^-(x)= Y_0(t-2\Pdo l,\bt) +(2+2\Pim)l \Be_1$.
\end{definition}To aid the reader, we describe the geometry of the $Y_{\mathrm{edge}}$ immersion in some detail.  For $t \in [a,a+1]$, the image is a geodesic
hyperannulus sitting on a unit sphere with the sphere centered at $\boldsymbol \zeta^+$. The annulus is centered at $\boldsymbol \zeta^+ + \Be_1$ with inner radius $\delta'$. When $t \in [a+1,a+2]$, the immersion 
smoothly interpolates between
the annular region on the dislocated sphere and an annular region centered at $\Be_1$ on a unit sphere centered at the origin. 
For $t \in [a+2,a+3]$, the immersion remains on the unit sphere centered at the origin, while for $t \in [a+3,a+4]$, the immersion smoothly transits between this sphere and
a Delaunay piece with parameter $\tau$. The same procedure happens toward the other end. First, the Delaunay piece transits back to a unit sphere centered
at $\left(Y_\tau(2\Pdo l,\Theta)\cdot \Be_1\right) \Be_1$. This position represents the location of the end of
a Delaunay piece with parameter $\tau$ and $l$ periods, with initial end at the origin. Finally,
this sphere transits to a unit sphere centered at $\boldsymbol \zeta^- + \left(Y_\tau(2\Pdo l,\Theta)\cdot \Be_1\right) \Be_1$, a dislocation of $\boldsymbol \zeta^-$ from the previously
described sphere.

Of course, the $Y_{\mathrm{ray}}$ immersion has the same behavior as $Y_{\mathrm{edge}}$ near the origin. The only difference is that the Delaunay immersion continues out to infinity and there is no transiting back to a sphere.

\begin{prop}\label{geopropcentral}
Let $g:= Y_{\mathrm{edge}}^*(g_{\Rn})$ or $g:=Y_{\mathrm{ray}}^*(g_{\Rn})$ as the situation dictates. For a fixed, large constant $x>a+5$,
\[
\| Y_{\mathrm{edge}}[\tau,l,{\boldsymbol \zeta}^+,{\boldsymbol \zeta}^-]-Y_0:C^k((a, x) \times \Ssn, g)\| \leq C(k,x)(|\boldsymbol \zeta^+| + |\tau|)
\]
\[
\| Y_{\mathrm{edge}}[\tau,l,{\boldsymbol \zeta}^+,{\boldsymbol \zeta}^-]-Y_0^-:C^k((2\Pdo l -x, 2\Pdo l-a) \times \Ssn, g)\| \leq C(k,x)(|\boldsymbol \zeta^-| + |\tau|)
\]and
\[
\| Y_{\mathrm{ray}}[\tau,{\boldsymbol \zeta}^+]-Y_0:C^k((a, x) \times \Ssn, g)\| \leq C(k,x)(|\boldsymbol \zeta^+|  + |\tau|).
\]

\end{prop}
\begin{proof}
On the region where $t \in [a+1, a+2] \cup[2\Pdo l-(a+2),2\Pdo l-(a+1)]$, 
the only difference between the immersions comes from the cutoff function applied to $\boldsymbol \zeta^\pm$, 
where the $\pm$ is appropriate for the domain. 
Thus the $C^k$ estimates on these regions are immediate.

For the other regions, we first note that the immersion $Y_0$ defines $\tanh(s)=x_1$, $\sinh(s)=\rho_0(x_1)$ from \ref{radiuslemma}. 
Using an ODE comparison for $k(t)$ and $\tanh(t)$, we can appeal to \ref{radiuslemma} to get the $C^k$ estimates for the remaining regions. 
\end{proof}

\begin{definition}
\label{Defn:Herror}
Let $H_X$ denote the mean curvature of the immersion $X:\Omega \subset \Real \times \Ssn \to \Real^{n+1}$.

Let 
\[
H_{\mathrm{dislocation}}[\tau, l,\boldsymbol \zeta^+, \boldsymbol \zeta^-],H_{\mathrm{gluing}}[\tau, l,\boldsymbol \zeta^+, \boldsymbol \zeta^-]: [a, 2\mathbf p_\tau l -a]\times \Ssn \to \Real,
\]
\[
H_{\mathrm{dislocation}}[\tau, \boldsymbol \zeta^+], H_{\mathrm{gluing}}[\tau, \boldsymbol \zeta^+]:[a,\infty) \times \Ssn\to\Real
\] such that
\begin{align*}
H_{\mathrm{dislocation}}[\tau, l,\boldsymbol \zeta^+, \boldsymbol \zeta^-](x)&:=\left\{ \begin{array}{ll} H_{ Y_{\mathrm{edge}}[\tau,l,{\boldsymbol \zeta}^+,{\boldsymbol \zeta}^-]} - 1&\text{if } x\in \left([a,a+2] \cup [2\mathbf p_\tau l -(a+2), 2\mathbf p_\tau l -a] \right) \times \Ssn,\\
 0&\text{otherwise},\end{array}\right.  
\\
H_{\mathrm{gluing}}[\tau, l,\boldsymbol \zeta^+, \boldsymbol \zeta^-](x)&:=\left\{ \begin{array}{ll} H_{ Y_{\mathrm{edge}}[\tau,l,{\boldsymbol \zeta}^+,{\boldsymbol \zeta}^-]} - 1& \text{if } x \in [a+3,a+5] \times \Ssn, \\
& \text{or if } x \in [2\mathbf p_\tau l -(a+5),2\mathbf p_\tau l -(a+3)] \times \Ssn \\
  0&\text{otherwise},\end{array}\right.  
\\
H_{\mathrm{dislocation}}[\tau, \boldsymbol \zeta^+](x)&:=\left\{ \begin{array}{ll} H_{ Y_{\mathrm{ray}}[\tau,{\boldsymbol \zeta}^+]} - 1&\text{if } x\in [a,a+2]  \times \Ssn,\\
 0&\text{otherwise},\end{array}\right.  
\\
H_{\mathrm{gluing}}[\tau,\boldsymbol \zeta^+](x)&:=\left\{ \begin{array}{ll} H_{ Y_{\mathrm{ray}}[\tau,{\boldsymbol \zeta}^+]} - 1&\text{if } x \in [a+3,a+5]  \times \Ssn, \\
  0&\text{otherwise}.\end{array}\right.  
\end{align*}
\end{definition}From these definitions and \ref{geopropcentral} we immediately bound the error on the mean curvature.
\begin{corollary}\label{Cor:Herror}
\begin{align*}
&\|H_{\mathrm{dislocation}}[\tau, l,\boldsymbol \zeta^+\boldsymbol \zeta^-]:C^{0, \beta}([a,2\mathbf p_\tau l-a]\times \Ssn,g)\|\leq C(\beta)\left( |\boldsymbol \zeta^+|+ |\boldsymbol \zeta^-|\right)\\
&\|H_{\mathrm{gluing}}[\tau, l,\boldsymbol \zeta^+\boldsymbol \zeta^-]:C^{0, \beta}([a,2\mathbf p_\tau l -a]\times \Ssn,g)\|\leq C(\beta) |\tau|\\
&\|H_{\mathrm{dislocation}}[\tau, \boldsymbol \zeta^+]:C^{0, \beta}([a,\infty)\times \Ssn,g)\|\leq C(\beta) |\boldsymbol \zeta^+|\\
&\|H_{\mathrm{gluing}}[\tau, \boldsymbol \zeta^+]:C^{0, \beta}([a,\infty)\times \Ssn,g)\|\leq C(\beta) |\tau|
\end{align*}
\end{corollary}

\begin{lemma}\label{Lemma:Hdis}
For $g$ as in \ref{geopropcentral}, $N_X$ denoting the unit normal of the immersion $X$, and $b \in (a+3, \mathbf p_\tau)$, 
\begin{align*}
&\int_{[a,b] \times \Ssn}H_{\mathrm{dislocation}}[\tau, l,\boldsymbol \zeta^+, \boldsymbol \zeta^-] N_{ Y_{\mathrm{edge}}[\tau,l, {\boldsymbol \zeta}^+, \boldsymbol \zeta^-]}  dg = 0,\\
&\int_{[2\mathbf p_\tau l -b,2 \mathbf p_\tau l -a] \times \Ssn}H_{\mathrm{dislocation}}[\tau, l,\boldsymbol \zeta^+, \boldsymbol \zeta^-] N_{ Y_{\mathrm{edge}}[\tau,l, {\boldsymbol \zeta}^+, \boldsymbol \zeta^-]}  dg = 0,\\
&\int_{[a,b] \times \Ssn}H_{\mathrm{dislocation}}[\tau, \boldsymbol \zeta^+] N_{ Y_{\mathrm{ray}}[\tau,{\boldsymbol \zeta}^+]}  dg = 0.
\end{align*}
\end{lemma}
\begin{proof}
We prove the result for the ray immersion as the others follow identically. For convenience we also remove the notation $[\tau, \boldsymbol \zeta^+]$.

First recall that  $H_{\mathrm{dislocation}}$ is supported on $[a+1, a+2]\times \Ssn$. Thus
\[
n\int_{[a,b] \times \Ssn}H_{\mathrm{dislocation}} N_{ Y_{\mathrm{ray}}}  dg =\int_{[a+ 1/2,a+5/2] \times \Ssn} nH_{ Y_{\mathrm{ray}}} dg -n \int_{[a+1/2,a+5/2] \times \Ssn} N_{ Y_{\mathrm{ray}}}  dg.
\]By the divergence theorem and since ${ Y_{\mathrm{ray}}}= Y_0 + \boldsymbol \zeta^+$ on $[a,a+1]\times \Ssn$, ${ Y_{\mathrm{ray}}}= Y_0$ on $[a+2, a+3] \times \Ssn$ the first term can be rewritten as
\begin{align*}
\int_{[a+ 1/2,a+5/2] \times \Ssn} \sum_{i=1}^{n+1} \Delta_g x_i \Be_i dg
&= \int_{\partial([a+ 1/2,a+5/2] \times \Ssn)} \sum_{i=1}^{n+1} (\nabla_g x_i \cdot \eta_{Y_{\mathrm{ray}}}) \Be_i d\sigma_g\\
&= \int_{\partial([a+ 1/2,a+5/2] \times \Ssn)} \sum_{i=1}^{n+1} (\nabla_{g_0} x_i \cdot \eta_{Y_{0}}) \Be_i d\sigma_{g_0}\\
&=\int_{[a+ 1/2,a+5/2] \times \Ssn} nH_{ Y_0} dg_0
\end{align*}where $d\sigma_g$ is the induced metric on the boundary. By similar logic, we note that
\[
 \int_{[a+1/2,a+5/2] \times \Ssn} N_{ Y_{\mathrm{ray}}}  dg= \int_{[a+1/2,a+5/2] \times \Ssn} N_{ Y_{0}}  dg_0
\] and thus
\[
n\int_{[a,b] \times \Ssn}H_{\mathrm{dislocation}} N_{ Y_{\mathrm{ray}}}  dg= n \int_{[a,b] \times \Ssn}(H_{Y_0}-1)N_{Y_0} dg_0 =0
\]
\end{proof}

\section{Linear theory on Delaunay hypersurfaces} 
\label{DelaunayLinear}
In this section, we solve semi-local linear problems on Delaunay surfaces with small parameter.
Throughout the paper we denote the linearized operator in the induced metric by $\mathcal L_g$. 
On a Delaunay immersion as described in Appendix \ref{DelSection}, by \eqref{FF} and \eqref{modA} the operator takes the form
\begin{equation} 
\label{Lg}
\mathcal L_g:=\Delta_g+|A_g|^2 = \frac 1{r^2} \partial_{tt} + \frac{n-2}{r^2}w'\partial_t + \frac 1{r^2} \Delta_{\Ssn}+n(1+(n-1)\tau^2 r^{-2n}).
\end{equation} 

\begin{assumption}
\label{ass:b}
Throughout this section, 
we will assume $\bunder\gg1$ is a fixed constant, 
chosen as large as necessary and depending only on $n$ and $\epsilon_1$, 
where $\epsilon_1$ is a small constant which depends on $\gamma \in (1,2), \beta \in (0,1)$. 
In particular, $\bunder$ is independent of the constant $\maxT>0$, which will be chosen as small as needed, in terms of $\bunder$. 
We also assume given $b\in \left(\frac9{10}\bunder,\frac{11}{10}\bunder\right)$.
Unless otherwise stated we will denote by $C$ positive constants which depend on $\bunder$ but not on $b, \maxT$. 
\end{assumption} 

\begin{definition} 
\label{domaindefinitions}
Given $0<|\tau| < \maxT$ and a Delaunay immersion $Y_\tau:\Real\times \Ssn\to \Real^{n+1}$ defined as in Appendix \ref{DelSection},  
we define the following regions on the domain:
\begin{enumerate}
\item $\Lambda_{x,y}:= [b+x,\Pdo-(b+y)] \times \Ssn$
\item $\Cout_x := \{b+x\}\times \Ssn$
\item $\Cin_y:=\{\Pdo-(b+y)\} \times \Ssn$
\item $\Sm_x:= [\Pdo-(b+x), \Pdo +(b+x)]\times \Ssn$
\item $\Smext_x:=[b+x, 2\Pdo-(b+x)]\times \Ssn$
\item $\Sp_x := [2\Pdo -(b+x),2\Pdo+(b+x)]\times \Ssn$
\item $\Spext_x:= [\Pdo+(b+x), 3\Pdo-(b+x)]\times \Ssn$
\end{enumerate} 
Here $0\le x,y < \Pdo-b$, where $\Pdo -b>0$ is guaranteed by the smallness of $\maxT$ in terms of $\bunder$. When $x=y=0$ we may drop the subscript.
\end{definition}

Notice that for $\maxT$ small enough, by \ref{radiuslemma}, 
\ref{Cat_lemma} the immersion of the region $\Sp$ has geometry roughly like $\Ss^n$ while the immersion of the region $\Sm$, 
after an appropriate rescaling, looks roughly like a catenoid. 
Following usual terminology 
we refer to these regions as \emph{standard regions} 
and we refer to $\Spext, \Smext$ as \emph{extended standard regions}. 
The extended standard regions contain one standard region and two adjacent regions with $t$-coordinate length $\Pdo -2b$. 
We have labeled one such region $\Lambda$ and we refer to $\Lambda$ as a \emph{transition} or \emph{intermediate region}. 

\subsection*{The linearized equation on the transition region}
Let $\rout, \rin$ denote the radius of the meridian spheres at $\Cout, \Cin$ respectively, in the induced metric. That is 
$$\rout = r_{\tau}(b) \text{ and }\rin = r_{\tau}(\Pdo-b).$$

We consider a flat metric on $\Lambda$ given by 
\begin{equation}\label{D:s}
g_A:= ds^2 + s^2 g_{\Ssn} 
\text{ where $s:[b,\Pdo-b] \to \Real^+$ satisfies } 
\left\{\begin{array}{l}\frac {ds}{dt} = r(t)\\ s(\Pdo-b) = \rin
\end{array} \right.
\end{equation} 
\begin{lemma}\label{lemma:rvss}
Let $\gamma\in (1,2), \beta \in (0,1)$. Given $0<\delta<\min\{\frac 1{100}, \frac 1{10n}\}$, there exists $\bunder$ large enough and $\maxT>0$ small enough depending on $\bunder$ such that for all $0<|\tau|<\maxT$, for  $\Lambda$ defined by $\tau$ and $b$ satisfying \ref{ass:b}: 

\begin{align}\label{r_metric_equiv}
\|1: C^{0, \beta}(\Lambda, r, g, r^{-2})\| &\sim_{10} \|r^{2}:C^{0,\beta}(\Lambda, r, g)\|,\notag \\
\|r^{-2n}: C^{0, \beta}(\Lambda, r, g, r^{-2})\| &\sim_{10} \|r^{2-2n}:C^{0,\beta}(\Lambda, r, g)\|
\end{align}
Moreover, for $s$ defined by \eqref{D:s},
\begin{equation}\label{soverr}
\left| \frac sr - 1\right| \leq 5\delta,
\qquad
\left| \frac{ds}{dr}-1\right|\leq 4\delta,
\qquad
\left|\frac{d^2s}{dr^2}\right| \leq \frac{C(n)}r\delta.
\end{equation}As a consequence, for any $v\in C^{k, \beta}(\Lambda)$, for $0 \leq k \leq 2$,
\begin{align}\label{uniformnorms}
{\|v:C^{k,\beta}(\Lambda,r,g, r^{\gamma-2})\|}&\sim_{10}{\|v:C^{k,\beta}(\Lambda,s,g_A,s^{\gamma-2})\|},\notag\\
{\|v:C^{k,\beta}(\Lambda,r,g, r^{-n-\gamma})\|}&\sim_{10}{\|v:C^{k,\beta}(\Lambda,s,g_A,s^{-n-\gamma})\|}.
\end{align}
\end{lemma}

\begin{proof}
Notice that the geometry of $Y_\tau$ near $t=0$ and $t=\Pdo$ (see \ref{radiuslemma}, \ref{Cat_lemma}) 
implies that by picking $\bunder$ large, independent of $\maxT$ and $\maxT$ sufficiently small, for all  $0<|\tau|<\maxT$ we have the bound $r \in \left( |\tau|^{\frac 1{n-1}}/\delta, \delta\right)$ on $\Lambda$.

To prove \eqref{r_metric_equiv}, consider a fixed $(u, \bt) \in \Lambda$ and note that $r^{-2}(u)g= \frac{r^2}{r^2(u)}\left(dt^2+g_{\Ss^{n-1}}\right)$. Observe that as $\frac{r( u)}{r(t)} = e^{w( u)-w(t)}$ and $|w'| \in(1-3\delta^2, 1]$ by the choice of $\bunder, \maxT$, 
\[
\frac{99}{100}|t-u|< (1-3\delta^2)|t-u| \leq |w( t)-w(u)| = \left| \int_u^{ t} w'( \tilde t) \, d \tilde t\right|.
\]
Therefore, if $|u-t|> \frac 15$ then the length of a curve connecting $(u,\bt)$, $(t,\bt)$ in the metric $\frac{r^2}{r^2(u)}dt^2$ is at least $\frac 1{10}$ as
\[
\int_u^t e^{|w(u)-w(s)|}ds \geq \int_u^t e^{99|u-s|/100}ds = \frac{100}{99}({e^{99|u-t|/100}-1}).
\]
It follows that
a ball of radius $1/10$ about $(u,\bt)$ in the metric $r^{-2}(u) g$, is contained in the cylinder $[u-1/5, u+1/5] \times \Ss^{n-1}$. 
Now for $m=0$ or $m= -2n$, 
\[
\frac{r^{-m}(t)r^2(u)}   {r^{2-m}(t)} =  \frac{ \tau^{\frac{2-m}{n}}e^{2 w( u) - m w(t)}} { \tau^{\frac{2-m}n} e^{(2-m)w(t)}}= e^{2w(u)-2w(t)}.
\]The $C^0$ equivalence follows as $(1-3\delta^2)|t-u| \leq |w(t) - w(u)| \leq |t-u|$ on $\Lambda$ and $|t-u| \leq \frac 15$ for every comparison in the weighted metric. To get the $C^{0,\beta}$ equivalence, first observe that 
\[
\frac{r^2(u)}{r^2(t)}(r^2(t))' = \frac{2w'(t)r^2(u)}{r(t)} = 2w'(t) r(u) e^{w(u)-w(t)}.
\]For $|u-t| < \frac 15$, the above is bounded by $4\delta$ on $\Lambda$ and the first equivalence holds.
In the other case,
\[
\frac{\frac d{dt}r^{-2n}(t)r^2(u)}{\frac d{dt}r^{2-2n}(t)}= \frac{2nr^{-2n-1}(t)r^2(u)}{(2n-2)r^{1-2n}(t)}= \frac {2n}{2n-2}\cdot\frac{r^2(u)}{r^2(t)}
\]so the same comparisons as in the $C^0$ case give bounds on the ratio here, which implies the second equivalence in \eqref{r_metric_equiv}. 

Recall that by \eqref{req}, $r'(t) = w'(t)r(t)$ where  
$w = \log \left(|\tau|^{-1/n}r\right)$. Substituting into \eqref{weq},
\[
\frac{dr}{d t} = r\sqrt{1 - \left(r+ \tau r^{1-n}\right)^2}.
\]Therefore, 
\[
\frac{ds}{dr} = \left(1 - \left(r+ \tau r^{1-n}\right)^2\right)^{-1/2}.
\]
As the maximum of the function $|r+\tau r^{1-n}|$, restricted to $\Lambda$, occurs on $\partial \Lambda$, by choosing $\maxT>0$ perhaps smaller, when $0<|\tau|<\maxT$ we can bound
\begin{equation}\label{delta_Lambda}
|\delta + \delta^{1-n} \tau\:| \leq 2\delta, \; \; \; |\delta^{-1}|\tau|^{\frac 1{n-1}}+ |\tau|^{-1} \delta^{n-1}\tau\:| \leq 2\delta.
\end{equation}
The derivative estimates then in \ref{soverr} follow from \eqref{delta_Lambda} and the observation that 
\begin{align*}
\left|\frac{d^2s}{dr^2}\right| = &\left|\frac{(r+ r^{1-n}\tau)(1+(1-n)r^{-n}\tau)}{(1-(r+ r^{1-n}\tau)^2)^{3/2}}\right|\notag\\
=&\left|\frac 1r\cdot\frac{(r+ r^{1-n}\tau)(r+(1-n)r^{1-n}\tau)}{(1-(r+ r^{1-n}\tau)^2)^{3/2}}\right| .
\end{align*}
By the fact that $\frac{ds}{dr}>0$ and the estimate on $ds/dr$ we conclude the proof of \eqref{soverr} by 
\begin{align*}
 (1-4 \delta) (r -  \rin )\leq s(r)-\rin&=\int_{\rin}^{r}\frac{ds}{d\tilde r}d\tilde r \leq   (1+4 \delta) (r -  \rin ). 
\end{align*}

The $C^0$ equivalence of the norms in  \eqref{uniformnorms} follows immediately from \eqref{soverr}, 
and indeed the ratio of the weight functions will always contribute error ratios no worse than $(1+10\delta n) \leq 2$. 
To prove equivalence up to higher derivatives, 
observe that for a fixed $(u,\bt) \in \Lambda$, 
\[
 s^{-2}(u)g_A = \frac{1}{s^2(u)}\left( ds^2+ s^2g_{\Ssn}\right)= \frac{r^2}{s^2(u)}dt^2 + \frac{s^2}{s^2(u)}g_{\Ssn}.
\]
Fix a point $(u, \bt) \in \Lambda$ and consider the ball of radius $1/10$ about this point with respect to the metric $s^{-2}(u)g_A$. In the $t$-direction, the inequality
\[
\frac{r(t)}{r(u)}(1-5\delta) \leq \frac{r(t)}{s(u)} = \frac{r(t)}{r(u)}\frac{r(u)}{s(u)}\leq \frac{r(t)}{r(u)}(1+5\delta)
\]implies that it is enough to consider $|t-u|\leq \frac 25$. Moreover, as
\begin{equation}\label{eq:thetaderiv}
\frac {99}{100} \leq \frac{s(u)}{s(t)}\cdot\frac{r(t)}{r(u)} \leq  \frac{100}{99},
\end{equation}$|t-u|<\frac 25$ is sufficient for the $\Ssn$ direction as well. 

All derivatives purely in the $\Ssn$ direction are comparable in the norms as indicated by \eqref{eq:thetaderiv}. So we consider only partial derivatives involving $t$. The $C^1$ comparison is straightforward since (presuming $\partial_t v \neq 0$) 
\[
1-5\delta\leq \left(\frac{r(u)}{r(t)}\partial_t v \right)\cdot \left( \frac{s(u)}{r(t)}\partial_t v \right)^{-1} \leq 1+5 \delta.
\]Second derivatives in the $t$ direction then follow since
\[
\left(\frac{r^2(u)r'(t)}{r^3(t)}\right) \cdot \left( \frac{r^2(u)r'(t)}{r^3(t)}\right)^{-1}, \quad \quad \left( \frac{r^2(u)}{r^2(t)}\right)\cdot  \left(\frac{s^2(u)}{r^2(t)} \right) ^{-1}
\]satisfy equally good inequalities. The mixed partials and third derivatives again satisfy equal ratio estimates and the result follows.
\end{proof}

We define operators 
\begin{equation}\label{D:Lgl}
\mathcal L_{g_A}:= \Delta_{g_A} = \partial_{ss} +\frac{n-1}s \partial_s + \frac 1{s^2}\Delta_{\Ssn}, 
\qquad 
\mathcal L_g^\lambda:= \mathcal L_g + \lambda.
\end{equation} 
We first demonstrate that in an appropriately weighted metric, for sufficiently small $\lambda$, the operator $\mathcal L_g^\lambda$ is close to $\mathcal L_{g_A}$: 

\begin{lemma} 
\label{annularlemma} 
Let $\gamma\in (1,2), \beta \in (0,1)$. For $\epsilon_1>0$ there exists $\bunder$ large enough depending on $\epsilon_1$ and $\maxT>0$ small enough depending on $\bunder$ such that for all $0<|\tau|<\maxT$ the following holds: 
Consider $\Lambda$ defined by $\tau$ and $b$ satisfying \ref{ass:b}. Let $0\leq |\lambda| < (2\rout)^{-1}$. Then for all $V \in C^{2,\beta}(\Lambda)$ 
\begin{equation}\notag
\begin{aligned} 
\|\mathcal L_g^\lambda V - \mathcal L_{g_A} V:C^{0,\beta}(\Lambda, r,g, r^{-n-\gamma})\| \leq& 
\epsilon_1\|V: C^{2,\beta}(\Lambda, r,g, r^{2-n-\gamma})\|,
\\
\|\mathcal L_{g}^\lambda V - \mathcal L_{g_A}V:C^{0,\beta}(\Lambda,r,g,r^{\gamma-2})\| \leq & 
\epsilon_1 \|V:C^{2,\beta}(\Lambda,r,g,r^{\gamma})\|.
\end{aligned} 
\end{equation}
\end{lemma}

\begin{proof}Choose $\delta>0$ small enough so that $C(n) \delta <\epsilon_1/2$ where $C(n)$ is a fixed constant depending only on $n$.  Decrease $\delta$ if necessary so that it also satisfies the hypotheses in \ref{lemma:rvss}. Then choose $\bunder, \maxT$ as in \ref{lemma:rvss} for this $\delta$.

Applying \eqref{r_metric_equiv} and recalling \eqref{modA}, 
\begin{align*}
\|\, |A_g|^2+ \lambda:C^{0,\beta}(\Lambda, r,g, r^{-2})\|& =\|n(1+(n-1)\tau^2 r^{-2n})+\lambda:C^{0,\beta}(\Lambda, r,g, r^{-2})\| \\
& \leq 100\|nr^2+n(n-1)\tau^2 r^{2-2n}+\lambda r^2:C^{0,\beta}(\Lambda, r,g)\| \leq C(n)\delta.
\end{align*}

By calculation,
\[
\mathcal L_{g_A} - \mathcal L_g^\lambda =\frac 1{r} \left(\frac{n-1}{s}-\frac{n-1}{r}w'\right) \partial_t+ \left(\frac 1{s^2}-\frac 1{r^2}\right)\Delta_{\Ssn}-n(1+(n-1)\tau^2r^{-2n})- \lambda.
\]
Given the constraints on $\delta$, the estimates of \ref{lemma:rvss} and multiplicative properties of H\"older norms imply the result.
\end{proof}

\begin{definition}\label{defn:fhat}We define $\hf_0:   \Real\times \Ssn       \to \Real$ (recall \ref{FF}) such that
\begin{equation}\notag
\begin{aligned}
\hf_0:= & \nu_\tau \cdot \Be_1 = \frac{\ovr'}{\ovr} = \pm \sqrt{1-\left(\ovr+ \tau \ovr^{1-n}\right)^2}.
\end{aligned} 
\end{equation} 
Here the sign for $\hf_0$ depends on the domain of definition but note that $\hf_0$ is odd about $t=0$.

\end{definition}
\begin{lemma}
The lowest eigenvalue for the Dirichlet problem for $\mathcal L_g$ on $\Lambda$ is bounded below by $(2\rout)^{-1}$.
\end{lemma}

\begin{proof}
First notice that $\mathcal L_g \hf_0 =0$ and $\hf_0(0)=\hf_0(\Pdo)=0$. Moreover, by definition, $\hf_0 <0$ on $(0,\Pdo)\times \Ssn$. Classical theory implies that on $(0, \Pdo) \times \Ssn$, the lowest eigenvalue for the Dirichlet problem for $\mathcal L_g$ is $0$. 
Domain monotonicity then implies that on $\Lambda \subset (0, \Pdo) \times \Ssn$, the lowest eigenvalue for the Dirichlet problem for $\mathcal L_g$ is positive. Suppose $\lambda_1$ is the lowest eigenvalue for the Dirichlet problem on $\Lambda$ and that $0< \lambda_1<(2\rout)^{-1}$. For any $0<\lambda  < (2\rout)^{-1}$, \ref{annularlemma} applies to the operators $\mathcal L^\lambda_g, \mathcal L_{g_A}$. Let $\widetilde V$ satisfy $\mathcal L_{g_A}\widetilde V =0$ on $\Lambda$, $\widetilde V|_{\Cout}=1, \widetilde V_{\Cin}=0$. By inspection, one determines the estimate
\[
\|\widetilde V:C^{2,\beta}(\Lambda,r,g)\| \leq C(\beta).
\] Using \ref{annularlemma} with the weaker decay estimate $r^{-2}$, we may iterate to produce $V$ such that
$\mathcal L^{\lambda_1} V =0$ with the same boundary data and $\|V:C^{2,\beta}(\Lambda,r,g)\| \leq C(\beta)$. Let $f$ be the lowest eigenfunction for $\mathcal L_g$. Then $\mathcal L_g f= -\lambda_1 f$ and $f>0$ on $\Lambda$ with $f=0$ on $\partial \Lambda$. Since $f \not\equiv 0$, there exists $C$ sufficiently large such that $Cf>V$ on a domain $\Omega \subset \Lambda$. Then $\mathcal L_g(V-Cf) = -\lambda_1(V-Cf)$ and $Cf-V>0$ on $\Omega \subset \Lambda$. 
Domain monotonicity then implies $\lambda_1$ is not the lowest eigenvalue, giving a contradiction.
\end{proof}

\begin{corollary}\begin{enumerate}
\item The Dirichlet problem  on $\Lambda$ for $\mathcal L_g^\lambda$ with $0\leq |\lambda| <(4\rout)^{-1}$ and given $C^{2,\beta}$ Dirichlet data has a unique solution.
\item For $E \in C^{0, \beta}(\Lambda)$ there exists a unique $\varphi \in C^{2,\beta}(\Lambda)$ such that $\mathcal L_g^\lambda \varphi =E$ and $\varphi|_{\partial \Lambda}=0$. Moreover
\[
\|\varphi:C^{2,\beta}(\Lambda,g)\| \leq  C(\beta,\gamma)\rout \|E:C^{0,\beta}(\Lambda,g)\|.\]
\end{enumerate}
\end{corollary}
\begin{proof}
The first item follows immediately from the lemma and by noting that if $|\lambda|<(4\rout)^{-1}$ then the lowest eigenvalue for $\mathcal L_g^\lambda$ is greater than $(4\rout)^{-1}$. The second follows from the Rayleigh quotient and standard techniques.
\end{proof}

We now use \ref{annularlemma} and \ref{flatannuluslinear} to prove the decay estimates we desire. Note that \ref{flatannuluslinear} gives the analogous decay estimates for solutions to $\mathcal L_{g_A} V =E$ on flat annuli.

\begin{definition}\label{phidef}
For $i =1, \dots, n$, let $\phi_i$ denote the $i$-th component of the canonical immersion of $\Ssn$ into $\Real^n$. For convenience going forward, let $\phi_0\equiv 1$. 
\end{definition}

Note that $\Delta_{\Ssn} \phi_i = - (n-1) \phi_i$ and $\Delta_{\Ssn} \phi_0=0$  
and that the functions are $L^2$ orthogonal but we have chosen not to normalize them.
Since we will be particularly interested in understanding the low harmonics of a function on the boundary of $\Lambda$, we introduce the following notation.

\begin{definition}\label{LowHdef}
Let $\mathcal H_k[C]$ denote the finite dimensional space of spherical harmonics on the meridian sphere at $C$ 
that includes all of those up to (and including) the $k$-th eigenspace. 
That is, $\mathcal H_0[C]$ is the span of $\{\phi_0\}$ and
$\mathcal H_1[C]$ is the span of $\{\phi_0, \dots, \phi_n\}$.
\end{definition}

\begin{prop}\label{RLambda}Given $\beta \in(0,1)$ and $\gamma \in (1,2)$, there exists $\bunder$ large enough depending on $\beta,\gamma$ and $\maxT>0$ small depending on $\bunder$ such that the following holds.\\ 
For $0<|\tau|<\maxT$ and $b$ satisfying \ref{ass:b} and any $|\lambda| <(4\rout)^{-1}$, there are linear maps $\mathcal{R}^{\mathrm{out}}_{\Lambda,\lambda}, \mathcal R^{\mathrm{in}}_{\Lambda,\lambda}:C^{0,\beta}(\Lambda) \to C^{2, \beta}(\Lambda)$ such that, given $E \in C^{0,\beta}(\Lambda)$:
\begin{enumerate}[(i)]
\item  if $V^{\mathrm{out}} :=  \mathcal{R}^{\mathrm{out}}_{\Lambda,\lambda}( E)$ then
\begin{itemize} 
\item $ \mathcal{L}_g^\lambda V^{\mathrm{out}} =E \text{ on }\Lambda.$
\item $V^{\mathrm{out}}|_{\Cout} \in \mathcal H_1[\Cout]$ and vanishes on $\Cin$.
\item $\|V^{\mathrm{out}}:C^{2,\beta}(\Lambda,r,g, r^{\gamma}) \|\leq C(\beta, \gamma)\|E:C^{0,\beta}(\Lambda,r,g,r^{\gamma-2}) \|$.
\end{itemize}
\item  if $V^{\mathrm{in}} :=  \mathcal{R}^{\mathrm{in}}_{\Lambda,\lambda}( E)$ then
\begin{itemize} 
\item $ \mathcal{L}_g^\lambda V^{\mathrm{in}} =E \text{ on }\Lambda.$
\item $V^{\mathrm{in}}|_{\Cin} \in \mathcal H_1[\Cin]$ and vanishes on $\Cout$.
\item $\|V^{\mathrm{in}}:C^{2,\beta}(\Lambda,r,g, r^{2-n-\gamma}) \|\leq C(\beta, \gamma)\|E:C^{0,\beta}(\Lambda,r,g, r^{-n-\gamma}) \|$.
\end{itemize}
\end{enumerate}
In either case, $\mathcal{R}^{\mathrm{out}}_{\Lambda,\lambda}, \mathcal R^{ \mathrm{in}}_{\Lambda,\lambda}$ both depend continuously on the choice of $\tau, b$.  
\end{prop}
\begin{proof}We prove the result for $\mathcal R^{\mathrm{out}}_{\Lambda,\lambda}$ as the other argument follows similarly.
Let $E \in C^{0, \beta}(\Lambda)$ where $\bunder, \maxT$ of \ref{annularlemma} are determined by choosing $\epsilon_1<1/(20C(\beta,\gamma))$. We now apply \ref{flatannuluslinear} with $s$ defined as a function of $t$ as in \eqref{D:s} and the domain of definition equal to $\Lambda$. Thus, there exists $V_0= \mathcal R^{\mathrm{out}}_A(E)$ such that
\begin{enumerate}
\item $\mathcal L_{g_A} V_0 = E$, 
\item $V_0|_{\Cout} \in \mathcal H_1[\Cout]$ and vanishes on $\Cin$, 
\item $\|V_0:C^{2,\beta}(\Lambda,s,g_A, s^{\gamma}) \|\leq C(\beta, \gamma)\|E:C^{0,\beta}(\Lambda,s,g_A,s^{\gamma-2}) \|$. 
\end{enumerate} \ref{uniformnorms} and \ref{annularlemma} together imply that
\[
\|\mathcal L_{g}^\lambda V_0 - E:C^{0,\beta}(\Lambda,r,g, r^{\gamma-2}) \|\leq 10\epsilon_1C(\beta, \gamma)\|E:C^{0,\beta}(\Lambda,r,g,r^{\gamma-2}) \|.
\]
We complete the proof by iteration.
\end{proof}In a similar fashion, we can prove the following corollary.
\begin{corollary}\label{linearcor}
Assuming $\epsilon_1$ of \ref{annularlemma} is small enough in terms of $\epsilon_2$ and $\beta\in(0,1), \gamma \in(1,2)$, for any $0\leq |\lambda| < (4\rout)^{-1}$, there are two linear maps:
\begin{equation}\notag 
\begin{aligned}
\mathcal{R}^{\mathrm{out}}_{\partial,\lambda} :\{u \in C^{2, \beta}(\Cout): u \text{ is } L^2(\Cout, g_{\Ssn}) \text{-orthogonal to } \mathcal H_1[\Cout]
 \} \to C^{2, \beta}(\Lambda),  
\\ 
\mathcal{R}^{\mathrm{in}}_{\partial,\lambda} :\{u \in C^{2, \beta}(\Cin): u \text{ is } L^2(\Cin,  g_{\Ssn}) \text{-orthogonal to } \mathcal H_1[\Cin]
 \} \to C^{2, \beta}(\Lambda) .  
\end{aligned} 
\end{equation} 
such that the following hold: 
\begin{enumerate}
\item If $u$ is in the domain of $\mathcal R^{\mathrm{out}}_{\partial,\lambda}$  and $V^{\mathrm{out}}:= \mathcal{R}^{\mathrm{out}}_{\partial,\lambda}( u)$ then
\begin{itemize}
\item $\mathcal{L}_g^\lambda V^{\mathrm{out}}=0 \text{ on } \Lambda$.
\item $V^{\mathrm{out}}|_{\Cout}-u \in \mathcal H_1[\Cout]$  and $V^{\mathrm{out}}$ vanishes on $\Cin$.
\item $\|V^{\mathrm{out}}|_{\Cout} - u:C^{2, \beta}(\Cout,  g_{\Ssn})\| \leq \epsilon_2\|u:C^{2,\beta}(\Cout,  g_{\Ssn})\|$.
\item $\|V^{\mathrm{out}}:C^{2,\beta}(\Lambda, r,g,(r/\rout)^{\gamma})\| \leq C(\beta, \gamma)\|u:C^{2, \beta}(\Cout,  g_{\Ssn})\|$.
\end{itemize}
\item If $u$ is in the domain of $\mathcal R^{\mathrm{in}}_{\partial,\lambda}$ and $V^{\mathrm{in}}:= \mathcal{R}^{\mathrm{in}}_{\partial,\lambda}( u)$ then
\begin{itemize}
\item $\mathcal{L}_g^\lambda V^{\mathrm{in}}=0 \text{ on } \Lambda$.
 \item $V^{\mathrm{in}}|_{\Cin}-u \in \mathcal H_1[\Cin]$  and $V^{\mathrm{in}}$ vanishes on $\Cout$.
\item $\|V^{\mathrm{in}}|_{\Cout} - u:C^{2, \beta}(\Cin, g_{\Ssn})\| \leq \epsilon_2\|u:C^{2,\beta}(\Cin, g_{\Ssn})\|$.
\item $\|V^{\mathrm{in}}:C^{2,\beta}(\Lambda, r,g,(\rin/r)^{n-2+\gamma})\| \leq C(\beta, \gamma)\|u:C^{2, \beta}(\Cin,  g_{\Ssn})\|$.
\end{itemize}

\end{enumerate} In either case  $\mathcal{R}^{\mathrm{out}}_{\partial,\lambda}, \mathcal R^{ \mathrm{in}}_{\partial,\lambda}$ depend continuously on $\tau,b$.
\end{corollary}
\begin{proof}Again, we prove the result only for $\mathcal R^{\mathrm{out}}_{{\partial,\lambda}}$. 
We first note that as an immediate corollary to \ref{flatannuluslinear}, we may define a linear map 
\[\mathcal R^{\mathrm{out}}_{\partial,A}:\{u \in C^{2, \beta}(\Cout): u \text{ is } L^2(\Cout, g_{\Ssn}) \text{-orthogonal to } \mathcal H_1[\Cout]
 \} \to C^{2, \beta}(\Lambda)
\]such that if $u$ is in the domain of $\mathcal R^{\mathrm{out}}_{\partial,A}$ and $\tilde V^{\mathrm{out}}:= \mathcal R^{\mathrm{out}}_{\partial, A}(u)$ then
\begin{enumerate}
\item $\mathcal L_{g_A} \tilde V^{\mathrm{out}} = 0$ on $\Lambda$,
\item $\tilde V^{\mathrm{out}}|_{\Cout}=u$ and $\tilde V^{\mathrm{out}}|_{\Cin}=0$,
\item $\|\tilde V^{\mathrm{out}}:C^{2,\beta}(\Lambda, s,g_A,s^{\gamma})\| \leq C(\beta, \gamma)\sout^{-\gamma}\|u:C^{2, \beta}(\Cout,  g_{\Ssn})\|$.
\end{enumerate}Set
\[
\mathcal R^{\mathrm{out}}_{\partial,\lambda} = \mathcal R^{\mathrm{out}}_{\partial,A} - \mathcal R^{\mathrm{out}}_{\Lambda,\lambda} \mathcal L_g^\lambda\mathcal R^{\mathrm{out}}_{\partial,A}.
\]The previous estimates immediately imply the result.
\end{proof}

We now introduce Dirichlet solutions to $\mathcal L_g^\lambda$ for $\lambda$ in the specified region. These solutions will allow us to understand the behavior of the low harmonics of any function defined on $\Lambda$.

\begin{definition}\label{annulardecaydef}For any $0\leq |\lambda| <(4\rout)^{-1}$ and $i=0, \dots, n$, let $V_i^\lambda[\Lambda, a_1, a_2]$, $\widetilde V_i[\Lambda, a_1, a_2]$ denote solutions to the Dirichlet problem given by 
\[
\mathcal L_g^\lambda V_i^\lambda[\Lambda, a_1, a_2] = 0, \qquad \mathcal L_{g_A} \widetilde V_i[\Lambda, a_1, a_2]=0
\]with boundary conditions
\begin{align*}
V_0^\lambda[\Lambda, a_1, a_2]  &= \widetilde V_0[\Lambda, a_1, a_2]  =a_1 \text{ on } C^{\mathrm{out}}\\
V_0^\lambda[\Lambda, a_1, a_2]  &= \widetilde V_0[\Lambda, a_1, a_2]  =a_2 \text{ on } C^{\mathrm{in}}\\
V_i^\lambda[\Lambda, a_1, a_2] &= \widetilde V_i[\Lambda, a_1, a_2]  = a_1 \phi_i  \text{ on } C^{\mathrm{out}}, \text{ for } i =1, \dots, n\\
V_i^\lambda[\Lambda, a_1, a_2] &= \widetilde V_i[\Lambda, a_1, a_2]  = a_2 \phi_i  \text{ on } C^{\mathrm{in}}, \text{ for } i =1, \dots, n.
\end{align*} 
\end{definition}
Recall that $s(\rin) = \rin:= \ssin$ and set $\sout := s(\rout)$. 
By \eqref{soverr}, $|\sout/\rout -1|\leq 4\delta$. 
We observe that, recall \ref{phidef}, 
\[
\widetilde V_0[\Lambda, 1, 0] = \frac{s^{2-n} - s^{2-n}_{\mathrm{in}}}{s^{2-n}_{\mathrm{out}} - s^{2-n}_{\mathrm{in}}},\qquad \widetilde V_0[\Lambda, 0, 1] = \frac{s^{2-n} - s^{2-n}_{\mathrm{out}}}{s^{2-n}_{\mathrm{in}}-s^{2-n}_{\mathrm{out}}},
\]
\[
\widetilde V_i[\Lambda, 1, 0]= \frac{s-s^n_{\mathrm{in}}s^{1-n}}{s_{\mathrm{out}} -s^n_{\mathrm{in}}s^{1-n}_{\mathrm{out}}} \phi_i , \qquad \widetilde V_i[\Lambda, 0,1]=\frac{s-s^n_{\mathrm{out}}s^{1-n}}{s_{\mathrm{in}} -s^n_{\mathrm{out}}s^{1-n}_{\mathrm{in}}}\phi_i.
\]

\begin{lemma}\label{annulardecaylemma}
For each $0 \leq |\lambda| <(4\rout)^{-1}$, 
$V_0^\lambda$ is constant on each meridian sphere and each $V_i^\lambda$ is a multiple of $\phi_i$ on each meridian sphere. 
Moreover, there exists a choice as in \ref{annularlemma} of $\epsilon_1>0$ small enough so the following hold: 
\begin{enumerate}
\item $\|V_0^\lambda[\Lambda,1,0]-\widetilde V_0[\Lambda,1,0]:C^{2,\beta}(\Lambda, r,g)\|\leq C(\beta)\epsilon_1$.
\item $\|V_0^\lambda[\Lambda,0,1]-\widetilde V_0[\Lambda,0,1]:C^{2,\beta}(\Lambda, r,g, (\rin/r)^{n-2})\|\leq C(\beta)\epsilon_1$.
\item $\|V_i^\lambda[\Lambda,1,0]-\widetilde V_i[\Lambda,1,0]:C^{2,\beta}(\Lambda, r,g,r)\|\leq C(\beta)\epsilon_1$.
\item $\|V_i^\lambda[\Lambda,0,1]-\widetilde V_i[\Lambda,0,1]:C^{2,\beta}(\Lambda, r,g,(\rin/r)^{n-1})\|\leq C(\beta)\epsilon_1$.
\end{enumerate}
\end{lemma}
\begin{proof}
By inspection $\widetilde V_0[\Lambda, 1,0], \widetilde V_0[\Lambda, 0,1]$ satisfy the estimates
\[
\|\widetilde V_0[\Lambda,1,0]:C^{2,\beta}(\Lambda,r,g)\| \leq C(\beta),
\]
\[
\|\widetilde V_0[\Lambda,0,1]:C^{2,\beta}(\Lambda,r,g,r^{2-n})\| \leq C(\beta) \rin^{n-2}.
\]
By \ref{annularlemma},
\[
\|\mathcal L^\lambda_g \widetilde V_0[\Lambda,1,0]:C^{0,\beta}(\Lambda,r,g,  r^{-2})\| \leq C(\beta)\epsilon_1,
\]
\[
\|\mathcal L^\lambda_g \widetilde V_0[\Lambda,0,1]:C^{0,\beta}(\Lambda,r,g,r^{-n})\| \leq C(\beta )\rin^{n-2}\epsilon_1.
\]
Using \ref{RLambda} applied to the operator $\mathcal L_g^\lambda$ (with $\gamma=0$), let $\widehat V_{\mathrm{out}}=\mathcal R_\Lambda^{\mathrm{out}}(\mathcal L^\lambda_g \widetilde V_0[\Lambda,1,0])$ and $\widehat V_{\mathrm{in}}=\mathcal R_\Lambda^{\mathrm{in}}(\mathcal L^\lambda_g \widetilde V_0[\Lambda,0,1])$.
Then
\[
\| \widehat V_{\mathrm{out}}:C^{2,\beta}(\Lambda,r,g)\| \leq C(\beta)\epsilon_1,
\qquad
\| \widehat V_{\mathrm{in}}:C^{2,\beta}(\Lambda,r,g,r^{2-n})\| \leq C(\beta)\rin^{n-2}\epsilon_1.
\]Note that the boundary data is in $\mathcal H_0[\Cout], \mathcal H_0[\Cin]$. 
Set
\[
V:=\widetilde V_0[\Lambda,A_{\mathrm{out}},A_{\mathrm{in}}] - A_{\mathrm{out}}\widehat V_{\mathrm{out}}- A_{\mathrm{in}}\widehat V_{\mathrm{in}}.
\]where $A_{\mathrm{out}}, A_{\mathrm{in}}$ are chosen such that 
\[
V |_{\Cout} = 1, \quad V |_{\Cin} =0.
\]Then $\mathcal L^\lambda_g V=0$ by construction and since the Dirichlet problem has a unique solution
\[
V^\lambda_0[\Lambda,1,0] = V.
\] By definition, 
\[
\left\{\begin{array}{ll} 1&= A_\mathout - A_\mathout \widehat V_\mathout(\rout) - A_\mathin \widehat V_\mathin (\rout)\\
0&= A_\mathin - A_\mathout \widehat V_\mathout(\rin) - A_\mathin \widehat V_\mathin(\rin).
\end{array}\right. 
\]Inspection of the estimates implies that $|1-A_\mathout| \leq C(\beta)\epsilon_1$ and $|A_\mathin| \leq C(\beta) \epsilon_1$. Item (1) then follows from the triangle inequality and all previous estimates.

For item (2), choose $A_\mathout, A_\mathin$ such that
\[
V:=\widetilde V_0[\Lambda,A_{\mathrm{out}},A_{\mathrm{in}}] - A_{\mathrm{out}}\widehat V_{\mathrm{out}}- A_{\mathrm{in}}\widehat V_{\mathrm{in}}.
\]where $A_{\mathrm{out}}, A_{\mathrm{in}}$ are chosen such that 
\[
V |_{\Cout} = 0, \quad V |_{\Cin} =1.
\]As before, the choice of boundary data and uniqueness of Dirichlet solutions implies that
\[
V^\lambda_0[\Lambda,0,1]= V.
\]
Note that in this case
\[
\left\{\begin{array}{ll} 0&= A_\mathout - A_\mathout \widehat V_\mathout(\rout) - A_\mathin \widehat V_\mathin (\rout)\\
1&= A_\mathin - A_\mathout \widehat V_\mathout(\rin) - A_\mathin \widehat V_\mathin(\rin).
\end{array}\right. 
\]
Again, the estimates imply that $|A_\mathout| \leq C(\beta) \epsilon_1(\frac{\rin}{\rout})^{n-2}$ and $|1-A_\mathin| \leq C(\beta) \epsilon_1$.

For the estimates on $V_i^\lambda[\Lambda,1,0]$, $V_i^\lambda[\Lambda,0,1]$ we note that
\[
\|\mathcal L^\lambda_g \widetilde V_i[\Lambda,1,0]:C^{0,\beta}(\Lambda,r,g, r^{-1})\| \leq C(\beta)\epsilon_1,
\]
\[
\|\mathcal L^\lambda_g \widetilde V_i[\Lambda,0,1]:C^{0,\beta}(\Lambda,r,g,(\rin/r)^{n-1}r^{-2})\| \leq C(\beta)\epsilon_1.
\]For $\widehat V_i[1,0] := \mathcal R_\Lambda^{\mathrm{out}}(\mathcal L^\lambda_g \widetilde V_i[\Lambda,1,0])$ and $\widehat V_i[0,1]:= \mathcal R_\Lambda^{\mathrm{in}}(\mathcal L^\lambda_g \widetilde V_i[\Lambda,0,1])$ we have the estimates
\[
\| \widehat V_i[1,0]:C^{2,\beta}(\Lambda,r,g,r)\| \leq C(\beta)\epsilon_1,
\]
\[\| \widehat V_i[0,1]:C^{2,\beta}(\Lambda,r,g,(\rin/r)^{n-1})\| \leq C(\beta)\epsilon_1.
\]Note that the boundary data is in $\mathcal H_1[\Cout], \mathcal H_1[\Cin]$. 
Using these estimates with the same techniques previously outlined implies the result.
\end{proof}

\subsection*{Solving the linearized equation semi-locally on $\Spext, \Smext$}
The goal of this subsection is to prove \ref{linearpluslemma} and \ref{linearminuslemma} which provide 
semi-local estimates on $\Spext$ and $\Smext$. 
In contrast to \cite{BKLD} 
we do not attempt to solve a Dirichlet problem with zero boundary data. 
Instead, we solve an ODE where solutions to the ODE are allowed to grow at a particular rate back toward the nearest central sphere. 

Throughout the subsection we will decompose functions by their projections onto various spaces of the kernel of the operator $\Delta_{\Ssn}$. 
For this reason, we introduce the following notation.
\begin{definition}
Let $L_k$ denote the projection of the operator $\mathcal L_g$ onto the $k$-th space of eigenfunctions for the operator $\Delta_{\Ssn}$. That is,
\[
L_k:= \frac 1{r^2} \partial_{tt} + \frac{n-2}{r^2}\frac {r'}r \partial_t +\left[n(1+(n-1)\tau^2 r^{-2n})-\frac{k}{r^2}(n-2+k)\right].
\]
\end{definition}
We will use the projected operators to decompose the local linear problems and determine separate estimates for the high and the low eigenvalues.  
For ease of notation, we introduce the following decomposition which we will use throughout this subsection.
\begin{definition}\label{f_decomp}For $j=0, \dots, n$, let $\phi_j$ be defined as in \ref{phidef}. For $j \geq n+1$, choose $\phi_j$ such that $\{\phi_{n+1}, \phi_{n+2}, \dots\}$ is an $L^2$ orthonormal basis for the remaining eigenspaces of $\Delta_{\Ssn}$. (Recall that $\{\phi_0, \dots, \phi_n\}$ is an $L^2$-orthogonal basis for the lowest two eigenspaces of $\Delta_{\Ssn}$.)

Let $f \in C^{k,\beta}$ on $\Spext$ or $\Smext$. 
We define the decompositions  
\[
f(t,\bt)= \sum_{i=0}^\infty  f_i(t) \phi_i = f_0 + f_1 + f_{\mathrm{high}} 
\ 
\text{ where } 
\ 
f_1 := \sum_{i=1}^n f_i(t) \phi_i, 
\text{ } 
f_{\mathrm{high}} := \sum_{i=n+1}^\infty  f_i(t) \phi_i. 
\]
\end{definition}

We first consider the linear problem for functions with no low harmonics.
\begin{lemma}\label{highharmonicslemma}
Let $\beta \in (0,1)$, $\gamma \in (1,2)$. For $\bunder$ chosen as in \ref{annularlemma} and $\maxT>0$ satisfying the requirements of \ref{annularlemma} and the inequality $\maxT^{\frac 1{n-1}} \leq 1/(2n^2)$, let $b$ satisfy \ref{ass:b}. Then 
there exist linear maps $\mathcal R_{\mathrm{high}}^+, \mathcal R_{\mathrm{high}}^-$ where
\[
\mathcal R_{\mathrm{high}}^\pm:\{ E^\pm \in C^{0,\beta}(\widetilde S^\pm): E^\pm = E_{\mathrm{high}}^\pm, \supp(E^\pm) \subset S^\pm_1\} \to C^{2,\beta}(\widetilde S^\pm)
\]such that for $E^\pm$ in the domain of $\mathcal R_{\mathrm{high}}^\pm$, and $f^\pm:=\mathcal R^\pm_{\mathrm{high}}(E^\pm)$,
\begin{enumerate}
\item \label{onef}$\mathcal L_g f^\pm = E^\pm$.
\item \label{twof}$f^\pm = f_{\mathrm{high}}^\pm$.
\item \label{threef}$f^\pm = 0$ on $\partial \widetilde S^\pm$.
\item $\|f^+:C^{2,\beta}(\Spext,r,g,r^\gamma)\|\leq C(\bunder,\beta,\gamma)\|E^+:C^{0, \beta}(\Sp_1,r,g)\|$.
\item $\|f^-:C^{2,\beta}(\Smext,r,g,(\rin/r)^{n-2+\gamma})\|\leq C(\beta,\gamma)\|E^-:C^{0, \beta}(\Sm_1,r,g)\|$.
\end{enumerate}
Finally, $\mathcal R_{\mathrm{high}}^\pm$ depend continuously on $\tau$.
\end{lemma}
The proof will follow from the decay estimates determined on $\Lambda$ and the following lemma.

\begin{lemma}\label{highprojlemma} 
For a fixed $n \in \mathbb N$, $n>2$, consider $\bunder$ chosen as in \ref{annularlemma} and $\maxT>0$ satisfying the requirements of \ref{annularlemma} and the inequality $\maxT^{\frac 1{n-1}} \leq 1/(2n^2)$. Then the following holds:

For any $0<|\tau|<\maxT$ and $b$ satisfying \ref{ass:b}, let $\widetilde S^\pm$ be the domain defined by $\tau$ and $b$ as in \ref{domaindefinitions}. 
Consider the two sets of functions 
$X^\pm:=\{ f\in L^2(\widetilde S^\pm) \, : \, f=f_{\mathrm{high}}, f|_{\partial \widetilde S^\pm}=0,\int_{\widetilde S^\pm} f^2 =1\}$.  
Then
\[
 \inf_{f \in X^\pm}{-\int_{\widetilde S^\pm} f \mathcal L_g f}    
=    
 \inf_{f \in X^\pm}{ \int_{\widetilde S^\pm} |\nabla f|^2 - |A|^2 f } 
\geq 1.
\]
\end{lemma}
\begin{proof}

Let $f= \sum_{i=n+1}^\infty f_i \phi_i$. Since $i \geq n+1$,
\[\int_{\Ssn}|\nabla_{\Ssn}\phi_i|^2dg_{\Ssn}\geq 2n\int_{\Ssn}\phi_i^2dg_{\Ssn} = 2n.
\]Therefore, recalling \eqref{modA}, 
\begin{align}
\notag \int_{\widetilde S^\pm} |\nabla f|^2 - |A|^2 f^2 dx &=\int_{\widetilde S^\pm} \frac 1{r^2}|\sum_i (f_i'\phi_i+ f_i \nabla_{\Ssn}\phi_i)|^2 - |A|^2 f^2 \, dx\\
& \geq \int r^{n-2}\sum_i (f_i')^2 + nr^{n-2}\sum_i f_i^2(2-r^2-(n-1)\tau^2r^{2-2n}) \, dt.\label{k1}
\end{align}On $\widetilde S^+$, $r \in (|\tau|^{\frac 1{n-1}}/\delta,1+ O(|\tau|))$. Therefore, by the bound on $\delta>0$ imposed in \ref{annularlemma} , \eqref{k1} is bounded below by
\[
n\int r^{n-2} \sum_i f_i^2 \, dt\geq \frac n2 \int r^n \sum_i f_i^2 \, dt \geq 1.
\]The last inequality follows since $\|f\|_{L^2(\widetilde S^+)} = 1$.

It remains to show the estimate on $\widetilde S^-$. We now demonstrate that the positive terms on the right hand side are sufficiently large to more than overcome the negative contribution. First observe that
\[
-2 \int r^{n-2} w' f_i f_i' \, dt =- \int r^{n-2}w'(f_i^2)'\, dt =   \int (r^{n-2} w')' f_i^2 \, dt = \int r^{n-2}f_i^2(w''+ (n-2)(w')^2)\, dt.
\]
Now we use Cauchy-Schwarz and an absorbing inequality to note that
\begin{align*}
-2 \int r^{n-2} w' f_i f_i' \, dt & \leq 2\left( \int r^{n-2}(w')^2 f_i^2 \, dt \cdot \int r^{n-2} (f_i')^2 \, dt\right)^{1/2} \\
& \leq (n-2) \int r^{n-2}(w')^2 f_i^2 \, dt + \frac 1{n-2}\int r^{n-2} (f_i')^2 \, dt.
\end{align*}
Combining the above and simplifying,
\[
 (n-2) \int  r^{n-2}f_i^2w'' \, dt \leq \int r^{n-2} (f_i')^2 \, dt .
\]
Thus, recalling \eqref{w_derivs}
\begin{align*}
 \int r^{n-2} (f_i')^2 &+ nr^{n-2} f_i^2(2-r^2-(n-1)\tau^2r^{2-2n}) \, dt\\
 & \geq \int r^{n-2}f_i^2 \left((n-2)w'' + 2n-nr^2 -n(n-1)\tau^2 r^{2-2n} \right)\, dt\\ 
 &= \int r^{n-2}f_i^2\left( 2n-2(n-1)r^2 +(n-2)^2 \tau r^{2-n} -2(n-1) \tau^2r^{2-2n}\right)\, dt.
\end{align*}
We simplify the above expression by using \eqref{w_derivs} to note that
\[
(2n-2)(w')^2 = (2n-2) - 2(n-1)r^2 - 4(n-1)\tau r^{2-n} - 2(n-1)\tau^2 r^{2-2n}.
\]Thus,
\begin{align*}
 \int r^{n-2} (f_i')^2+ nr^{n-2} f_i^2(2-r^2-(n-1)\tau^2r^{2-2n}) \, dt &\geq \int  r^{n-2}f_i^2\left( (2n-2)(w')^2 + 2+ n^2 \tau r^{2-n} \right) \, dt \\
 & \geq \int  r^{n-2}f_i^2\left(  2+ n^2 \tau r^{2-n} \right) \, dt .
\end{align*}
Now, since $r \geq |\tau|^{\frac 1{n-1}}$, the hypothesis on $\maxT$ implies that
\[
2+ n^2 \tau r^{2-n} \geq 1.5 + \left(0.5 -n^2|\tau|^{1-\frac{n-2}{n-1}}\right) \geq 1.5 \geq r^2.
\]
Immediately we observe that
\[
-\int_{\Smext} f \mathcal L_g f \geq  \int r^n \sum_if_i^2 \, dt= 1.
\]

\end{proof}

We can now complete the proof of \ref{highharmonicslemma}.
\begin{proof}
Given $E=E_{\mathrm{high}}$, the existence of $f$ satisfying items \eqref{onef}, \eqref{twof}, \eqref{threef} follows from standard theory using the coercivity estimate provided. We determine the decay estimates in the following manner. First, the coercivity estimate implies that $\|f\|_{L^2(\widetilde S)} \leq C \|E\|_{L^2(\widetilde S)}$. The uniform geometry on $S_2$--in the natural scaling, which is the metric we use--allows us to boost these $L^2$ estimates to $C^{k,\alpha}$ using Schauder theory and De Giorgi-Nash-Moser techniques. Thus, for $S=S^\pm$,
\[
\|f:C^{2,\beta}(S_2,r,g)\|\leq C(\beta)\|E:C^{0,\beta}(S_1,r,g,r^{-2})\|\leq C(\bunder,\beta)\|E:C^{0,\beta}(S_1,r,g)\|
\]where the second inequality follows from \ref{radiuslemma}. Using the estimates of \ref{linearcor}, since $\mathcal L_g f = 0$ on $\widetilde S \backslash S_1$, we then determine for $S=S^+$,
\[\|f:C^{2,\beta}(\Lambda, r,g,r^{\gamma})\| \leq \rout^{-\gamma}C(\beta, \gamma)\|f:C^{2, \beta}(\Cout_1,  g_{\Ssn})\|,\]and for $S=S^-$,
\[\|f:C^{2,\beta}(\Lambda,r, g,(\rin/r)^{n-2+\gamma})\| \leq C(\beta, \gamma)\|f:C^{2, \beta}(\Cin_1,  g_{\Ssn})\|.\]Combining these estimates as appropriate, and noting that $\rout^{-\gamma} \leq C(\bunder, \gamma)$, gives the result.
\end{proof}
As the previous lemma provides solvability for high harmonics and good estimates on the solutions, we solve the semi-local linearized problem by appealing to \ref{annulardecaylemma} to understand the behavior of $f_0, f_1$.

To easily adapt this argument to the global problem at hand, we will use notation that, in the setting of a single Delaunay surface, makes little sense. We presume that $\Spext:= \Lc \cup \Sp \cup \Lf$ and do not explain the definitions of $\Lc, \Lf$ until they are needed later. Suffice it to say that on $\Lc$ we allow our solution to grow toward the boundary but on $\Lf$ we force the solution to decay to the boundary at a prescribed rate.
\begin{lemma}\label{linearpluslemma}Given $\beta \in (0,1), \gamma \in (1,2)$,
for each $\Sp$ there exists a linear map
\[
\mathcal R_{\Spext}:\{E\in C^{0,\beta}(\Spext): E \text{ is supported on } \Sp_1\} \to C^{2,\beta}(\Spext, g)
\]such that the following hold for $E$ in the domain of $\mathcal R_{\Spext}$ and $\varphi = \mathcal R_{\Spext}(E)$:
\begin{enumerate}
\item $\mathcal L_g \varphi = E$ on $\Spext$.
\item $\|\varphi:C^{2,\beta}(\Sp_1,r,g)\| \leq C(\bunder,\beta)\|E:C^{0,\beta}(\Sp_1, r,g)\|$.
\item $\|\varphi:C^{2,\beta}(\Lf,r,g,r^\gamma)\| \leq C(\bunder,\beta, \gamma)\|E:C^{0,\beta}(\Sp_1,r,g)\|$.
\item $\|\varphi:C^{2,\beta}(\Lc,r,g,(\rin/r)^{n-1})\| \leq C(\bunder,\beta)\rin^{1-n}\|E:C^{0,\beta}(\Sp_1,r,g)\|$.
\item $\mathcal R_{\Spext}$ depends continuously on $\tau$.
\end{enumerate}
\end{lemma}

\begin{proof}Consider $E$ in the domain of $\mathcal R_{\Spext}$ and decompose $E = E_0 + E_1 + E_{\mathrm{high}}$. For $E_0$, there exists a unique $\varphi_0(t)$ such that $L_0 \varphi_0 = E_0$ and $\varphi_0(2\Pdo + b+1)=\varphi_0'(2\Pdo +b+1)=0$. Since $E_0 \equiv 0$ on $\Spext \backslash \Sp_1$, $\varphi_0 \equiv 0$ on $\Lf \backslash \Sp_1:= \Lf_{1,0}$. By standard ODE theory, we note that
\[
\|\varphi_0:C^{2,\beta}(\Sp_2,r,g)\|\leq C(\beta)\|E_0:C^{0,\beta}(\Sp_1,r,g,r^{-2})\|\leq C(\bunder,\beta)\|E_0:C^{0,\beta}(\Sp_1,r,g)\|\\
\]where the final inequality uses \ref{radiuslemma}. At $t = 2\Pdo -(b+1)$, determine the unique $a_0, b_0$ such that 
\[
\varphi_0(2\Pdo -(b+1)) = a_0 V_0[\Lc,1,0] (2\Pdo -(b+1))+ b_0V_0[\Lc,0,1](2\Pdo -(b+1))
\] where $V_0$ are the functions defined in \ref{annulardecaydef}. Then  on $\Lc \backslash \Sp_1:= \Lc_{0,1}$
\[
\varphi_0= a_0V_0[\Lc,1,0] + b_0V_0[\Lc,0,1].
\] Combining the estimates of \ref{annulardecaylemma} and the ones above imply
\[
\|\varphi_0:C^{2,\beta}(\Lc,r,g, (\rin/r)^{n-1})\|\leq C(\beta)\rin^{1-n}\|E_0:C^{0,\beta}(\Sp_1,r,g)\|.
\]For $E_1$ we proceed in a similar fashion and produce $\varphi_1$ such that $L_1 \varphi_1 = E_1$, $\varphi_1 \equiv 0$ on $\Lf_{1,0}$ and 
\[
\|\varphi_1:C^{2,\beta}(\Sp_1,r,g)\|\leq C(\bunder,\beta)\|E_1:C^{0,\beta}(\Sp_1,r,g)\|.
\]Using this estimate and the fact that $L_1\varphi_1 \equiv 0$ on $\Lc_{0,1}$, we determine $a_i, b_i$ such that
$\varphi_1 = \sum_{i=1}^n V_i[\Lc,a_i,b_i]$ on $\Lc_{0,1}$. The estimate
\[
\|\varphi_1:C^{2,\beta}(\Lc,r,g, (\rin/r)^{n-1})\|\leq C(\beta)\rin^{1-n}\|E_1:C^{0,\beta}(\Sp_1,r,g)\|
\]follows again by combining the estimates on the coefficients with the estimates of \ref{annulardecaylemma}. 

Finally, for $E_{\mathrm{high}}$ we determine $\varphi_{\mathrm{high}} = \mathcal R_{\mathrm{high}}^+(E_{\mathrm{high}})$ by \ref{highharmonicslemma} which provides the decay estimate on $\Lf$ and does not contribute to growth on $\Lc$.

Setting
\[
\varphi:= \varphi_0 + \varphi_1 + \varphi_{\mathrm{high}}
\]implies the result.
\end{proof}

\begin{lemma}\label{linearminuslemma}Given $\beta \in (0,1), \gamma \in (1,2)$,
for each $\Sm$ there exists a a linear map
\[
\mathcal R_{\Smext}:\{E\in C^{0,\beta}(\Smext): E \text{ is supported on } \Sm_1\} \to C^{2,\beta}(\Smext, g)
\]such that the following hold for $E$ in the domain of $\mathcal R_{\Smext}$ and $\varphi = \mathcal R_{\Smext}(E)$:
\begin{enumerate}
\item $\mathcal L_g \varphi = E$ on $\Smext$.
\item $\|\varphi:C^{2,\beta}(\Sm_1,r,g)\| \leq C(\beta,\gamma)\|E:C^{0,\beta}(\Sm_1, r,g, r^{-2})\|$.
\item $\|\varphi:C^{2,\beta}(\Lf,r,g,(\rin/r)^{n-2+\gamma})\| \leq C(\beta,\gamma)\|E:C^{0,\beta}(\Sm_1,r,g,r^{-2})\|$. 
\item $\|\varphi:C^{2,\beta}(\Lc,r,g,r)\| \leq C(\beta,\gamma)\rin^{-1}\|E:C^{0,\beta}(\Sm_1,r,g,r^{-2})\|$. 
\item $\mathcal R_{\Smext}$ depends continuously on $\tau$.
\end{enumerate}
\end{lemma}
\begin{proof}
The proof is essentially identical to the proof for $\Spext$, though we use the estimates appropriate for growth away from $\rin$ on $\Lc$. 
We skip the details. 
\end{proof}

\section{The Initial Hypersurfaces}\label{InitialSurface}

In this section we assume given a family of graphs $\calF$---defined as in 
\ref{FamilyDefinition}---and we construct families of initial immersions which depend 
on a parameter $\utau$ which determines an overall scaling for the weights.  
The first step in the construction is to describe an abstract surface $M$ based on the central graph $\Gamma$ of $\calF$. 
At the same time we construct parametrizations for $M$ which depend on $\Gamma$ and $\utau$. 
We then define a family of immersions of $M$ into $\Rn$ which 
depends on $\utau$ and is parametrized by parameters $(d,\boldsymbol \zeta)$.  
The construction of each initial immersion is based on one of the graphs of $\calF$ chosen on the basis of $(d,\boldsymbol \zeta)$ and $\utau$.  

\begin{assumption}\label{ass:tgamma}
In what follows $\bunder \gg 1$ will be as in Section \ref{DelaunayLinear}, 
large enough to invoke all of the results of that section, 
but independent of the small constant $\maxT>0$. 
In this section, 
we choose a small constant $\maxTG>0$ which will depend on $\maxT>0$ (and thus on $\bunder$) and  
on $\max_{e \in E(\Gamma) \cup R(\Gamma)} \left|\hat \tau[e]\right|$ but not on the structure of the graph $\Gamma$ or on the parameters $d,\boldsymbol \zeta$. 
Note in particular that $\bunder$ will be independent of $\maxTG$. 

While we are free to decrease $\maxTG$ as necessary, we presume throughout this section that
\begin{equation}
\max_{e \in E(\Gamma) \cup R(\Gamma)} \left|\hat \tau[\Gamma(0,0),e]\right|\maxTG <\maxT/2.
\end{equation}
Moreover, the constant $\utau$ will be chosen so that
\[
0<|\utau| < \maxTG.
\]
\end{assumption}
Let $\tz:=E(\Gamma) \cup R(\Gamma)\to \Real\backslash\{0\}$ such that
\[\tz:= \utau \hat \tau[\Gamma(0,0),e].
\]
\begin{remark}
Note that the choice of $\maxTG$ implies that 
\[
0<|\tz|<\maxT/2.
\]
\end{remark}
\subsection*{The abstract surface $M$}
Given a flexible, central graph $\Gamma$ with the rescaled function $\tau_0$, we determine an abstract surface which will be mapped into $\Real^{n+1}$ by translating and rotating the maps described in Section \ref{BuildingBlocks}.  We construct $M$ in the following manner, noting that $M$ depends only on $\Gamma$ and $\utau$ and not on $d,\boldsymbol \zeta$.

\begin{definition} 
\label{def:a2}
We choose $\delta'>0$, depending only on $\Gamma$, 
such that for each $p \in V(\Gamma)$ and all 
$e\ne e' \in E_p$ we have $| \mathbf v[p,e] - \mathbf v[p,e'] | >50\delta'$. 
Recall that by \eqref{adef} this defines also a constant $a$ such that 
$\tanh(a+1)=\cos(\delta')$.
\end{definition}

\begin{definition} 
For $p \in V(\Gamma)$
define 
\begin{equation}
M[p]= \mathbb S^n_{V_p}:=\mathbb S^n \backslash  D^{\Ssn}_{V_p}(\delta') 
\quad \text{ where } \quad 
V_p := \{\Bv\pe \,:\, e \in E_p\}.
\end{equation}

As the length of each edge domain depends upon the period and the number of periods, we set
\begin{equation*}
\RH:= 2\Pe l[e].
\end{equation*}
For $e \in E(\Gamma)$, let
\begin{equation*}
M[e] = [a,\RH-a]\times \Ssn
\end{equation*}
while for $e \in R(\Gamma)$, let
\begin{equation*}
M[e] = [a,\infty)\times \Ssn.
\end{equation*}
\end{definition}

\begin{definition}\label{ReDef}
For $e \in E(\Gamma) \cup R(\Gamma)$, let $\RRR[e]:\Rn \to \Rn$ denote the rotation such that
\[
\RRR[e](\Be_i)=\Bv_i[e]
\]for $i=1,\dots, n+1$, where here the $\Bv_i[e]$ refer to the ordered orthonormal frame chosen in \ref{gammaframe}. (The existence of such a rotation
follow precisely because we chose an ordered frame.)
\end{definition}
\begin{definition}
Let 
\[
M'= \left(\bigsqcup_{p \in V(\Gamma)}M[p]\right) \bigsqcup\left( \bigsqcup_{e \in E(\Gamma) \cup R(\Gamma)} M[e]\right)\]
and let 
\begin{equation}\label{eq:sim}
M=M'/\sim 
\end{equation}
where we make the following identifications:\\
For $\pe \in A(\Gamma)$ with $p=p^+[e]$ and $x \in M[e] \cap \left([a,a+1] \times \Ssn\right)$,
\[
 x \sim \left(\RRR[e] \circ Y_0(x)\right) \cap M[p].
\]
For $\pe \in A(\Gamma)$ with $p=p^-[e]$ and $ x\in M[e] \cap \left([\RH-(a+1),\RH-a] \times \Ssn\right)$,
\[
 x \sim \left( \RRR[e] \circ Y_0 (t-\RH,\bt )\right) \cap M[p].
\]

\end{definition}

\subsection*{Standard and transition regions}In enumerating the important regions of the graph, we frequently reference the triple $[p,e,\cdot]$ where the third
component will be described below. For each  $e\in E(\Gamma) \cup R(\Gamma)$, we enumerate the standard and transition regions along 
the Delaunay piece by counting upward
as we move away from each central sphere. As in \cite{BKLD}, we denote a region as \emph{standard} if the limiting geometry as $\utau \to 0$ is well understood and a region as \emph{transition} otherwise. See Section \ref{DelSection} for a more complete description of these regions. Recall that $2l[e]$ denotes the length
of an edge $e$. Thus, an edge $e$ will have
$2l[e]-1$ standard regions and $2l[e]$ transition regions. We make precise the following definition.
\begin{definition} 
\label{pendef}
We define 
\begin{align*}
 \VS(\Gamma) := &\{[p,e,m] :e \in E(\Gamma),\ppe\in A(\Gamma), m \in \{1, 2, \dots, l[e]\}\}\\&\bigcup\{\pen : e \in E(\Gamma),\pme\in A(\Gamma), m \in\{1, 2, \dots, l[e]-1\} \}
 \\&\bigcup \{\pen : e \in R(\Gamma), [p,e] \in A(\Gamma), m \in \mathbb N\}, \\ 
 \VSp(\Gamma) := &\{[p,e,m] \in\VS(\Gamma) : m\text{ is even}\,\}, \\ 
 \VSm(\Gamma) := &\{[p,e,m] \in\VS(\Gamma) : m\text{ is odd}\,\}, \\  
\VN(\Gamma):=&\{ [p,e,m']:e \in E(\Gamma),[p^\pm[e],e] \in A(\Gamma), m' \in \{1, 2, \dots, l[e]\}\} \\ 
&\bigcup \{[p,e,m'] : e \in R(\Gamma), [p,e] \in A(\Gamma), m' \in \mathbb N\}.
\end{align*}
\end{definition}
We choose this notation so that the set $\VS(\Gamma)$ enumerates every standard region on an edge or ray exactly once. 
Moreover, the enumeration of the standard regions is such that it increases along $M[e]$ as one moves further away from the nearest boundary. 
$\VSp(\Gamma)$ and $\VSm(\Gamma)$ enumerate the spherical and catenoidal regions respectively. 
$\VN(\Gamma)$ enumerates every transition region exactly once. 
Notice that $V_S(\Gamma) \subset \VN(\Gamma)$ and $\VN(\Gamma) \backslash V_S(\Gamma)= \{ [p^-[e],e,l[e]] : e \in E(\Gamma)\}$. 

We now define regions of particular importance. 
A verbal description of these regions follows.
\begin{figure}[h]
\begin{center}
\includegraphics[width=5in]{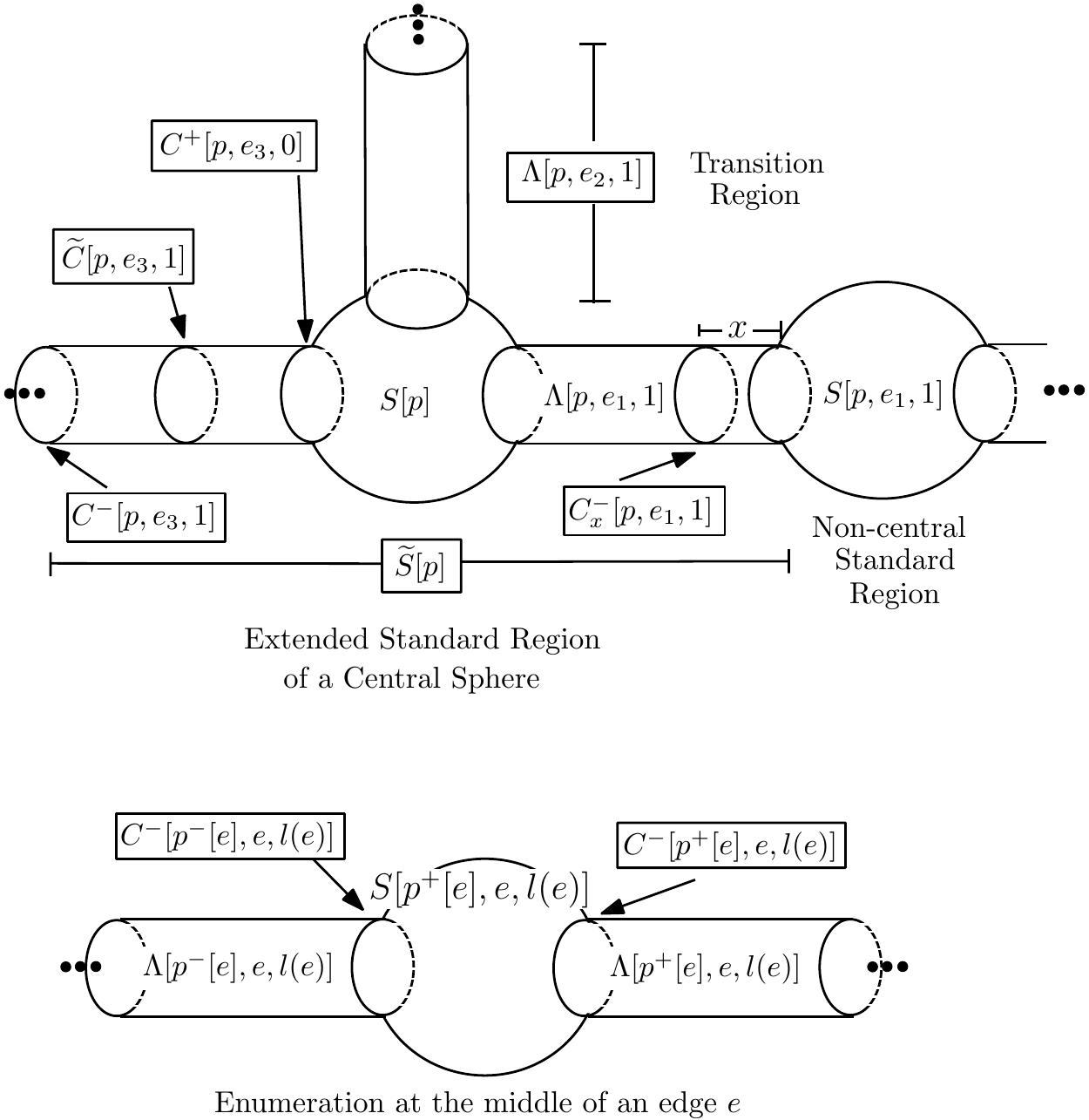}
            \caption{Two schematic renderings of regions of $M$. The top one is near a vertex $p$ with $|E_p|=3$ and the bottom one at a standard region associated to the center of an edge $e$. Note that, in the figure, standard regions appear spherical and transition regions appear cylindrical.}\label{STRPthin}
     \end{center}
      \end{figure}
Recall that $a$ is determined by \ref{def:a2}. 
The constant $\bunder$ determines the size of each standard and transition region. 
We use $x,y$ in subscripts to modify the size of the regions and the boundary circles. 
For example, $S[p] \subset S_x[p]$ while $\widetilde S_x[p] \subset \widetilde S[p]$.

\begin{definition} 
\label{regions} 
For $p \in V(\Gamma)$, 
$\pen\in \VS(\Gamma)$,  
and $[p,e,m']\in \VN(\Gamma)$, (recall \ref{pendef}), 
we define the following. 
\begin{enumerate}
\item \label{centstand} $S_x[p]:=M[p] \bigsqcup_{\{e|p=p^+[e]\}}\left(M[e]\cap[a, \bunder+x]\times \Ssn\right)$\\
\indent \indent \indent \indent $\bigsqcup_{\{e|p=p^-[e]\}}\left(M[e] \cap [\RH-(\bunder+x),\RH-a] \times \Ssn\right)$
\item \label{centextstand}$\widetilde S_x[p]:=M[p]\bigsqcup_{\{e|p=p^+[e]\}}\left(M[e]\cap[a, \Pe-( \bunder+x)]\times \Ssn\right)$
\\ \indent \indent \indent\indent $\bigsqcup_{\{e|p=p^-[e]\}}\left(M[e] \cap [\RH-(\Pe-( \bunder+x)),\RH-a] \times \Ssn\right)$
\item \label{stand1}$S_x[p^+[e],e,m]:= M[e]\cap [m\Pe -(\bunder+x),m\Pe +(\bunder+x)]\times \Ssn$
\item \label{stand2}$S_x[p^-[e],e,m]:= M[e]\cap [\RH-(m\Pe +(\bunder+x)),\RH-(m\Pe -(\bunder+x))]\times \Ssn$
\item \label{extstand1}$\widetilde S_x[p^+[e],e,m]:=  M[e]\cap [(m-1)\Pe +(\bunder+x),(m+1)\Pe -(\bunder+x)]\times \Ssn$
\item \label{extstand2}$\widetilde S_x[p^-[e],e,m]:=  M[e]\cap [\RH-((m+1)\Pe -(\bunder+x)),\RH-((m-1)\Pe +(\bunder+x))]\times \Ssn$
\item\label{neckregion1} $\Lambda_{x,y}[p^+[e],e,m']:=  M[e]\cap  [(m'-1)\Pe +(\bunder+x),m'\Pe-(\bunder+y)]\times \Ssn$
\item\label{neckregion2} $\Lambda_{x,y}[p^-[e],e,m']:=  M[e]\cap  [\RH-(m'\Pe-(\bunder+y)),\RH-((m'-1)\Pe +(\bunder+x))]\times \Ssn$
\item $\Cout_x[p^+[e],e,m']:=  M[e]\cap \{(m'-1)\Pe +(\bunder+x)\}\times \Ssn$ for $m'$ odd,
\item $\Cout_x[p^+[e],e,m']:=  M[e]\cap \{m'\Pe -(\bunder+x)\}\times \Ssn$ for $m'$ even,
\item $\Cout_x[p^-[e],e,m']:=  M[e]\cap \{\RH-((m'-1)\Pe +(\bunder+x))\}\times \Ssn$ for $m'$ odd,
\item $\Cout_x[p^-[e],e,m']:=  M[e]\cap \{\RH-(m'\Pe -(\bunder+x))\}\times \Ssn$ for $m'$ even,
\item $\Cin_x[p^+[e],e,m']:=  M[e]\cap  \{m'\Pe -(\bunder+x)\}\times \Ssn$ for $m'$ odd,
\item $\Cin_x[p^+[e],e,m']:=  M[e]\cap  \{(m'-1)\Pe +(\bunder+x)\}\times \Ssn$ for $m'$ even,
\item $\Cin_x[p^-[e],e,m']:=  M[e]\cap  \{\RH -(m'\Pe -(\bunder+x))\}\times \Ssn$ for $m'$ odd,
\item $\Cin_x[p^-[e],e,m']:=  M[e]\cap  \{\RH-((m'-1)\Pe +(\bunder+x))\}\times \Ssn$ for $m'$ even,
\end{enumerate}
The constant $\bunder>a+5$ was initially determined in Section \ref{DelaunayLinear} but may be further increased in forthcoming sections as necessary. We let $0\leq x,y<\Pe-\bunder$ where positivity of $\Pe-\bunder$ is guaranteed by the smallness of $\maxTG$ in relation to $\bunder$.
We set the convention to drop the subscript $x$ when $x=0$; i.e. $S[p]=S_0[p]$.  Moreover, we denote $\Lambda_{x,x}=\Lambda_x$.
\end{definition}

Notice that unlike in the case $n=2$, not all of the regions $S\pen$ have the same geometric limit as $\utau \to 0$. With this notation, each $S\pen\subset M$ with $\pen \in \VS^+(\Gamma)$ will correspond to a \emph{standard region} or \emph{almost spherical region}. For $\pen \in \VS^-(\Gamma)$, $S\pen$ corresponds to a \emph{standard region} or \emph{almost catenoidal region}. Each $\Lambda[p,e,m']$ will correspond to a \emph{transition} or \emph{neck} region. For $e \in E(\Gamma)$, the middle standard region on $M[e]$ bears the label $S[p^+[e],e,l[e]]$. Each $\widetilde S\pen$ is an \emph{extended standard region} and contains both the standard region and the two adjacent transition regions. The $\widetilde S[p]$ are \emph{central extended standard regions} and contain all adjacent transition regions, where adjacency is determined by $e \in E_p$. 

Finally, note that the spheres $\Cout, \Cin$ are enumerated so that
\[
\partial \Lambda_{x,y}[p,e,m']=\Cout_x[p,e,m'] \cup \Cin_y[p,e,m'] \text{ for } m' \text{ odd}, 
\]
\[
\partial \Lambda_{x,y}[p,e,m']=\Cin_x[p,e,m'] \cup \Cout_y[p,e,m'] \text{ for } m' \text{ even}.
\]The superscripts $\mathrm{out,in}$ are used to match those that are used throughout Section \ref{DelaunayLinear}. 
We extend the definition here to include all meridian spheres that exhibit the same behavior under an immersion as those from \ref{domaindefinitions}.

\subsection*{The graph $\Gamma(\tilde d,\tilde \ell)$}
We use the parameters $\dz$ to determine a graph in $\calF$. 
Recall that by assumption $\Gamma$ is a central graph in a family $\calF$.

We presume throughout that $d: V(\Gamma)\to \Real^{n+1}$ where
\begin{equation}\label{drestriction}
\left|d[p ]\right| \leq |\utau|^{1 + \frac 1{n-1}} \text{ for all } p \in V(\Gamma).
\end{equation}

Choose $\tilde d \in D(\Gamma)$ (recall \ref{n1}, \ref{Def:dlz}) such that
\begin{equation}\label{Find_d}
\tilde d [\cdot]= \frac 1{\utau}d[\cdot].
\end{equation}

Choose $\Gamma(\tilde d,0) \in \calF$ and let
\[
\taue:= \utau \hat \tau[\Gamma(\tilde d,0),e].
\]
\begin{remark}The smooth dependence of $\Gamma(\tilde d,0)$ on $\tilde d$ implies that 
\[
0<|\td|< |\tz|(1+C|\utau|^{\frac 1{n-1}})<\maxT.
\]
\end{remark}
We now determine the value of the function $\tilde \ell \in L(\Gamma)$ (recall \ref{n1}, \ref{Def:dlz}) that will rely -- for each $e$ -- on $l[e], \taue,$ and two vectors $\boldsymbol \zeta[p^\pm,e]\in \Rn$.
The maps $\boldsymbol \zeta[p^\pm,e]$ will effectively describe the dislocation of each attached Delaunay piece from its central sphere. Though rays are not in the domain
of $\tilde \ell$, they can be dislocated from their vertex, and thus when describing $\boldsymbol \zeta[p,e]$ we must include rays in the domain.

\begin{definition}\label{zetadef}Let $\boldsymbol \zeta \in Z(\Gamma)$ (recall \ref{n1}, \ref{Def:dlz}) such that 
\begin{equation}
\boldsymbol \zeta[p,e]=\sum_{i=0}^{n} \zeta_i[p,e]\mathbf e_{i+1}.
\end{equation}
\end{definition}
\noindent As we will see, the norm of $\boldsymbol\zeta$ can be quite large compared to the norm of $d$. Throughout the paper, we allow 
\begin{equation}\label{zetarestriction}
\left| \boldsymbol\zeta\right|\leq \underline C   |\utau|
\end{equation}where $\underline C$ is a large, universal constant that is independent of $\utau$.

Let $\widetilde l \in L(\Gamma)$ such that
\begin{equation}
 \widetilde l[e]:= \left(2+2\Pimd\right) l[e].
\end{equation}Thus, a Delaunay piece with $l[e]$ periods and parameter $\taue$ will have length -- i.e. axial length -- equal to $\widetilde l[e]$. 
Recall \eqref{first_ell_def} which informs our choice of $\tilde \ell$.  
\begin{definition}
Choose $\tilde \ell\in L(\Gamma)$ such that
\begin{equation}\label{ellprimedefinition}
2 (l[e]+\tilde \ell[e])=\left|{\boldsymbol \zeta}\ppe - \left( {\boldsymbol \zeta}\pme + (\widetilde l[e],\boldsymbol 0) \right) \right|.
\end{equation}

\end{definition}

\begin{figure}[h]
\begin{center}
\includegraphics[width=5in]{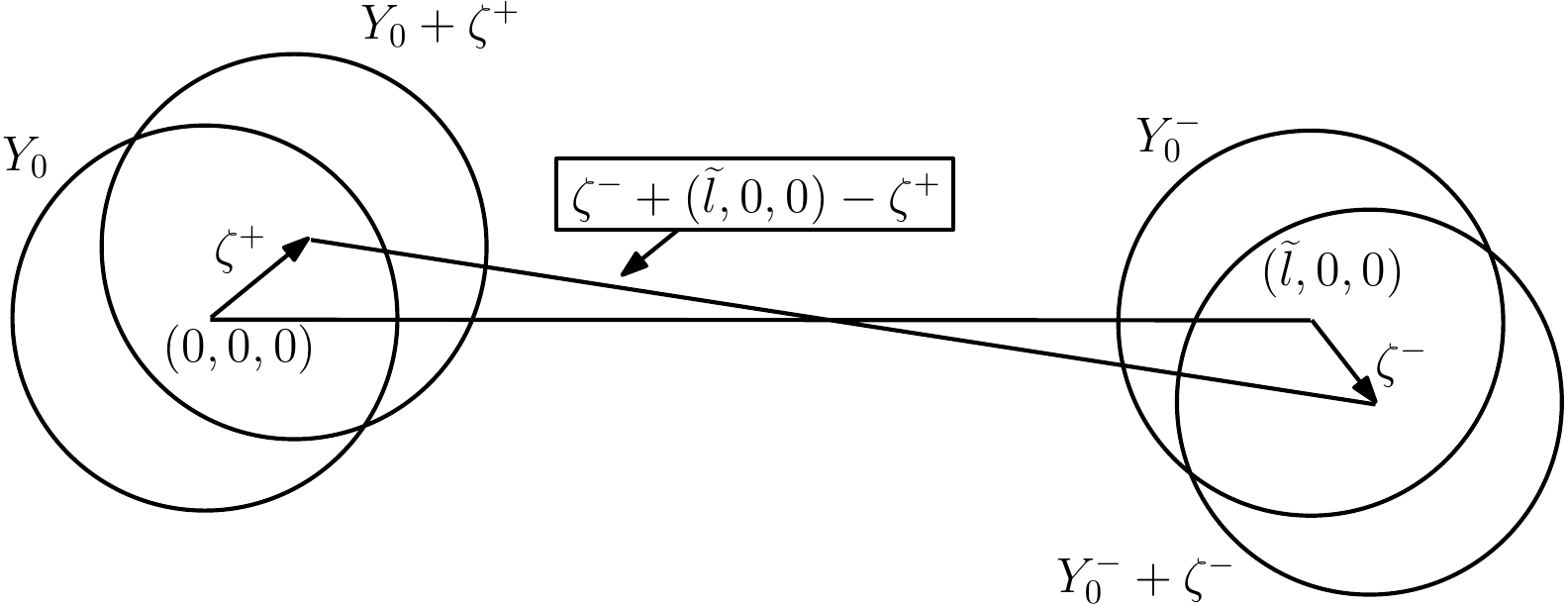}
            \caption{In the figure, we let $\zeta^+, \zeta^-$ correspond to $\boldsymbol \zeta[p^+[e],e], \boldsymbol \zeta[p^-[e],e]$ respectively. Also, notice that $Y_0^-$ is defined so that its center is at $(2+2\Pimd) l[e]\Be_1$.}\label{Dis}
            \end{center}
      \end{figure}
For clarity, we provide a systematic description of $\tilde \ell$. First, we position a segment of length $\widetilde l[e]$ so that it sits on
 the positive $x_1$-axis with one end fixed at the origin. Then we dislocate the two ends of this segment corresponding to $\boldsymbol \zeta\ppe$ and $\boldsymbol \zeta\pme$
where $\boldsymbol \zeta\ppe$ is the dislocation of the origin. We then measure the length of the segment connecting these two points. Finally,
we compare that length with the length of the edge $e$ in the graph $\Gamma$.

\begin{lemma}
For $\tilde \ell$ defined as in \eqref{ellprimedefinition}, we may decrease $\maxTG>0$ so that for all $0<|\utau|<\maxTG$, there exists $C>0$ depending on $\Gamma$ but independent of $\utau$, such that for all $e \in E(\Gamma)$, 
\begin{equation}\label{lrestriction}
 \left|\tilde \ell[e]\right| \leq {C}|\utau|^{\frac 1{n-1}}.
\end{equation}

\end{lemma}
\begin{proof}We immediately get the bounds
\[
\widetilde l[e]-2 \left|\boldsymbol \zeta\right| \leq \left|{\boldsymbol \zeta}\ppe - \left( {\boldsymbol \zeta}\pme + (\widetilde l[e],\boldsymbol 0) \right) \right| \leq \widetilde l[e] + 2\left|\boldsymbol \zeta\right|.
\]Thus,
\[
{ \frac{\widetilde l[e] - 2\left|\boldsymbol \zeta\right|}{2}}-{l[e]} \leq{\tilde \ell[e]}\leq \frac { \widetilde l[e] + 2\left|\boldsymbol \zeta\right|}{2}-{l[e]}.
\]The definition of $\widetilde l[e]$, the bound on $\boldsymbol \zeta$ given by \eqref{zetarestriction}, and the estimates of \eqref{Pim_est} then immediately imply the result.
\end{proof}

\begin{lemma}
For a central graph $\Gamma$ with an associated family $\calF$, 
we may decrease $\maxTG>0$ so that for all $0<|\utau|<\maxTG$ and $d, \boldsymbol \zeta$ as in \eqref{drestriction}, \eqref{zetarestriction}, 
there exists $\Gamma(\tilde d, \tilde \ell) \in \calF$ with $\tilde d$, 
$\tilde \ell$ given by \eqref{Find_d} and \eqref{ellprimedefinition} respectively, and a constant $C>0$ depending on $\Gamma$ but independent of $\utau$ such that
\begin{enumerate}\item
\begin{equation}\label{tauratio}
 \frac{\tau_d[e]}{\tau_0[e]}\in \left(1-C|\utau|^{\frac 1{n-1}}, 1+C|\utau|^{\frac 1{n-1}}\right),
\end{equation} and
\begin{equation}\label{diffeodifference}
\left| 1 - \frac{\Pde}{\Pe}\right| \leq  -C \frac{|\utau|^{\frac 1{n-1}}}{\log (|\tz|)} \leq -C \frac{|\utau|^{\frac 1{n-1}}}{\log (C|\utau|)}.
\end{equation}
\item Recalling \ref{FrameLemma},
\begin{equation}\label{tauratio2}
\angle(\Bv_1[e;0,0], \Bv_1[e;\tilde d,0]) \leq C|\utau|^{\frac 1{n-1}}
\end{equation}
\item 
\begin{equation}\label{ddifftau}
\angle(\Bv_1[e;\tilde d,0], \Bv_1[e;\tilde d,\tilde \ell])  \leq C \underline C|\utau|.
\end{equation}
\end{enumerate}
\end{lemma}

\begin{proof}
The smooth dependence of $\Gamma(\tilde d, \tilde \ell)$ on $(\tilde d, \tilde \ell)$ and \eqref{drestriction} together imply \eqref{tauratio} and \eqref{tauratio2}. 
To see \eqref{diffeodifference}, note that by \eqref{tauratio} and \ref{periodasymptotics}, there exists $\tau'$ between $\td, \tz$ such that 
\[
\left| 1-  \frac{\Pde}{\Pe}\right| =\frac{|\td-\tz||\frac{d}{d\tau}|_{\tau =\tau'}\mathbf p_{\tau'}|}{|\Pe|}\leq C \frac{\left| 1- \frac{\td}{\tz}\right|}{|\log |\tz||}\leq -C\frac{|\utau|^{\frac 1{n-1}}}{\log |\tz|}.
\]
Finally, to see \eqref{ddifftau} let $\theta[e]:= \angle(e,e')$. At worst,
\[
\sin \theta[e] \leq \frac{2 \left|\boldsymbol \zeta\right|}{\sqrt{\widetilde l[e]^2 + 4\left|\boldsymbol \zeta\right|^2}} \leq C\left|\boldsymbol \zeta\right|.
\]Thus, $\theta[e] \leq C \underline C |\utau|$. 
\end{proof}

\begin{remark}Since $\tz/\utau = \hat \tau[\Gamma(0,0),e]$, the finiteness of the graph $\Gamma$ and \eqref{tauratio} imply that there exists $C$ depending only on $\Gamma$ such that ${|\tau_d[e]|}\sim_C{|\utau|}$. This gives us the freedom to replace any bounds in $|\tau_d[e]|^{\pm 1}$ by $C|\utau|^{\pm 1}$, reducing notation and bookkeeping.
\end{remark}

\subsection*{The smooth immersion}
The immersion we describe is an appropriate positioning of the building blocks described in Section \ref{BuildingBlocks}. Notice that the building blocks depend upon $\Gamma$ and the parameters $d, \boldsymbol \zeta$ and on $\utau$, but the immersions describing the building blocks are determined prior to any positioning.

For each $e \in E(\Gamma)$, the positioning of the associated Delaunay building block will depend upon a rotation that takes an orthonormal frame of the edge connecting the vectors $\tilde l[e] + \boldsymbol \zeta[p^-[e],e]$ and $\boldsymbol \zeta[p^+[e],e]$ to the orthonormal frame of the corresponding edge $e' \in E(\Gtdtl)$. We first prove that this rotation is well defined and determine the estimates we will need.

\begin{prop}\label{zetaframe}
For $\boldsymbol \zeta$ as defined in \ref{zetadef} and each $e \in E(\Gamma)$ there exists a unique orthonormal frame $F_{\boldsymbol \zeta}[e]=\{\Be_1[e], 
\dots, \Be_{n+1}[e]\}$,
 depending smoothly on $\boldsymbol \zeta$,
 such that
\begin{enumerate}
 \item  $\Be_1[e]$ is the unit vector parallel to  $\boldsymbol \zeta\pme + (\widetilde l[e],\boldsymbol 0) -\boldsymbol \zeta\ppe$ such that $\Be_1[e] \cdot \mathbf e_1>0$.
\item For $i=2, \dots, n+1$, $\Be_i[e]=\RRR[\Be_1, \Be_1[e]](\Be_i)$.
\item For $\Bv \in \Rn$, 
\begin{equation}\label{zetaframebound}
|\Bv - \RRR[\Be_1, \Be_1[e]](\Bv)| \leq C \left| \boldsymbol \zeta \right| \, |\Bv|.
\end{equation}
\end{enumerate}
\end{prop}
\begin{proof}
The first two items are by definition. If $\Be_1 = \Be_1[e]$ then the third item is immediately true as the rotation is the identity matrix. Now suppose $\Be_1 \neq \Be_1[e]$. By \ref{rotationdefn}, for $\Bv = \Bv^T +\Bv^\perp$ where $\Bv^T$ is the projection onto the $2$-plane spanned by $\Be_1,\Be_1[e]$, $\RRR[\Be_1,\Be_1[e]](\Bv)=\RRR[\Be_1,\Be_1[e]](\Bv^T)+ \Bv^\perp$.

Writing $\Bv^T = a_1 \Be_1 + a_2 \left(\frac{\Be_1[e] - \Be_1 \cos \theta[e]}{\sin \theta[e]} \right):= a_1\Be_1 + a_2 \mathbf z$, the definition of the rotation implies that
\[
\RRR[\Be_1,\Be_1[e]](\Bv^T)-\Bv^T = \sin \theta[e](a_1 \mathbf z - a_2 \Be_1) + (1-\cos \theta[e]) \Bv^T
\]where $\theta[e]$ is the smallest angle between $\Be_1, \Be_1[e]$. Recall, in the proof of \eqref{ddifftau}, we observed that $\sin \theta[e] \leq C\left| \boldsymbol \zeta \right|$. 
Therefore, $\cos \theta[e] \geq 1- C\left| \boldsymbol \zeta \right|^2$. It follows that
\[
\left|\RRR[\Be_1,\Be_1[e]](\Bv^T)-\Bv^T \right| \leq C\left| \boldsymbol \zeta \right|\,|\Bv^T|.
\]
\end{proof}
\begin{definition}
 For $e \in R(\Gamma)$ we simply let
$\Be_i[e]:=\Be_i$.
\end{definition}
Using the frame previously defined, we describe the rigid motion that will position each Delaunay building block.

\begin{definition}
\label{defn:RT}
For each $e \in E(\Gamma)\cup R(\Gamma)$ with $e'$ denote the corresponding edge on the graph $\Gamma(\tilde d,\tilde \ell)$, 
let $\RRR[e;\dz]$ denote the rotation in $\Rn$ such that $\RRR[e;\dz] (\Be_i[e])=\Bvp_i[e; \tilde d,\tilde \ell]$ for $i=1,\dots, n+1$ (recall \ref{FrameLemma}).
Let $\TTT[e;\dz]$ denote the translation in $\Rn$ such that $\TTT[e;\dz](\RRR[e;\dz](\boldsymbol \zeta\ppe))=p^+[e']$. 
Letting $\UUU[e; \dz]= \TTT[e;\dz] \circ \RRR[e;\dz]$ we see that for all $c_i \in \Real$,
\begin{equation}
\UUU[e; \dz]\left(\boldsymbol \zeta\ppe + c_i \Be_i[e]\right)= p^+[e'] + c_i \Bvp_i[e; \tilde d,\tilde \ell].
\end{equation}
\end{definition}

At each $p' \in V(\Gamma(\tilde d,\tilde \ell))$, we position a spherical building block. The rigid motion required for positioning these building blocks is simply a translation. The immersion of the building block associated with $p'$ depends upon a diffeomorphism determined by the frames $F_\Gamma[e]$ and the frames $F_{\boldsymbol \zeta}[e]$, for $e \in E_p$ where $p \in V(\Gamma)$ corresponds to $p'$.

For each $p \in V(\Gamma)$, let $\{e_1, \dots, e_{|E_p|}\}$ be an ordering of the edges and rays that have $p$ as an endpoint. For $i=1, \dots, |E_p|$, let 
\[
F_i[p] = \{\Bv[p,e_i], \Bv_2[e_i], \dots, \Bv_{n+1}[e_i]\}.\]
Notice that $F_i[p]$ is a set of vectors where the first vector represents the direction the edge or ray $e$ emanates from $p$ in the graph $\Gamma$ and the next $n$ vectors complete the orthonormal frame $F_\Gamma[e_i]$ given in \ref{gammaframe}.  
Recalling \ref{Rnframe}, let 
\[
F_{\boldsymbol \zeta,i} [p]=\{\mathrm{sgn}[p,e_i] \RRR[e_i;\dz](\Be_1), 
 \dots, \RRR[e_i;\dz](\Be_{n+1})\}.\]
This set of vectors almost corresponds to rotating the elements of the standard frame in $\Rn$ by $\RRR[e_i;\dz]$. 
The only change from the rotation is on the first element, which will differ from the rotation by a minus sign if $p = p^-[e_i]$. 
For the reader, it may be useful to note that in general $\RRR[e;\dz](\Be_i)\neq \Bv_i[e;\tilde d,\tilde \ell]$. 
See Figure \ref{EmbeddingPic}.

\begin{figure}[h]
\begin{center}
\includegraphics[width=5in]{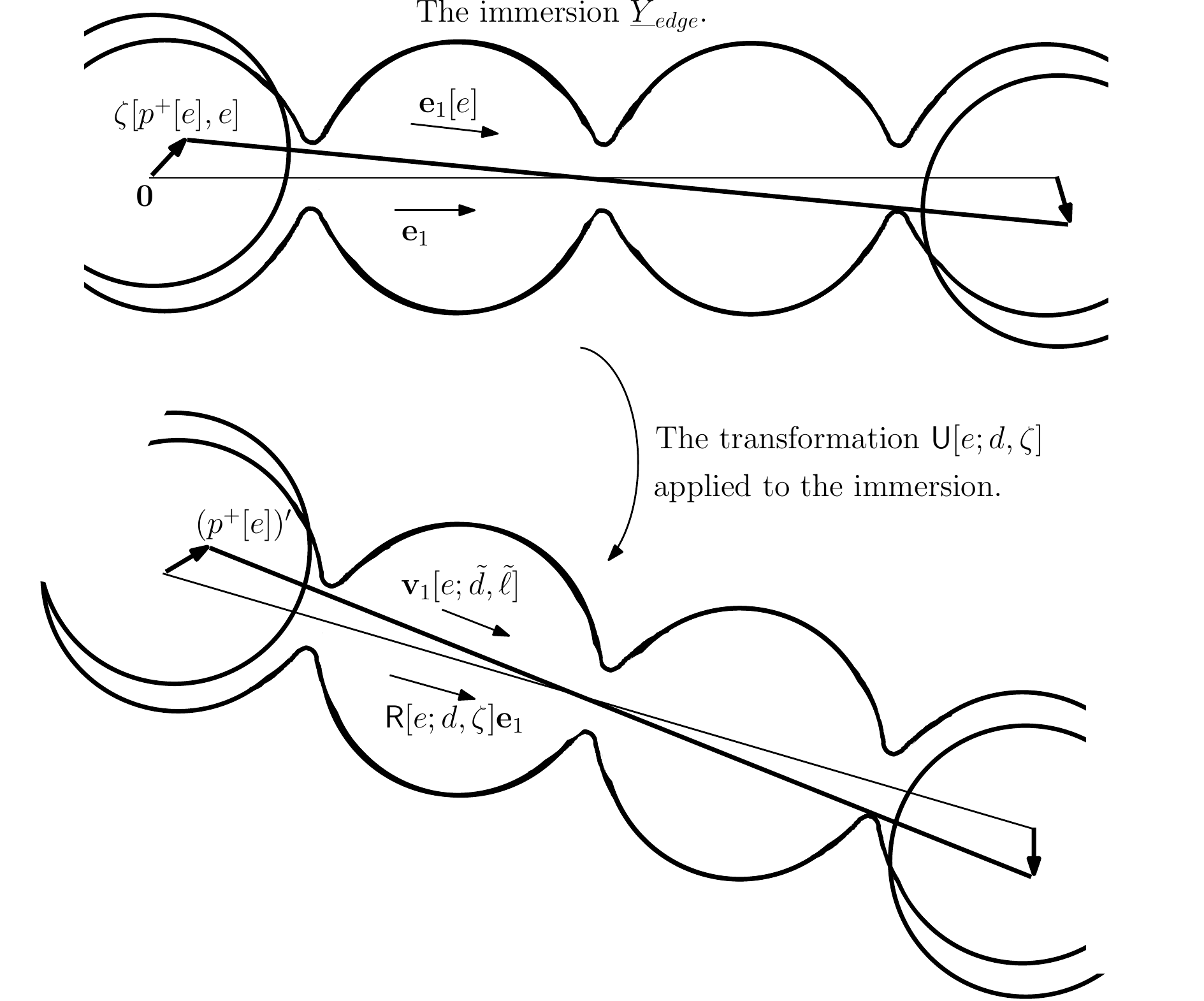}
            \caption{A rough idea of the immersion of one edge. Notice that the transformation $\UUU[e; \dz]$ sends the dislocated spheres to the vertices of the graph. The bold segment in the bottom picture corresponds to the positioning of the edge on the graph $\Gamma(\tilde d,\tilde \ell)$. The Delauney piece has axis parallel to $\RRR[e;\dz]\Be_1= \Bv_1[e;\tilde d, 0]$, which is parallel to the corresponding edge on the graph $\Gamma(\tilde d,0)$.}\label{EmbeddingPic}
                 \end{center}
                  \end{figure}

 These sets of vectors will determine the diffeomorphisms describing the spherical building blocks. The geodesic disks removed from each $M[p]$ will be repositioned under the diffeomorphism. The centers of the repositioned disks do not correspond to the vectors $\mathbf v[\Gtdtl,p,e]$. Rather, the repositioned disks will be centered at the vectors $\signep \RRR[e;\dz](\Be_1)$. The diffeomorphism $\hat Y$ defined in Section \ref{BuildingBlocks} will guarantee that the immersion is well-defined. 
Let 
\[
 W[p] :=\{ F_1[p] , \dots , F_{|E_p|}[p]\}, \;  W'[p] := \{F_{\boldsymbol \zeta,1}[p] , \dots , F_{\boldsymbol \zeta,|E_p|}[p]\}.
\]

\begin{definition}\label{tddefn}Let $\tsd:\bigsqcup_{e \in E(\Gamma) \cup R(\Gamma)} M[e] \to \Real$ such that for $e \in E(\Gamma)$, 
\begin{align}
\notag\tsd|_{M[e]}(t,\bt):= &\psi[a+3,a+2](t)\cdot t+\psi[\RH-(a+3),\RH-(a+2)](t)\cdot t \\&
+ \psi[a+2,a+3](t) \cdot \psi[\RH-(a+2),\RH-(a+3)](t) \cdot \left(  \frac{\Pde}{\Pe}t\right)
\end{align}and for $e \in R(\Gamma)$,
\begin{equation}
\tsd|_{M[e]}(t,\bt):= \psi[a+3,a+2](t)\cdot t+ \psi[a+2,a+3](t) \cdot \left(  \frac{\Pde}{\Pe}t\right).
\end{equation}
Note that $t_0(t,\bt)=t$.
\end{definition}

\begin{definition} 
\label{immersiondef}
Let $\hYtdz:M \to \Real^{n+1}$ be defined so that, recall \ref{defn:sphere},  \ref{defn:Yedge}, \ref{defn:RT},
\begin{equation*}
 \hYtdz(x):=\left\{ \begin{array}{ll}
                     p'+     \hat Y[ W[p],  W'[p]](x)& x \in M[p]\\
\UUU[e;\dz] \circ Y_{\mathrm{edge}}[ \taue, l[e], \boldsymbol \zeta\ppe, \boldsymbol \zeta\pme](\tsd(x),\bt)& x=(t,\bt) \in M[e], e \in E(\Gamma)\\
\UUU[e;\dz] \circ Y_{\mathrm{ray}}[ \taue, \boldsymbol \zeta\ppe]( \tsd(x),\bt)& x=(t,\bt) \in M[e], e \in R(\Gamma)
                           \end{array}\right.
\end{equation*}
where $p'\in V(\Gtdtl)$ is the vertex corresponding to $p $.

Let $\Htdz\in C^\infty(M)$ denote the mean curvature of $\hYtdz(M)$.
\end{definition}

Notice that a Delaunay building block will only be positioned parallel to the associated edge of $\Gtdtl$ if $\boldsymbol \zeta[p^+[e],e]=\boldsymbol \zeta[p^-[e],e]$ as in that case $\Be_1[e]=\Be_1$.
\begin{definition}\label{defn:H}
Recalling \ref{Defn:Herror}, define $H_{\mathrm{dislocation}}[\dz],H_{\mathrm{gluing}}[\dz]:M \to \Real$ in the following manner:
\begin{equation*}
H_{\mathrm{dislocation}}[\dz](x):=\left\{ \begin{array}{ll} H_{\mathrm{dislocation}}[\td,l[e], \boldsymbol \zeta^+[p^+[e],e], \boldsymbol \zeta^-[p^-[e],e]](\tsd(x),\bt)\\  \qquad \qquad \qquad \qquad \qquad \qquad x=(t,\bt) \in  M[e], e \in E(\Gamma),\\
H_{\mathrm{dislocation}}[\td, \boldsymbol \zeta^+[p^+[e],e]](\tsd(x),\bt)\\ \qquad \qquad  \qquad \qquad \qquad \qquad x=(t,\bt) \in  M[e] , e \in R(\Gamma),\\
  0\text{ otherwise},\end{array}\right.  
\end{equation*}
\begin{equation*}
H_{\mathrm{gluing}}[\dz](x):=\left\{ \begin{array}{ll} H_{\mathrm{gluing}}[\td,l[e], \boldsymbol \zeta^+[p^+[e],e], \boldsymbol \zeta^-[p^-[e],e]](\tsd(x),\bt)& x=(t,\bt) \in  M[e], e \in E(\Gamma)\\
H_{\mathrm{gluing}}[\td, \boldsymbol \zeta^+[p^+[e],e]](\tsd(x),\bt)& x=(t,\bt) \in  M[e] , e \in R(\Gamma),\\
  0&\text{otherwise}.\end{array}\right.  
\end{equation*}
\end{definition}
As an immediate consequence of the immersion and the definitions, and using the estimates of \ref{Cor:Herror}, we have the following characterization of the global mean curvature error function.
\begin{corollary}
\label{Hbounds}
All of the functions described above are smooth. Moreover the smooth function $H_{\mathrm{error}}[\dzeta]:=\Htdz -1:M \to \Real$ can be decomposed as
\[
 H_{\mathrm{error}}[\dzeta] = H_{\mathrm{dislocation}}[\dzeta] + H_{\mathrm{gluing}}[\dzeta].
\]Moreover,
\begin{enumerate}
\item $\|H_{\mathrm{gluing}}[\dz]:{C^k}( M[e] ,g)\| \leq C(a,k)|\utau|.$ 
\item $\|H_{\mathrm{dislocation}}[\dz]:C^k( M[e],g)\| \leq C(a,k)\left| \boldsymbol \zeta \right|\leq C(a,k)\underline C |\utau|$.
\end{enumerate}
\end{corollary}

\section{The linearized equation}\label{GlobalSection}

The goal of this section is to state and prove \ref{LinearSectionProp}. We demonstrate for any immersion $\hYtdz$ with $\dz$ satisfying \eqref{drestriction}, \eqref{zetarestriction} respectively, with $0<|\utau|<\maxTG$, we can modify the inhomogeneous term in such a way as to solve the linear problem in weighted H\"older spaces. While the construction of 
an initial surface in $\Rn$ is fairly similar for $n=2, n>2$, solving the linear problem for $n=2$ is much simpler than for $n>2$. There 
are a few reasons for this, not the least being that in two dimensions the Laplace operator simply scales by the conformal change.

\begin{assumption}\label{ass:tgamma6}
We presume throughout this section that $\maxTG>0$ and $\bunder \gg1$ satisfy the assumptions of the previous sections. Moreover, while $\bunder$ is fixed by the assumptions in Section \ref{DelaunayLinear}, we may further decrease $\maxTG$ as $\bunder$ is independent of $\maxTG$.

The constant $\utau$ will always satisfy $0<|\utau|<\maxTG$ and $\dz$ will satisfy \eqref{drestriction}, \eqref{zetarestriction} respectively for this fixed $\utau$. The immersion $\hYtdz$ will be as described in \ref{immersiondef}.
\end{assumption}
\begin{definition} 
Let $\ur_d:\bigsqcup_{e \in E(\Gamma) \cup R(\Gamma)}M[e]\to \Real$ such that
\begin{equation}
\ur_d|_{M[e]}:= r_{\td} \circ \tsd|_{M[e]} \text{ (recall \ref{tddefn})}.
\end{equation}
Moreover, let
\begin{equation}
\begin{gathered} 
 \urout[e;d]:=\ur_{d}(\bunder,\bt)=r_{\td}(b) , \qquad  
\urin[e;d]:=  \ur_{d}(\Pe-\bunder,\bt) =r_{\td}(\, \Pde-b\, ) , 
\end{gathered} 
\end{equation}
where $b:= t_d (\bunder,\bt)$   
and $\bunder$ is as in \ref{ass:b}. 
\end{definition}
Note that by \eqref{diffeodifference} $b$ is then as in \ref{ass:b}.

\begin{remark} On $M[e]\cap ([a+4,\RH - (a+4)]\times \Ssn)$ for $e \in E(\Gamma)$ and on $M[e]\cap ([a+4,\infty) \times \Ssn)$ for $e \in R(\Gamma)$,
\begin{equation}
g=\hYtdz^*(g_{\Real^{n+1}}) = \ur_d^2(d\tsd^2 + g_{\Ssn}).
\end{equation}
\end{remark}

\begin{lemma}\label{rratiolemma}On $\bigsqcup_{e \in E(\Gamma) \cup R(\Gamma)}M[e]$,
\[
\ur_d \sim_{C(\bunder)} \ur_0.
\]
\end{lemma}
\begin{proof}
First observe that by \ref{rmaxminlemma} and \eqref{tauratio}, for any $e \in E(\Gamma) \cup R(\Gamma)$, 
\[
\frac{\ur_{d}(\Pe, \bt)}{\ur_{0}(\Pe, \bt)} = \left(\frac{|\td|}{|\tz|}\right)^{\frac 1{n-1}}\left(1+O(|\utau|^{\frac 1{n-1}})\right)  = 1+O(|\utau|^{\frac 1 {n-1}}),
\]and
\[
\frac{\ur_{d}(2\Pe, \bt)}{\ur_{0}(2\Pe, \bt)} = 1+O(|\utau|).
\]Therefore, by the uniform geometry on each $S_{\bunder}\pen$, the previous estimates imply that for all $x \in S\pen$,
\[
\frac{\ur_d(x)}{\ur_0(x)} \sim_{C(\bunder)} 1.
\]

We will improve this estimate at $x = (\bunder, \bt)$ and use this improvement as the starting point to produce the equivalence on $\Lambda\pen$. By the triangle inequality, \ref{radiuslemma} and \eqref{diffeodifference},
\begin{align*}
|\ur_0(\bunder,\bt)-\ur_d(\bunder,\bt)| &\leq \left|r_{\tz}(\bunder) - \sech(\bunder) + \sech\left(\bunder\cdot \frac{\Pde}{\Pe}\right)-r_{\td}\left(\bunder\cdot \frac{\Pde}{\Pe}\right)\right| \\
& \quad \quad+ \left| \sech(\bunder) -\sech\left(\bunder\cdot \frac{\Pde}{\Pe}\right)\right|\\
& \leq C(\bunder) \left(|\utau|- \frac{|\utau|^{\frac 1{n-1}}}{\log |\tz|}\right). 
\end{align*}Thus, we may decrease $\maxTG$ so that 
\[
\left| \frac{\ur_d(\bunder,\bt)}{\ur_0(\bunder,\bt) }-1\right| \leq  |\utau|^{\frac 1{2(n-1)}}.
\]Let 
\[
f(x):= \log \frac{\ur_d(x)}{\ur_0(x)}
\]and observe that
\[
|f(\bunder, \bt)|\leq 2 |\utau|^{\frac 1{2(n-1)}}.
\]Going forward, we will assume always that $|f|< \frac 1{10}$ so that we are free to Taylor expand at will. Then, on any $\Lambda \pen$, letting $u_{\tau'}(t,\bt):= r_{\tau'}(\mathbf p_{\tau'}t/\Pe) + \tau' r_{\tau'}^{1-n}(\mathbf p_{\tau'} t/\Pe)$, 
\begin{align*}
\frac d{dt} f(t,\bt) &= \frac{\frac{d\ur_d}{dt_d}(t_d(t,\bt),\bt)\cdot\frac{\Pde}{\Pe}}{\ur_d(t,\bt)}- \frac{\frac{d\ur_0}{dt}(t,\bt)}{\ur_0(t,\bt)}\\
&=\sqrt{1- u_{\td}^2(t,\bt)}- \sqrt{1- u_{\tz}^2(t,\bt)} + \sqrt{1- u_{\td}^2(t,\bt)}\left( \frac{\Pde}{\Pe}-1\right)\\
&= -\frac{u_{\tau'}(t,\bt)}{\sqrt{1-u_{\tau'}^2(t,\bt)}}(u_{\td}(t,\bt)-u_{\tz}(t,\bt)) +  \sqrt{1- u_{\td}^2(t,\bt)}\left( \frac{\Pde}{\Pe}-1\right)
\end{align*} for some $\tau'$ between $\td, \tz$. As $\ur_d = e^f \ur_0$, 
\[
u_{\td} - u_{\tz} = (e^f-1)\ur_0 + (\td-\tz)\ur_0^{1-n} + \td(e^{(1-n)f}-1)\ur_0^{1-n}.
\]
Since we are presuming $|f|$ is small,
\[
 u_{\td} - u_{\tz} =(f\cdot h)\ur_0 + \left(\frac{\td}{\tz} -1\right)\tz\ur_0^{1-n} + \td(1-n)(f \cdot h) \ur_0^{1-n}
\]where $|h| \leq 1 + 2|f|$. 
Thus,
\[
\frac d{dt}f = fA +B
\]where
\[
A:= -\frac{u_{\tau'}}{\sqrt{1-u_{\tau'}^2}}\left( h \ur_0 + \td(1-n) h \ur_0^{1-n}\right) ,
\]
\[ B:=  -\frac{u_{\tau'}}{\sqrt{1-u_{\tau'}^2}}\left(\frac{\td}{\tz} -1\right)\tz\ur_0^{1-n} + \sqrt{1- u_{\td}^2}\left( \frac{\Pde}{\Pe}-1\right) .
\]Since we are presuming that $|f|$ is small, $| u_{\tau'} | \leq C(\ur_0+|\tz| \ur_0^{1-n})$. Moreover, since $\frac d{dt}\ur_0 \sim_{C(\bunder)} \ur_0$ on $\Lambda\pen$ as long as $|f|< 1/10$ we may estimate for $\ur_0(x) \in [\urin[e;0], \urout[e;0]]$,  
\begin{align*}
\left|\int_{\bunder}^{t_0(x)} A(t,\bt) \, dt\right| &\leq C(\bunder) \int_{\urout[e;0]}^{\ur_0(x)} \left( r + |\utau| r^{1-n} + |\utau|^2 r^{1-2n}\right) dr\\
& \leq C(\bunder,n)\left| \left( r^2 + |\utau|r^{2-n}+ |\utau|^2 r^{2-2n}\right)\big|_{\urout[e;0]}^{\ur_0(x)}\right|\\
& \leq C(\bunder, n).
\end{align*}Moreover, for all $\ur(x) \in[ \urin[e;0], \urout[e;0]]$, 
\begin{align*}
\left|\int_{\bunder}^{t_0(x)} B(t, \bt) \, dt\right| & \leq C(\bunder) \left|\int_{\urout[e;0]}^{\ur_0(x)} \left((1+ |\utau| r^{-n})|\utau|^{\frac 1{n-1}} - \frac{|\utau|^{\frac 1{n-1}}}{r\log |\tz|} \right) dr\right| \\
& \leq C(\bunder) |\utau|^{\frac 1{n-1}}.
\end{align*}
It follows that for $x = (t',\bt)$, 
\[
f(x) = \exp\left(\int_{(\bunder}^{t'} A(t) dt\right)\left( f(\bunder,\bt) + \int_{\bunder}^{t'} B(t) dt\right).
\]Thus, as long as $|f|< \frac 1{10}$, the previous estimates imply that
\[
|f(x)| \leq C(\bunder, n) |\utau|^{\frac 1{2(n-1)}}.
\]Since the result holds for $x=(\bunder, \bt)$, it will continue to hold for all $x\in \Lambda\pen$. This implies the $C^0$ equivalence.

\end{proof}

Because the problem proves more tractable when considering norms that allow for the natural scaling, we will record the initial error estimates with respect to this scaling. 

\begin{definition}\label{rhodef} 
We define the function $\rho_d: M \to \Real$ such that
\begin{equation}\label{rhoeq}
\rho_d(x) = \left\{\begin{array}{ll}
1 & \text{if } x \in M[p], p \in V(\Gamma)\\
(\psi[a+4, \bunder](t)\cdot \psi[\RH-(a+4),\RH-\bunder](t)\cdot \ur_d(x) &\\
\indent + \psi[\bunder,a+4](t)+\psi[\RH-\bunder,\RH-(a+4)](t)& \text{if } x=(t, \bt) \in M[e] , e \in E(\Gamma) \\
\psi[a+4, \bunder](t)\cdot \ur_d(x)+ \psi[\bunder,a+4](t)& \text{if } x=(t, \bt) \in M[e], e \in R(\Gamma)
\end{array}\right. 
\end{equation}
\end{definition}

Observe that $\rho_d$ is a smooth function that behaves like $\ur_{d}$ on each $M[e]$ and is $1$ on each central sphere. The cutoff function smooths out the transition between them.

Because the error was previously determined in the standard H\"older norm, we now record the error induced by gluing and dislocation in the chosen weighted metric.
\begin{prop}\label{Hestimates} 
\begin{itemize}
\item $\|H_{\mathrm{gluing}}[\dzeta]:{C^{0,\beta}}(M,\rho_d, g)\| \leq C(\beta,\bunder)| \utau|$. 
\item $\|H_{\mathrm{dislocation}}[\dzeta]:{C^{0,\beta}}(M,\rho_d, g)\| \leq C(\beta, \bunder)|{\boldsymbol \zeta}| \leq C(\beta, \bunder) \underline C | \utau|$.
\end{itemize}
\end{prop}
\begin{proof}
First recall that $H_{\mathrm{gluing}}[\dzeta], H_{\mathrm{dislocation}}[\dzeta]$ are supported on $\cup_pS[p]$. Thus, for all $x \in \supp(H_{\mathrm{gluing}} [\dzeta]\cup H_{\mathrm{dislocation}}[\dzeta])$, $ \rho_d \sim_{ C(\bunder)}1$. The bounds then follow immediately from \ref{Hbounds}. 
\end{proof}

\begin{prop} 
\label{find:c1}
There exists $c_1(\bunder,k,n)>0$ such that
\begin{equation}\label{rhoest}
\|\rho_d^{\pm 1}:C^{k}(M,\rho_d,g, \rho_d^{\pm1})\| \leq c_1(\bunder,k,n).
\end{equation}
\end{prop}
\begin{proof}
First note that the uniform geometry of $\Omega = S_{\bunder}[p], S^+_{\bunder}\pen, S^-_{\bunder}\pen$ in the $g$ metric immediately implies the estimate
 \[
\|\rho_d^{\pm 1}:C^{k}(\Omega,\rho_d,g, \rho_d^{\pm1})\| \leq C(\bunder,k).
\]Now for any fixed $x\in \Lambda\pen$, consider the function $\hat \rho(y):= \rho_d(y)/ \rho_d(x)$. Then,
\[
\hat \rho(y) = e^{ w_{\td}(\tsd(y))-w_{\td}(\tsd(x))}.
\]Because of the local nature of the estimate and since $\frac {d}{dt_d}w_{\td}\circ \tsd \in [1-3\delta, 1]$, we are interested in $y$ such that $|\tsd(y) - \tsd(x)|  \leq \frac 15$. (Recall the proof of \eqref{r_metric_equiv} and note that $|\tsd(y)-t_0(y)| \leq C|\utau|^{\frac 1{n-1}}$ by \ref{periodasymptotics} and \eqref{diffeodifference}.) Thus
\[
| w_{\td}(\tsd(y))-w_{\td}(\tsd(x))| = \left| \int_{\tsd(x)}^{\tsd(y)} \frac{dw_{\td}}{d\tsd}\right| \leq | \tsd(y)-\tsd(x)| \leq \frac 15 .
\]This implies the $C^0$ estimates. Now note that
\[
\frac{\partial}{\partial \tsd} \hat \rho= \frac{d}{d\tsd}w_{\td} \cdot \hat \rho, \quad \quad \frac{\partial^2}{\partial \tsd^2} \hat \rho = \frac{d^2}{d\tsd^2}w_{\td}\cdot \hat \rho+\left ( \frac{d}{d\tsd}w_{\td}\right)^2 \hat \rho.
\]Using the estimates in Appendix \ref{DelSection}, we recall that 
\[
-\frac{d^2}{d\tsd^2}w_{\td}= \ur_{d}^2 + (2-n)\td \ur_{d}^{2-n} + (1-n) \td^2 \ur_{d}^{2-2n}.
\] Moreover, $\ur_{d} \geq r_{\td}^{\min} \geq C(\bunder)|\utau|^{\frac 1{n-1}} + O(|\utau|^{\frac 2{n-1}})$. Taken together we see that on $\Lambda\pen$, 
\[
\left|(2-n)\td \ur_{d}^{2-n} + (1-n) \td^2 \ur_{d}^{2-2n}\right| \leq C(\bunder,n).
\] It follows from the previous analysis and these new estimates that
\[
\|\rho_d^{\pm 1}:C^2(M,\rho_d,g, \rho_d^{\pm 1})\|\leq C(\bunder, n).
\]As $w_{\td}$ satisfies a second order ODE, any higher derivatives in $w_{\td}$ can be written in terms of the function and its first and second derivatives. Since $\frac{\partial}{\partial \tsd^k} \hat \rho$ can be written in terms of $\frac{\partial}{\partial \tsd^{m}}w_{\td}$ and $\hat \rho$ where $m = 0, \dots, k-2$, the uniform bounds for $\rho_d$ in $C^k$ follow immediately. For $\rho_d^{-1}$ we only need to note that the denominator will contain the power $\hat \rho^{-k-1}$ which will also be controlled in terms of some constant $C(k)$.
\end{proof}

\subsection*{Solving the semi-linear problem on $\widetilde S[p]$}
The goal of this subsection is to prove \ref{linearpartp}. 
We wish to solve a linearized problem with zero boundary data and fast decay toward all boundary components. 
These requirements force us to proceed as in the lower dimensional version \cite{BKLD} and introduce the extended substitute kernel. 
We prove that a modified version of the linearized problem is solvable by what are now standard methods (see for example \cite{HaskKap}).

We introduce maps $\widetilde Y[p]$ on $\widetilde S[p]$ which are useful parametrizations of $\Ss^{n}$. 
A comparison between these maps and $\hYtdz$ will help us understand the possible obstructions 
to solving the linearized problem on these central spheres by considering the linearized operator 
in the induced metric of $\widetilde Y$ (which corresponds to the metric on the round sphere).
\begin{definition}
 Let $\widetilde Y[p]:\widetilde S[p] \to \mathbb S^n\subset \Rn$ such that (recall \ref{ReDef})
\begin{equation}\label{tildeYp}
\widetilde Y[p](x)=\left\{ \begin{array}{ll}
  \hat Y[W,W](x) &\text{if }x \in M[p],\\
\RRR[e]\circ Y_0(x)& \text{if } p=p^+[e], x \in M[e]\cap\left([a,\Pe-\bunder]\times \Ssn\right), \\
\RRR[e]\circ Y_0 (t-\RH,\bt )&\text{if } p=p^-[e], x=(t,\bt), \\
&\:\:x \in M[e]\cap\left( [\RH-(\Pe-\bunder),\RH-a]\times \Ssn\right).
 \end{array}\right.
\end{equation}
Let $\widetilde g:=  \widetilde Y[p]^*( g_{\Real^{n+1}})$.
\end{definition}

\begin{prop} 
\label{geolimit}
Let $p \in V(\Gamma)$ and let $p'$ be the corresponding vertex in the graph $\Gtdtl$. Then
 \[
 \|(\hYtdz-p') - \widetilde Y[p]:C^k(S_x[p],\widetilde g)\|\leq C(k,x) \left(|\utau|^{\frac 1{n-1}}+ |\boldsymbol \zeta|\right) \leq C(k,x) |\utau|^{\frac 1{n-1}}.
\]
\end{prop}

\begin{proof}
Note that $(\widetilde Y[p] + p) \big|_{M[p]}= Y_{0,\mathbf 0}\big|_{M[p]}$. Since the mapping $\hYtdz|_{M[p]}$ depends smoothly on $\dz$ and approaches $Y_{0, \mathbf 0}$ as $|\utau| \to 0$, the result on $M[p]$ follows immediately. 

Recall \ref{defn:Yedge}, \ref{defn:RT}, \ref{tddefn}. For  $e \in E(\Gamma)$, $x \in M[e]\cap \widetilde S[p^+[e]]$, we determine the $C^k$ estimate by  considering the norms of the two immersions
\begin{align*}
&\RRR[e;\dz]\circ Y_1(y):=\RRR[e;\dz]\circ \left(Y_{\mathrm{edge}}[\taue, l, {\boldsymbol \zeta}\ppe, {\boldsymbol \zeta}\pme] (\tsd(y),\bt)- Y_0(y)\right),\\
&\left(\RRR[e;\dz] - \RRR[e]\right)\circ Y_0(y)
\end{align*} and applying the triangle inequality.
 
First, \ref{defn:Yedge}, \ref{tddefn}, \eqref{Y0} imply that, for $y=(t,\bt)$,
\[
\|Y_1:C^0(M[e] \cap [a,a+3]\times \Ssn, \widetilde g)\|\leq C(\left| \boldsymbol \zeta \right|+ |\tanh(t)-\tanh(\tsd(y))| + |\sech(t)-\sech(\tsd(y))|).
\]By \eqref{diffeodifference},
\[
 |\tanh(t)-\tanh(\tsd(y))|+|\sech(t)-\sech(\tsd(y))|\leq -C\frac{|\utau|^{\frac 1{n-1}}}{\log(C |\utau|)}.
\]The $C^k$ estimates on this region then follow from the definitions and further applications of \eqref{diffeodifference}. On $[a+4,x]\times \Ssn$, for $y=(t,\bt)$,
\[
Y_1(y):= (k_d(t_d(y)) - \tanh(t), (r_d(t_d(y))-\sech(t) )\bt).
\] \eqref{diffeodifference} and \ref{radiuslemma} then imply the $C^k$ bounds on this region. The immersion on $[a+2,a+5]\times \Ssn$ is simply a smooth transition between the immersions at $t=a+2$ and $t=a+5$. Thus, the $C^k$ estimates hold here since the transition function and its derivatives are well controlled.

For the second immersion, note that $\|Y_0:C^k(S_x[p],\widetilde g)\| \leq C(x,k)$. Moreover, by\eqref{zetaframebound},  
the smooth dependence of $\calF$ on $\tilde d,\tilde  \ell$, and 
\begin{align*}
\left|(\RRR[e;\dz] - \RRR[e] )\Be_i \right|&\leq \left|\RRR[e;\dz](\Be_i-\Be_i[e]) \right|+\left| \RRR[e;\dz]\Be_i[e] - \RRR[e]\Be_i\right|\\
& = \left|(I - \RRR[\Be_1,\Be_1[e]])\Be_i\right|+ \left|\Bv_i[e;\tilde d,\tilde\ell] - \Bv_i[e]\right|, \text{ recall } \ref{FrameLemma}\\
& \leq C\left| \boldsymbol \zeta \right|  + C(|\utau|^{\frac 1{n-1}} + |\boldsymbol \zeta|) \text{ by } \eqref{tauratio2}, \eqref{ddifftau}\\
& \leq C\left(|\utau|^{\frac 1{n-1}}+|\boldsymbol \zeta|\right).
\end{align*}

Identical arguments hold for $e \in R(\Gamma)$ and the immersion $Y_{\mathrm{ray}}$ replacing $Y_{\mathrm{edge}}$. When $p=p^-[e]$, the only modification in the proof comes from orientation of $\Bv\pe= -\Bv_1[e]$. The definition of $\widetilde Y$ accounts for that modification by taking $\RRR[e]$ of the reflection of $Y_0(t,\bt)$ about the $\Be_1$ axis.
\end{proof}

We now consider the nature of the approximate kernel of $\mathcal L_g$ on $\widetilde S[p]$. 
By approximate kernel, we mean the span of eigenfunctions of $\mathcal L_g$ with small eigenvalue. 
Following standard methodology 
we will use the methods of Appendix B in \cite{KapAnn} and compare each $\widetilde S[p]$ in the induced metric with an appropriate embedding of the round sphere. 
The maps $\widetilde Y[p]$ will be used for the comparison. 

We also find it helpful to define scaled Jacobi functions, induced by translation vector fields. Notice that for $d=0, \boldsymbol \zeta=0$, these functions behave on $\widetilde Y[p](S[p])$ like an orthonormal basis of eigenfunctions for the lowest two eigenspaces of the operator $\Delta_{\Ss^n}+n$.
\begin{definition}Recalling \ref{Vpdefn},
 let $\hF_i[\dz] : M \to \Real$ for $i=0, \dots, n$ be defined by 
\begin{equation} 
\label{hFdefeq}
 \hF_i[\dz](x) := \frac{ N_{\dz}(x)\cdot \mathbf e_{i+1} }{\| N_{\Ss^n}\cdot \Be_{i+1}\|_{L^2(\Ss^n)}}=\Cn_{n}^{-\frac 12} N_{\dz}(x)\cdot \mathbf e_{i+1}. 
\end{equation}Here $N_{\dz}$ is the unit normal to the immersion $\hYtdz$.
\end{definition}

Before determining the approximate kernel, we prove a technical lemma. This lemma provides suppremum bounds for the eigenfunctions with low eigenvalue.
\begin{lemma}\label{boundedeflemma}
Let $f$ be an eigenfunction for $\mathcal L_g$ on $\widetilde S[p]$ with eigenvalue $0\leq|\lambda| <(4\rout)^{-1}$. Then
\[
\|f:C^{2,\beta}(\widetilde S[p],\rho_d,g)\| \leq C(\beta).
\]
\end{lemma}
\begin{proof}Suppose $\mathcal L_g^\lambda f = 0$ on $\widetilde S[p]$ for some $f$.
We first note that the uniform geometry on $S_5[p]$ (even as $\utau \to 0$) implies the boundedness on $S_1[p]$. Next, we note that on each $\Lambda[p,e,1]$ adjoining $S[p]$ the Dirichlet problem for $\mathcal L_g^\lambda$ has a unique solution. At $\Cout_1[p,e,0]$, decompose $f=f_0+f_1+f_{\mathrm{high}}$ following \ref{f_decomp}. Determine $a_i$ such that $(f_0+f_1)\big|_{\Cout_1}= \sum_{i=0}^n a_i V_i^\lambda[\Lambda[p,e,1],1,0]\big|_{\Cout_1}$. Then the equality holds on all of $\Lambda[p,e,1] \backslash S_1[p]$ and the bound follows by \ref{annulardecaylemma}. For $f_{\mathrm{high}}$ on $\Lambda[p,e,1]$, the bounds follow immediately from the estimates of \ref{linearcor}, with $\Cout$ replace by $\Cout_1$.
\end{proof}

\begin{lemma} 
\label{approxkerprop}
For $\bunder$ as in \ref{ass:tgamma} and $\epsilon>0$ there exists $\maxTG>0$ sufficiently small such that for each $0<|\utau|<\maxTG$:

Let $p \in V(\Gamma)$. 
Then the Dirichlet problem for $\mathcal L_g$ on $\widetilde S[p]$ has exactly $n+1$ eigenvalues in $[-\epsilon, \epsilon]$ and no other eigenvalues in $[-1,1]$. 
There exists a set $\{f_0[p], \dots, f_{n}[p]\}$ that are an orthonormal basis for the \emph{approximate kernel} for $\widetilde S[p]$ such that
$f_i[p] \in C_0^\infty (S[p])$. Moreover, $f_{i}[p]$ depends continuously on all of the 
parameters of the construction and satisfies
\[\|f_{i}[p] - \hF_i[\dz]: C^{2,\beta}(S_5[p],g) \| < \epsilon.
\]
\end{lemma}

\begin{proof}We prove the proposition by some modifications of the results of \cite{KapAnn}, Appendix B, which determine relationships between eigenfunctions and eigenvalues of the Laplace operator for two Riemannian manifolds that are shown to be
close in some reasonable sense. (Throughout the proof, all references to Appendix B or enumerations with B are references to the paper \cite{KapAnn}.) 

In the lower dimensional setting in \cite{BKLD,KapAnn}, 
the linear problem was solved in a conformal metric so that $\mathcal L_h = \Delta_h + c$ for some constant $c$. 
Therefore, it was enough to consider the Laplace operator and compare eigenvalues and eigenfunctions there. 
In this new setting, we are not free to choose such a conformal metric and the potential is not constant in the metric $g$. 
Therefore, we will adapt the ideas of Appendix B in \cite{KapAnn} to include our non-constant potential.

Let $(\widetilde S[p],g)$ and $(\Ss^n,g_{\Ss^n})$ be the two manifolds under consideration. Note that these manifolds satisfy assumptions (1) and (2) of B.1.4. Also note that assumption (3) was needed 
to provide suppremum bounds for the eigenfunctions. 
We observe that for each $\widetilde S[p]$, \ref{boundedeflemma} implies such a bound exists for eigenfunctions with eigenvalue 
$\le (4\max_{e \in E_p}\urout[e;d])^{-1}$.  This bound is $>2(n+1)$ and so it will be sufficient for our purposes.

We follow the convention that $\lambda$ is an eigenvalue for the operator $\mathcal L$ if there exists $f$ such that $\mathcal Lf=-\lambda f$. 
Then the operator $\mathcal L_{g_{\Ss^n}}:=\Delta_{\Ss^n}+n$ has lowest eigenvalues $-n,0,n+2$ and the eigenvalue $0$ has multiplicity $n+1$. Thus, the only eigenvalue for $\mathcal L_{g_{\Ss^n}}$ in the interval $[-1,1]$ is zero with multiplicity $n+1$. Indeed, an orthonormal basis of the kernel of $\mathcal L_{g_{\Ss^n}}$ is given by the functions
 \begin{equation}\label{Sphere_kernel}
 \hat f_{i, \Ss^n}:= \frac{ N_{\Ss^n}\cdot \Be_{i+1}}{\| N_{\Ss^n}\cdot \Be_{i+1}:L^2(\Ss^n)\|}, 
\qquad\text{ for } i=0,\dots, n, 
 \end{equation} 
where $N_{\Ss^n}$ is the inward normal to the unit hypersphere.
 
We now construct two functions $F_1, F_2$ that will satisfy the assumptions B.1.4 and one additional assumption. 
Let $\widetilde Y[p]: \widetilde S[p] \to \Ss^n$ be the function defined in \eqref{tildeYp} and recall that $\widetilde g:= \widetilde Y^*( g_{\mathbb S^n})$. 
Let $\overline \psi:\widetilde S[p] \to [0,1]$ be a smooth cutoff function on $\widetilde S[p]$ such that $\overline \psi \equiv 1$ on $S[p]$ and on each adjacent $\Lambda[p,e,1]$, $\overline \psi \equiv 1$ on $\Lambda[p,e,1] \cap ([\bunder, d_1] \times \Ssn)$, $\overline \psi \equiv 0$ on $\Lambda[p,e,1] \cap ([d_2, \Pe-\bunder]\times \Ssn)$ for $d_1, d_2$ where $d_1< d_2$ are chosen so that $\sech(d_1)=\underline \epsilon$ and $\sech(d_2)=\underline \epsilon/2$ for some $\underline \epsilon>0$ to be determined. If $\overline \psi(t,\Theta):= \frac 2 {\underline \epsilon}\sech(t)-1$ on $[d_1,d_2] \times \Ssn$ then $|\nabla_{\widetilde g} \overline \psi| \leq 4\underline \epsilon^{-1}$ and elsewhere the gradient vanishes. Moreover, by \ref{rratiolemma} and \ref{radiuslemma} $ |\nabla_g \overline \psi|  \sim_{C(\bunder)}  |\nabla_{\widetilde g} \overline \psi| $.

Let $F_1:C^\infty_0(\widetilde S[p]) \to C^\infty_0(\Ss^n)$ such that for $f \in C^\infty_0(\widetilde S[p])$, 
\[
F_1(f) \circ \widetilde Y[p] :=\overline \psi  f .
\] Let $F_2:C^\infty_0(\Ss^n) \to C^\infty_0(\widetilde S[p])$ such that for $f \in C_0^\infty(\Ss^n)$,  
\[
F_2(f):= \overline \psi  f \circ \widetilde Y[p].
\] For any $\epsilon>0$ there exists $\underline \epsilon$ sufficiently small so that the requirements of B.1.6 are met. In addition, we demonstrate that 
\begin{equation} 
\label{A_requirement}
\left|\int_{M_i}|A_i|^2F_i(f)^2 \, dg_i - \int_{M_{i'}}|A_{i'}|^2f^2\, dg_{i'} \right|\leq  \epsilon\|f\|_{\infty}^2 \text{ for } i \neq i'; i,i' \in\{1,2\}.
\end{equation}Here $(M_i,g_i), (M_{i'},g_{i'})$ correspond to the two manifolds and metrics of interest and $A_i, A_{i'}$ correspond to the second fundamental form on the appropriate manifold. We require an estimate like \eqref{A_requirement} since the Rayleigh quotient now includes such a term in the numerator.

We demonstrate \eqref{A_requirement} and a few of the estimates in B.1.6 and leave the rest to the reader as they can be easily verified.
Note that the first inequality in B.1.6 should read
\[
 \|F_if\|_\infty \leq 2 \|f\|_\infty
\]and this is immediately verified by the definitions. Further, since $n \geq 3$, using \ref{geolimit} with $x=d_2$ implies that
\[
\int_{\widetilde S[p]} |\nabla_{ g} \overline \psi|^2 \, dg \leq C(\bunder)(1+ C(d_2)  |\utau|^{\frac 1{n-1}}) \int_{\widetilde S[p]} |\nabla_{\widetilde g} \overline \psi|^2 \, d\widetilde g \leq  C(n,\bunder)(1+ C(d_2)  |\utau|^{\frac 1{n-1}})\underline \epsilon^{n-2}\leq C \underline \epsilon.
\]Again using \ref{geolimit}, by the definition of the $F_i$, 
\begin{equation}\begin{split} 
\label{middlecomparison}
\int_{M[p]} fg \, dg = (1+O( |\utau|^{\frac 1{n-1}}))\int_{\Ss^n \cap \widetilde Y[p](M[p])} F_1(f)F_1(g) \, dg_{\Ss^n},\\
\int_{M[p]} F_2(f)F_2(g) \, dg = (1+O( |\utau|^{\frac 1{n-1}}))\int_{\Ss^n \cap \widetilde Y[p](M[p])} fg \, dg_{\Ss^n}.
\end{split}
\end{equation} 
Therefore, to demonstrate that orthogonality is almost preserved, we only need consider the behavior on each $\Lambda[p,e,1]$. 
In that case, one can verify that
\begin{multline*}
\left|\int_{\Lambda[p,e,1]} F_2(f)F_2(g) \, dg - \int_{\Ss^n \cap \widetilde Y[p](\Lambda[p,e,1])} fg \, dg_{\Ss^n}\right|\leq 
\\ 
\le (1+ O( |\utau|^{\frac 1{n-1}}))\left|\int_{\Lambda[p,e,1]}(f \circ \widetilde Y[p])( g \circ \widetilde Y[p])(1-\overline \psi) \, dg\right| 
\leq C\underline \epsilon^n\|f\|_\infty \|g\|_\infty.
\end{multline*}
Other estimates in B.1.6 proceed similarly.

By \eqref{middlecomparison} and the fact that $|A_g| = n$ on $M[p]$, \eqref{A_requirement} holds on the domain $M[p]$. So we consider each $\Lambda[p,e,1]$. Of critical importance is the fact that while $|A_g|^2$ becomes unbounded as $\utau\to 0$, we may choose $\underline \epsilon$ small enough so that $\int_{[d_1, \Pe-\bunder] \times \Ssn}|A_g|^2 dg$ is as small as we like. To see this, first recall that $|\frac {dw_{\td}}{d\tsd}|\leq 1$ and thus by \eqref{diffeodifference}, $|\frac 1{\ur_{d}}\frac{\partial \ur_{d}}{\partial t}| \leq 1+ |\utau|^{\frac 1{2(n-1)}}$. 
Therefore, we may make the change of variables
\begin{multline*}
\int_{[d_1, \Pe-\bunder]\times \Ssn} |A_g|^2 dg = n\int_{d_1}^{\Pe-\bunder} \int_{\Ssn}( 1+(n-1)\td^2 \ur_{d}^{-2n} )\ur_{d}^n \,dt dg_{\Ssn} \leq 
\\
\leq n(1+|\utau|^{\frac 1{2(n-1)}})\omega_{n-1}\int^{\ur_{d}(d_1)}_{\urin[e;d]}( 1+(n-1)\td^2 r^{-2n} )r^{n-1}\, dr \leq
\\
\leq C(n)\left( \underline \epsilon^n +(n-1)\td^2 \urin[e;d]^{-n} \right) 
\leq C(n)\underline \epsilon^n
\end{multline*}
since $\td^2 \urin[e;d]^{-n}= O(|\utau|^{\frac{n-2}{n-1}})$. 
Therefore, given $\epsilon>0$ we may increase $d_1$ if necessary (decreasing $\underline \epsilon$) so that $\int_{[d_1,\Pe-\bunder]\times \Ssn}|A_g|^2\, dg \leq \epsilon/2$. Thus,
\begin{multline*}
\left|\int_{\Ss^n \cap \widetilde Y[p](\Lambda[p,e,1])}n F_1(f)^2\,dg_{\Ss^n} - \int_{\Lambda[p,e,1]} |A_g|^2 f^2 \, dg\right| \leq 
\\ 
\leq 
O( |\utau|^{\frac 1{n-1}})\int_{[\bunder,d_1]\times \Ssn} f^2 \, dg  
+  \int_{[d_1, \Pe-\bunder] \times \Ssn} |A_g|^2f^2 \, dg 
\leq   \epsilon\|f\|_\infty^2.
\end{multline*}
On the other hand,
\begin{multline*}
\left|\int_{\Lambda[p,e,1]} F_2(f)^2|A_g|^2 \, dg - \int_{\Ss^n \cap \widetilde Y[p](\Lambda[p,e,1])} nf^2 \, dg_{\Ss^n}\right|= 
\\ 
= O( |\utau|^{\frac 1{n-1}})\int_{\Ss^n \cap \widetilde Y[p]([\bunder,d_1]\times \Ssn)} n f^2\, dg_{\Ss^n} + 
\int_{\Ss^n \cap \widetilde Y[p]([d_1, \Pe -\bunder]\times \Ssn)} n f^2\, dg_{\Ss^n} \leq 
\\
\leq C(n) \underline \epsilon^n\|f\|_\infty^2 
\leq \epsilon \|f\|_\infty^2.
\end{multline*}

With the addition of the estimate \eqref{A_requirement}, 
B.2.2 holds for eigenfunctions and eigenvalues of $\mathcal L_g$. 
Perhaps the most difficult estimate to confirm in this new setting is B.2.2 (4). 
Using the Rayleigh quotient, we have that (for $\| \cdot \|^2$ signifying the squared $L^2$ norm)
 \[
 \| \, |A|\, f'_n\|^2 + \|f'_n\|^2 \geq \frac {\delta - C \epsilon}{\lambda_{2,n}+ \delta}.
 \]When $|A|^2 \equiv n$ (i.e. the manifold is $\Ss^n$), we immediately get the required lower bound on $\|f'_n\|^2$. On the other hand, if the manifold is $\widetilde S$, for each $\Lambda = \Lambda[p,e,1]$, 
 \[
 \int_\Lambda |A_g|^2 f^2 \, dg \leq C(n) \int_\Lambda f^2 \, dg + \epsilon\|f\|_\infty^2.
 \]Therefore, the lower bound holds in this case as well. All other applications of the Rayleigh quotient to the proof in B.2.2 are more obvious and do not need the small integrability condition of $|A_g|^2$ along $\Lambda$. 
 
Now we may apply the results of Appendix B, appropriately modified, to find an orthonormal basis of $n+1$ eigenfunctions on $\widetilde S$ with small eigenvalue that are $L^2$ close to those described in \eqref{Sphere_kernel}. 
We get the desired $C^{2,\beta}$ estimate in the following manner. Because of the uniform geometry of $S_6$ for $\maxTG$ sufficiently small, we can use standard linear theory on the interior to increase the $L^2$ norms of Appendix B to $C^{2,\beta}$ norms. Moreover, on $S_{d_1}$, 
\[F_2(\hat f_{i,\Ss^n}) =\hat f_{i,\Ss^n} \circ \widetilde Y.
\] By the definition of the immersions and \ref{geopropcentral} and \ref{geolimit}, for any $\epsilon>0$ we can choose $\maxTG$ sufficiently small so that
\[
\|F_2(\hat f_{i,\Ss^n})-\hF_i[\dz]:C^{2,\beta}( S_6, g)\| < \epsilon/2.
\]To make the dependence continuous, we let $f_{i}$ denote the normalized $L^2(\widetilde S, g)$ projection of $\hF_i[\dz]$ onto the span of $F_2(\hat f_{i, \Ss^n})$. 
\end{proof}

Following the general methodology, we introduce the extended substitute kernel. Notice that we have already solved the semi-local linearized problem everywhere except $\widetilde S[p]$. Thus, the extended substitute kernel is a much smaller space of functions than for previous constructions of similar type.

Let $p \in V(\Gamma)$. 
We fix $\Be'_i[p]$ for $i=1, \dots, n+1$ depending only on $\Gamma$ and such that 
$$
|\Be'_i[p]-\Be_i| < 54 \delta', 
\qquad\qquad 
\text{ and }  
\forall e\in E_p(\Gamma) ,
\qquad 
|\Be'_i[p]- \mathbf{v}\pe | >9\delta',
$$
which by the smallness of the parameters implies that $\Be'_i[p]\in S[p]\subset M$. 
We have then 

\begin{definition}[The substitute {kernel $\calK[p]$}]   
\label{D:calKp} 
We define 
$\calK[p]$ 
to be the span of  
(recall \ref{hFdefeq})  
$\left\{ \, \psi[2\delta',\delta'] \circ \dbold^{M,g}_{\Be'_1[p],\dots,\Be'_{n+1}[p]} \,  \hF_i[\dz] \, \right\}_{i=0}^n \subset C^\infty(M) $.    
We also define a basis $\{w_i[p]\}_{i=0}^n$ of $\calK[p]$ by 
\begin{equation} 
\label{wFdefeq}
 \int_{S[p]} w_i[p] \hF_j[\dz] \; dg =  \delta_{ij}.
\end{equation} 
\end{definition}

\begin{lemma}
\label{subskernellemma} 
\label{extsubslemma}
For each $p$, the following hold:
\begin{enumerate}
\item $w_i[p]$ is supported on $S[p]$.
\item $\|w_i[p]: C^{2, \beta}(M, g)\| \leq C$.
\item For $E \in C^{0,\beta}(\widetilde S[p],g)$ there is a unique $\widetilde w \in \mathcal{K}[p]$ such that $E+\widetilde w$ is $L^2(\widetilde S[p],g)$ orthogonal to the approximate kernel on $\widetilde S[p]$. Moreover, if $E$ is supported on $S_1[p]$, then
\[
\|\widetilde w:C^{2,\beta}(M,g)\|\leq C(\beta)\|E:C^{0,\beta}( S_1[p],g)\|.
\]
\end{enumerate}

\end{lemma}

\begin{definition}[The extended substitute {kernel $\calK\pe$}]  
\label{D:wi} 
For $i=0,\dots, n$, $\pe \in A(\Gamma)$, let $v_i\pe:M \to \Real$ such that (recall \ref{annulardecaydef})
\[v_{i}\pe(x):= \left\{\begin{array}{ll} \widetilde c_i\pe V_i[\Lambda [p,e,1],1,0](x) \psi[\bunder, \bunder+1]\circ \tsd(x)  ,& x\in \Lambda[p,e,1]\\
0, & x \in M \backslash \Lambda[p,e,1]
\end{array}\right.\]
where the $\widetilde c_i\pe$ are normalized constants so that (recall \ref{phidef})
\begin{equation} 
\label{tildecs2}
\widetilde c_i\pe V_i[\Lambda[p,e,1],1,0]= \phi_i
\end{equation} 
on $\Cout_1[p,e,0]$.
For each $i=0, \dots, n$ define $w_i\pe: M \to \Real$ such that
$
w_i\pe:= \mathcal L_g v_i\pe.
$
For each $\pe \in A(\Gamma)$, let $\mathcal{K}\pe = \langle w_0\pe, \dots, w_{n}\pe \rangle_{\mathbb{R}}$.
\end{definition}

\begin{remark}Note that the $\widetilde c_i\pe$ depend smoothly on $d$. 
Moreover, by \ref{ass:b}, \ref{annulardecaydef}, and \ref{annulardecaylemma}, 
\begin{equation}\label{wide_c_bound}
|\widetilde c_i\pe |\sim_{C(\bunder)} 1.
\end{equation}
\end{remark}

\begin{definition}[The global extended substitute {kernel $\calK$}]  
\label{D:K} 
We define the extended substitute kernel 
\[
 \mathcal K := \mathcal K_V \oplus \mathcal K_A, \qquad \text{ where } \qquad 
\calK_V:=  \bigoplus_{p \in V(\Gamma)} \mathcal K[p] , \quad  
\calK_A:=  \bigoplus_{\pe\in A(\Gamma)} \mathcal K\pe. 
\] 
\end{definition}

We now demonstrate how to solve a modified linear problem on $\widetilde S[p]$ with good estimates.

\begin{lemma}\label{linearpartp} 
Let
$\mathcal K_A[p] :=  \bigoplus_{e\in E_p} \mathcal K\pe$. 
Given $ \beta \in (0,1), \gamma \in(1,2)$, there is a linear map
\[
\mathcal R_{\widetilde S[p]}:\{E \in C^{0,\beta}(\widetilde S[p]):E \text{ is supported on } S_1[p]\} \to 
C^{2,\beta}(\widetilde S[p]) \oplus \mathcal K[p]\oplus \mathcal K_A[p],  
\] 
such that the following hold for $E$ in the domain of $\mathcal R_{\widetilde S[p]}$ above and $(\varphi,w_v,w_a)=\mathcal R_{\widetilde S[p]}(E)$:
\begin{enumerate}
\item $\mathcal{L}_g \varphi = E + w_v+w_a$ on $\widetilde S[p]$.
\item $\varphi$ vanishes on $\partial \widetilde S[p]$.
\item $\|w_v,w_a:C^{2,\beta}( S[p],g)\|+ \|\varphi: C^{2, \beta}(\widetilde S[p], \rho_d,g)\| \leq C( \beta)\|E: C^{0,\beta}(S_1[p], \rho_d,g)\|$.
\item $ \|\varphi:C^{2,\beta}(\Lambda[p,e,1], \ur_{d},g , \ur_{d}^{\gamma})\| \leq C( \bunder,\beta, \gamma)\|E:C^{0,\beta}(S_1[p], \rho_d,g)\|$ for all $e \in E_p$. 
\item $\mathcal{R}_{\widetilde S[p]}$ depends continuously on $\dz$.
\end{enumerate}
\end{lemma}
\begin{proof} \ref{subskernellemma} and classical theory together imply there exists $w_v \in \mathcal K[p]$ and $\widetilde\varphi \in C^{2,\beta}(\widetilde S[p])$ such that $\mathcal L_g \widetilde\varphi = E +  w_v$ and $\widetilde\varphi|_{\partial \widetilde S[p]}=0$. For each $\Lambda[p, e,1]$, $e \in E_p$, we modify $\widetilde\varphi$ using the elements $v_i[p,e]$. Let $\widetilde \varphi_e^T$ denote the projection of $\widetilde\varphi$ onto $\mathcal H_1[\Cout_1[p,e,0]]$. Let $\widetilde\varphi_e^\perp = \widetilde\varphi_e-\widetilde\varphi_e^T$ on $\Cout_1[p,e,0]$ and let $V_{\widetilde\varphi_e} := \mathcal R_\partial^{\mathrm{out}} (\widetilde\varphi_e^\perp|_{\Cout_1[p,e,0]})$.  Notice that $(\widetilde\varphi - V_{\widetilde\varphi_e})|_{\Cout_1[p,e,0]} \in \mathcal H_1[\Cout_1[p,e,0]]$ and we denote 
\[
(\widetilde\varphi - V_{\widetilde\varphi_e})|_{\Cout_1[p,e,0]} = \sum_{i=0}^{n}\alpha_i \phi_{i}|_{\Cout_1[p,e,0]}= \sum_{i=0}^{n} \alpha_i v_i[p,e]|_{\Cout_1[p,e,0]}.
\]
Standard theory implies that $\| \widetilde\varphi:C^{2,\beta}(S_2,\rho_d, g)\|\leq C(\bunder,\beta) \|E:C^{0,\beta}(S_1,\rho_d,g)\|$. Coupling this with \ref{linearcor} implies that $|\alpha_i| \leq C(\bunder,\beta) \|E:C^{0,\beta}(S_1,\rho_d,g)\|$. Set 
\[
\varphi = \widetilde \varphi - \sum_{e \in E_p}\sum_{i=0}^{n} \alpha_i v_i[p,e], \qquad w_a = - \mathcal L_g\left(\sum_{e \in E_p}\sum_{i=0}^{n} \alpha_i v_i[p,e]\right).
\]
By construction $\varphi_e^\perp = \varphi_e$ and thus on each $\Lambda[p,e,1]$, $\mathcal R_\partial^{\mathrm{out}} ( \varphi_e^\perp) = \varphi$.  \ref{linearcor} then provides the necessary decay. 
\end{proof}

\subsection*{Solving the linearized equation globally}
We will solve the global problem in a manner analogous to \cite{BKLD}. 
 The hypotheses of the semi-local lemmas require that the inhomogeneous term on each extended standard region is supported on the enclosed standard region. Thus, to solve the global problem we will first use a partition of unity defined on $M$ to allow us to consider the inhomogeneous problem on separate regions that allow for solvability and good estimates. After solving on each region separately, we patch the solutions back together. Obviously, the partitioning and patching introduces error. We demonstrate that the error estimates are sufficiently small to iterate away.

We first introduce the cutoff functions we require. 
\begin{definition} For $d$ satisfying \eqref{drestriction}, we define uniquely smooth functions
$\psi_{S[p]}[d]$, $\psi_{\widetilde S[p]}[d]$, $\psi_{S\pen}[d]$, $ \psi_{\widetilde S\pen}[d]$, $\psi_{\Lambda[p,e,m']}[d]$ such that
\begin{enumerate}[(i)]
\item $\psi_{S[p]}[d] = \psi_{\widetilde S[p]}[d] \equiv 1$ on $S[p]$, 
\[\psi_{S[p]}[d]=\left\{\begin{array}{ll}
\psi[\bunder+1, \bunder]\circ \tsd\big|_{M[e]}& \text{if } p=p^+[e]\\
\psi[\RH-(\bunder+1), \RH-\bunder]\circ \tsd\big|_{M[e]}& \text{if }p=p^-[e],
\end{array}\right.
\]
\[\psi_{\widetilde S[p]}[d]=\left\{\begin{array}{ll}
\psi[\Pe-\bunder, \Pe-(\bunder+1)] \circ \tsd\big|_{M[e]}& \text{if } p=p^+[e]\\
\psi[\RH-(\Pe -\bunder), \RH-(\Pe -(\bunder+1))]\circ \tsd\big|_{M[e]}& \text{if }p=p^-[e],
\end{array}\right.
\]and the functions are $0$ elsewhere.
\item 
\[\psi_{\Lambda[p^+[e],e,m']}[d]=\left\{\begin{array}{l}
\psi[(m'-1)\Pe+\bunder,(m'-1)\Pe+( \bunder+1)]\circ \tsd\big|_{M[e]}\cdot
\\ \quad \quad \psi[m' \Pe - \bunder, m'\Pe -(\bunder+1)]\circ \tsd\big|_{M[e]},
\end{array}\right.
\]
\[
\psi_{\Lambda[p^-[e],e,m']}[d]=\left\{\begin{array}{l}
\psi[\RH-(m'\Pe- \bunder), \RH-(m'\Pe-(\bunder+1))]]\circ \tsd\big|_{M[e]}\cdot \\
\quad \quad \psi[\RH - ((m'-1)\Pe + \bunder), \RH - ((m'-1)\Pe + (\bunder +1))]\circ \tsd\big|_{M[e]},
\end{array}\right.
\]and the functions are $0$ elsewhere.
\item For $m<l[e]$, 
$\psi_{S[p,e,m]}[d] = (1-\psi_{\Lambda\pen}[d])(1-\psi_{\Lambda[p,e,m+1]}[d])$ on $S_1\pen$ and is $0$ elsewhere.
\item For $m=l[e]$, $\psi_{S[p,e,l[e]]}[d] = (1-\psi_{\Lambda[p^+[e],e,l[e]}[d])(1-\psi_{\Lambda[p^-[e],e,l[e]]}[d])$ on $S_1\pen$ and is $0$ elsewhere.
\item For $m<l[e]$, $\psi_{\widetilde S[p,e,m]}[d] = \psi_{S\pen}[d] + \psi_{\Lambda[p,e,m]}[d] + \psi_{\Lambda[p,e,m+1]}[d]$.
\item For $m=l[e]$, $\psi_{\widetilde S[p,e,m]}[d] = \psi_{S[p,e,l[e]]} [d]+ \psi_{\Lambda[p^+[e],e,l[e]]} [d]+ \psi_{\Lambda[p^-[e],e,l[e]-1]}[d]$. 
\end{enumerate}
\end{definition}
Observe that $\psi_{S[p]}[d], \psi_{S\pen}[d], \psi_{\Lambda[p,e,m']}[d]$ form a partition of unity on $M$. Also note that each of the functions $\psi_{\widetilde S[p]}[d], \psi_{\widetilde S\pen}[d]$ are identically 1 on almost all of $\widetilde S[p], \widetilde S\pen$, respectively. Near the boundary they transition smoothly to zero. Finally, $\supp(\psi_{S[p]}[d]) \subset S_1[p], \supp (\psi_{S\pen}[d] )\subset S_1\pen$.

We now set the notation for defining a global $C^{2,\beta}$ function by pasting together appropriately cutoff local functions.
\begin{definition}\label{patching}
Let $u[p] \in C^{k, \beta}( \widetilde S[p])$, $u\pen \in C^{k,\beta}(\widetilde S\pen)$, $p \in V(\Gamma), \pen \in \VS(\Gamma)$, be functions that are 
zero in a neighborhood of $\partial \widetilde S[p]$, $\partial \widetilde S\pen$. We define $U=\mathbf U(\{u[p],u\pen\})\in C^{k,\beta}(M)$ to be the unique function such that
\begin{enumerate}[(i)]
\item $U|_{S[p]}=u[p], U|_{S\pen}=u\pen$.
\item $U|_{\Lambda[p,e,1]}=u[p] +u[p,e,1]$.
\item For $m'<l[e]$, $U|_{\Lambda[p,e,m']} = u[p,e,m'-1] + u[p,e,m']$.
\item For $U|_{\Lambda[p^+[e],e,l[e]]} = u[p^+[e],e,l[e]-1] + u[p^+[e],e,l[e]]$ while $U|_{\Lambda[p^-[e],e,l[e]]}=u[p^-[e],e,l[e]-1]+u[p^+[e],e,l[e]]$. 
\end{enumerate}
\end{definition}

Finally, we define the global norm that we will use. In order to close the fixed point argument,  the global norm we define must be uniformly equivalent for all immersions $\hYtdz$ that may arise. After defining the global norm, we establish this equivalence in \ref{metricequivalence}. Before we prove this equivalence and before precisely defining the global norm, we give some indication as to why we choose to define the norm in this particular manner.

Given any $\dz$ satisfying \eqref{drestriction}, \eqref{zetarestriction}, it will be straightforward to show that the semi-local norms we used in the semi-local settings are uniformly equivalent to the norm given by $d=0,\boldsymbol \zeta= \mathbf 0$. Therefore, on the semi-local level we are free to use those norms already given. On the other hand, the global norm will need to incorporate a decaying weight function. If this weight function is given entirely in terms of $\dz$ then the ratio between two norms for $d\neq 0$ will blow up along a ray of $M$. Therefore, it is convenient to use a decay function that depends upon the immersion given by $d=0, \boldsymbol \zeta =0$. To clearly distinguish the semi-local norms and the decay, we therefore define the global norm by taking the suppremum of semi-local norms on overlapping regions.

\begin{definition}[The global norms]  
\label{metricdefn} 
For $k \in \mathbb N$, $\beta \in (0,1)$, $\gamma \in (1,2)$, and $u\in C^{k,\beta}_{loc}(M)$,  
we define $\|u\|_{k,\beta,\gamma;\dzeta}$ to be the supremum of the following semi-norms (when they are finite) 
\begin{enumerate}[(i)]
\item $\|u:C^{k,\beta}(S_1[p],\rho_d,g)\| $ for each $p \in V(\Gamma)$,  
\item $\de^{-m/2}\|u:C^{k,\beta}(\Sp_1\penp,\rho_d,g)\| $ for each $\pen \in \VSp(\Gamma)$,  
\item $\de^{-(m-1)/2}\|u:C^{k,\beta}( \Smext \penm, \rho_d,g, f_d\,\ur_{d}^{k-2})\|$ for each $\pen \in \VSm(\Gamma)$,  
\end{enumerate}
where 
$f_d:\cup_{\penm}\Smext\penm \to \Real$ is such that, for $m \neq l[e]$,
\begin{equation}\label{fdef}
f_d(x) = \left\{ \begin{array}{ll}
\ur_{d}(x)^\gamma,& x\in \Lambda\penm\\
\urin[e;d]^{\gamma},& x \in \Sm\penm\\
\urin[e;d]^\gamma(\urin[e;d]/\ur_{d}(x))^{n-2+ \gamma}, & x\in \Lambda[p,e,m+1]
\end{array}\right.
\end{equation}
and when $m=l[e]$ define
\begin{equation}\label{fdef2}
f_d(x) = \left\{ \begin{array}{ll}
\ur_{d}(x)^\gamma,& x \in \Lambda[p^+[e],e,l[e]]\\
\urin[e;d]^{\gamma},& x \in \Sm[p,e,l[e]]\\
\ur_{d}(x)^\gamma, & x\in \Lambda[p^-[e],e,l[e]]
\end{array}\right.
\end{equation}
Also, observe that
\begin{equation}\label{ddef}
\de:= \urin[e;0]^{2\gamma+n-2} \sim_{C(\bunder)}| \tau_0[e]|^{1+(2\gamma-1)/(n-1)}.
\end{equation}

For $w=w_v+w_a \in \mathcal K$ such that
\[
 w_v=  \sum_{i=0}^{n} \sum_{p \in V(\Gamma)}\mu_i[p] w_i[p], \quad\quad w_a= \sum_{i=0}^{n} \sum_{\pe \in A(\Gamma)} \mu_{i}\pe w_i\pe,
\] we define the norms on $\mathcal K_V, \mathcal K_A$ such that
\begin{equation}
\left|w_v\right|_V^2:= \max_{p \in V(\Gamma) } \left\{\sum_{i=0}^n(\mu_i[p])^2\right\}, \; \:\; \left|w_a\right|_A^2:=\max_{[p,e] \in A(\Gamma) } \left\{{\sum_{i=0}^n(\mu_i[p,e])^2}\right\}.
\end{equation} 
\end{definition}
Notice that since $\rho_d \sim_{C(\bunder)} 1$ on $S_1[p], \Sp_1\penm$, the semi-norms on these regions in the definition above are uniformly equivalent to those taken with respect to the unscaled metric $g$. 
\begin{remark}\label{td_tzcomparison}
The total decay for the function $f_d$ moving through one $\Smext$ region is:
\begin{equation*}
\mathfrak t_d[e]= \urin[e;d]^{2\gamma + n-2} \sim_{C(\bunder)} |\td|^{1+(2\gamma-1)/(n-1)}.
\end{equation*} 
The global decay factor given by $\de$ does not correspond to this value. However, since
\[
\frac{\mathfrak t_d[e]}{\mathfrak t_0[e]} \sim_{ C(\bunder)} \left(\frac{\td}{\tz}\right)^{1+(2\gamma-1)/(n-1)} \sim_{2} 1,
\]where the second relation follows by \eqref{tauratio}, we have that
\[
\de \sim_{C(\bunder)} \mathfrak t_d[e]
\] on each $\Smext\penm$.
\end{remark}
\begin{remark}
While the definition of norms for $w$ might appear unnatural, the choice is motivated by the nature of the construction. Because the functions $w_i[p], w_i[p,e_j]$ have uniform $C^{2,\beta}$ bounds and are supported on $S_1[p]$,
\[
\|\cdot\|_{2,\beta,\gamma;\dzeta}\sim_{C(\bunder)} \left|\cdot\right|_V, \left|\cdot\right|_A
\] for elements of $\mathcal K_V, \mathcal K_A$.
\end{remark}
\begin{remark}
Observe that the exponent of $\de$ chosen in the definition corresponds -- in absolute value -- to the number of extended standard regions $\Smext$ between the region on which the norm is being determined and the closest central sphere. (Recall \ref{pendef}.)
\end{remark}

\begin{lemma} 
\label{metricequivalence}
There exists $\widetilde C(\bunder)>0$, independent of $\maxTG$, such that if $\maxTG$ is sufficiently small then for any $0<|\utau|< \maxTG$ and $\dz$ satisfying \eqref{drestriction}, \eqref{zetarestriction} the following holds:

If $u : M \to \Real$ such that $ \|u\|_{2,\beta,\gamma;\dzeta} <\infty$, then
\begin{equation}\label{Meq}
\|u\|_{2,\beta,\gamma;\dzeta} \sim_{\widetilde C(\bunder)}\|u\|_{2,\beta,\gamma;{0,\boldsymbol 0}} .
\end{equation}
\end{lemma}
 
\begin{proof}
The definition of the global norm allows us to consider the equivalence on the semi-local norms. 

Let $g_0:=Y_{0, \mathbf 0}^*(g_{\mathbb R^{n+1}})$. By \ref{geolimit}
\[
\|(Y_{0,\mathbf 0}-p)-(\hYtdz-p'):C^k(S_{\bunder}[p],\rho_0,g_{0})\| \leq C(k,\bunder) |\utau|^{\frac 1{n-1}}
\]
for each $p \in V(\Gamma)$ and corresponding $p' \in V(\Gtdtl)$. Thus, 
\[
\|u:C^{k,\beta}(S_{\bunder}[p],\rho_d,g)\| \sim_{ C(k, \bunder)} \|u:C^{k,\beta}(S_{\bunder}[p],\rho_0,g_0)\|
\]

Now consider the comparison for each $e \in E(\Gamma)$ and $x \in M[e] \cap [\bunder, \RH - \bunder] \times \Ssn$. On these regions, we consider $C^{k,\beta}$ norms on balls of radius $1/10$ with respect to the metrics
\[
\rho_0(x)^{-2}g_0= \left(\frac{\ur_0}{\ur_0(x)}\right)^2(dt_0^2 + g_{\Ssn}), \quad \quad \rho_d(x)^{-2}g =  \left(\frac{\ur_{d}}{\ur_{d}( x)}\right)^2\left({d\tsd^2} + g_{\Ssn}\right).
\] 
The equivalence of the weights and the metrics is immediately given by \ref{diffeodifference} and \ref{rratiolemma}. 
The argument for $e \in R(\Gamma)$ is identical so the proof is complete.
\end{proof}

We are now ready to state and prove the main proposition of this section. 
The strategy is as follows. 
We first presume that the inhomogeneous term is supported on the standard regions. 
Using the semi-local lemmas, we solve the problem on each extended standard region. 
We then patch together cutoffs of these semi-local solutions, which introduces error that can be removed by iteration. 

For the more general case, we first partition the inhomogeneous term and use \ref{RLambda} to solve the problem on each $\Lambda$. 
We then show that the problem remaining can be reduced to the first case.
\begin{prop}\label{LinearSectionProp}For each $\dz$ satisfying \eqref{drestriction}, \eqref{zetarestriction} respectively, 
there exists a linear map $\mathcal{R}_{\dz}: C^{0, \beta}(M) \to C^{2,\beta}(M)\oplus \mathcal{K}_V \oplus \mathcal K_A$ such that for $E \in C^{0,\beta}(M)$ and $(\varphi, w_v,w_a)=\mathcal{R}_{\dz}(E)$ the following hold:
\begin{enumerate}
\item $\mathcal{L}_g \varphi = E+ w_v+w_a$ on $M$.
\item $\|\varphi\|_{2,\beta,\gamma;\dz}+\left|w_v\right|_{V}+\left|w_a\right|_A\leq C( \beta,\gamma)\|E\|_{0, \beta, \gamma;\dz}$
\item $\mathcal{R}_{\dz}$ depends continuously on $\dz$.
\end{enumerate}
\end{prop}
\begin{proof}
We begin by presuming that $\supp(E) \subset\left( \cup_{V(\Gamma)} S_1[p] \bigcup \cup_{\VS(\Gamma)} S_1\pen\right)$. Let $\varphi\pen:=\mathcal R_{\widetilde S\pen}(E|_{S_1\pen})$ where $\mathcal R_{\widetilde S\pen}$ denotes the linear map from \ref{linearpluslemma} or \ref{linearminuslemma} as appropriate. 
 We will directly apply the results of Section \ref{DelaunayLinear} using the decay and metric dilation in terms of $\ur_d$ rather than $r$ to account for the coordinate change induced by the map $\tsd$. 

Let $(\varphi[p],  w_v[p],  w_a[p]):= \mathcal R_{S[p]}(E|_{S_1[p]})$, defined by \ref{linearpartp}. Let 
\[
\mathcal R^0 E:= \mathbf U\left( \{ \psi_{\widetilde S[p]}[d]\varphi[p], \psi_{\widetilde S\pen}[d]\varphi\pen\}\right) \in C^{2,\beta}(M),
\] 
\[
\mathcal W_v^0E:= \sum_{p \in V(\Gamma)}  w_v[p] \in \mathcal K_V,
\]
\[
\mathcal W_a^0E:= \sum_{p \in V(\Gamma)}  w_a[p] \in \mathcal K_A,
\]\[
\mathcal{E}E:=\mathbf U(\{[[\psi_{\widetilde S[p]}[d],\mathcal L_g]]\varphi[p], [[\psi_{\widetilde S\pen}[d],\mathcal L_g]]\varphi\pen\})\in C^{0,\beta}(M),
\]where here $[[\cdot,\cdot]]$ denotes the commutator. That is, $[[\psi_{\widetilde S[p]}[d],\mathcal L_g]]\varphi[p]= \psi_{\widetilde S[p]}[d]\mathcal L_ g \varphi[p] -\mathcal L_g (\psi_{\widetilde S[p]}[d]\varphi[p])$ and the like for $\varphi\pen$.

One easily checks that, as the support of $E$ implies $\psi_{\widetilde S[p]}[d] \mathcal L_g \varphi[p]= \mathcal L_g \varphi[p]$ and the like for $\varphi\pen$, 
\begin{equation}
\mathcal L_g \mathcal{R}^0E + \mathcal{E}E = E+ \mathcal{W}_v^0E+ \mathcal{W}_a^0E \qquad \text{on } M.
\end{equation}Notice that by construction, on regions where $\psi_{\widetilde S[\cdot]}[d]$ is not constant, $|\partial_{\tsd}^k \psi_{\widetilde S[\cdot]}[d]| \leq C$. We will use frequently without repeated reference the fact that
\begin{align*}
\|\varphi \psi:C^{2,\beta}(\Omega \cap \supp(|\nabla \psi|), \rho_d,g,f_d) \|&\leq  C\|\psi:C^3(\Omega \cap \supp(|\nabla \psi|), \rho_d,g)\| \|\varphi:C^{2,\beta}(\Omega, \rho_d, g, f_d)\| \\
&\leq C \|\varphi:C^{2,\beta}(\Omega, \rho_d, g, f_d)\|.
\end{align*}

Using the estimates of \ref{linearpluslemma}, \ref{linearminuslemma}, \ref{linearpartp} and the inequalities above, we quickly verify that\begin{align*}
\|\mathcal R^0E\|_{2,\beta,\gamma;\dz} &\leq C(\beta, \gamma)\|E\|_{0, \beta,\gamma;\dz},\\
\|\mathcal W_v^0E\|_{2,\beta,\gamma;\dz} &\leq C(\beta, \gamma)\|E\|_{0, \beta,\gamma;\dz},\\
\|\mathcal W_a^0E\|_{2,\beta,\gamma;\dz}& \leq C(\beta,\gamma)\|E\|_{0, \beta,\gamma;\dz}.
\end{align*}
The only tedious calculation to verify is the first one. But this is still quite straightforward as, by construction, while $\varphi\pen$ is allowed to grow in the direction of the nearest central sphere, the estimate on the rate of growth is much smaller than the growth allowed in that direction by the definition of the global norm. 

We can finish this step of the proof by iteration, once we determine the estimate
\begin{equation}\label{iterationest}
\|\mathcal EE\|_{0,\beta,\gamma;\dz} \leq \frac 12\|E\|_{0, \beta,\gamma;\dz}.
\end{equation}
To prove this, first note that $\supp(\mathcal EE) \subset \left( \left(\cup_{p } S_1[p] \backslash S[p]\right) \bigcup\left( \cup_{\pen}S_1\pen \backslash S\pen\right)\right)$. We now consider the estimates. On any $S_1[p]$, 
\begin{align*}
\|\mathcal EE:C^{0,\beta}(S_1[p],\rho_d,g)\| &\leq C(\bunder,\beta,\gamma)\max_{e_j \in E_p}\|\varphi[p,e_j,1]:C^{2,\beta}(S_1[p],\rho_d,g)\|\\
& \leq C(\bunder,\beta,\gamma)\max_{e_j \in E_p}\urin[e_j,d]^{-1}\|E:C^{0,\beta}(S_1[p,e_j,1],\rho_d,g,\ur_{d}^{-2})\| \\
& \qquad \qquad \qquad\qquad\text{ by } \ref{linearminuslemma} \, (4)\\
& \leq C(\bunder,\beta,\gamma)\max_{e_j \in E_p}\urin[e_j,d]^{\gamma-1}\|E\|_{0,\beta,\gamma;\dz} \text{ by } \eqref{fdef}.
\end{align*}
Now consider the estimates on $S_1\pen$ for $m$ odd (a catenoidal type region). First note that for $\gamma' \in (\gamma, 2)$ one can apply \ref{linearpluslemma} (3) to produce the estimate
\[
\|\varphi[p,e,m-1]:C^{2,\beta}(S_1\pen,\rho_d,g,\ur_d^{\gamma})\| \leq C(\bunder,\beta, \gamma, \gamma')\urin[e;d]^{\gamma'-\gamma}\|E:C^{0,\beta}(S[p,e,m-1],\rho_d,g)\|.
\]Note that we are using the fact that $\mathrm{supp}(\phi[p,e,m-1]) \cap S_1\penm$ is on $\Lf$ with respect to the domain $S[p,e,m-1]$. Also note that if $m=1$ use instead \ref{linearpartp} (4).

Then
\begin{align*}
\de^{-\frac{m-1}{2}}\|\mathcal EE:C^{0,\beta}&(\Lambda[p,e,m] \cap S_1\pen,\rho_d,g,\ur_d^{\gamma-2})\| \\ 
&\leq C(\bunder,\beta,\gamma)\de^{-\frac{m-1}{2}} \|\varphi[p,e,m-1]:C^{2,\beta}(S_1\pen,\rho_d,g,\ur_d^{\gamma})\|\\
&\leq C(\bunder,\beta,\gamma, \gamma')\de^{-\frac{m-1}{2}}\urin[e;d]^{\gamma'-\gamma}\|E:C^{0,\beta}(S[p,e,m-1],\rho_d,g,\ur_d^{-2})\|\\
& \leq C(\bunder,\beta, \gamma, \gamma') \urin[e;d]^{\gamma'-\gamma} \|E\|_{0,\beta,\gamma;\dz}.
\end{align*}On the other adjoining transition region we note that for $m<l[e]$,
\begin{align*}
\de^{-\frac{m-1}{2}}\left\|\mathcal EE:C^{0,\beta}\right.&\left.\left(\Lambda[p,e,m+1] \cap S_1\pen,\rho_d,g,\urin[e;d]^{\gamma}\left(\frac{\urin[e;d]}{\ur_d}\right)^{n-2+\gamma}\ur_d^{-2}\right)\right\| \\ 
&\leq C(\bunder,\beta,\gamma) \de^{-\frac{m-1}{2}}\urin[e;d]^{-\gamma} \cdot \\
& \quad \quad \quad \quad\left\|\varphi[p,e,m+1]:C^{2,\beta}\left(S_1\pen,\rho_d,g,\left(\frac{\urin[e;d]}{\ur_d}\right)^{n-1}\right)\right\|\\
&\leq C(\bunder,\beta,\gamma)\de^{-\frac{m+1}{2}}\de\urin[e;d]^{1-n-\gamma}\|E:C^{0,\beta}(S[p,e,m+1],\rho_d,g)\| \\
& \qquad \qquad \qquad\qquad\text{ by } \ref{linearpluslemma} \, (4)\\
& \leq C(\bunder,\beta, \gamma) \urin[e;d]^{\gamma-1} \|E\|_{0,\beta,\gamma;\dz} \text{ by } \eqref{tauratio}, \eqref{ddef}.
\end{align*}For $m=l[e]$, we just use the previous estimate twice, once on each $\Lambda[p^+[e],e,l[e]]$ and $\Lambda[p^-[e],e,l[e]]$. Given the definition of $f_d$ for $m=l[e]$, this proves the result.

One can perform similar estimates on $S_1\pen$ for $m$ even. Adapting the argument for this setting, we apply the results of \ref{linearminuslemma} for $\gamma' \in (\gamma,2)$ to see 
\begin{align*}
\de^{-\frac {m-2}2} & \left\|\mathcal EE:C^{0,\beta}\left(\Lambda[p,e,m] \cap S_1\pen,\rho_d,g,\urin[e;d]^{\gamma}\left(\frac{\urin[e;d]}{\ur_d}\right)^{n-2+\gamma}\ur_d^{-2}\right)\right\| \\ 
& \leq C(\bunder,\beta,\gamma) \de^{-\frac {m-2}2}\urin[e;d]^{-\gamma} \|\varphi[p,e,m-1]:C^{2,\beta}\left(S_1\pen,\rho_d,g,(\urin[e;d]/\ur_d)^{n-2+\gamma}\right)\|\\
&\leq C(\bunder,\beta,\gamma, \gamma')\de^{-\frac {m-2}2}\urin[e;d]^{\gamma'-\gamma}\|E:C^{0,\beta}(S_1[p,e,m-1],\rho_d,g,\ur_d^{-2}\urin[e;d]^\gamma)\|
\\
& \qquad \qquad \qquad\qquad\text{ by } \ref{linearminuslemma} \, (3)\\
&\leq C(\bunder,\beta,\gamma, \gamma')\urin[e;d]^{\gamma'-\gamma}\|E\|_{0,\beta,\gamma;\dz}.
\end{align*}And finally 
\begin{align*}
\de^{-\frac{m}{2}}\|\mathcal EE:C^{0,\beta}&(\Lambda[p,e,m+1] \cap S_1\pen,\rho_d,g,\ur_d^{\gamma-2})\| \\  
&\leq C(\bunder,\beta,\gamma)\de^{-\frac{m}{2}} \|\varphi[p,e,m+1]:C^{2,\beta}(S_1\pen,\rho_d,g,\ur_d)\|\\
&\leq C(\bunder,\beta,\gamma)\de^{-\frac{m}{2}}\urin[e;d]^{\gamma-1}\|E:C^{0,\beta}(S_1[p,e,m+1],\rho_d,g,\urin[e;d]^\gamma \ur_d^{-2})\|\\
& \qquad \qquad \qquad\qquad\text{ by } \ref{linearminuslemma} \, (4)\\
& \leq C(\bunder,\beta,\gamma)\urin[e;d]^{\gamma-1}\|E\|_{0,\beta,\gamma;\dz}.
\end{align*}For $\maxTG>0$ sufficiently small, 
\[
\max_{e \in E(\Gamma) \cup R(\Gamma)}\left\{C(\bunder,\beta,\gamma) \urin[e;d]^{\gamma -1} +C(\bunder,\beta,\gamma, \gamma')\urin[e;d]^{\gamma'-\gamma}\right\} < \frac 12.
\] This implies \eqref{iterationest}. As $\supp(\mathcal EE) \subset \supp(E)$ we can apply the same procedure and produce $\mathcal R^1E, \mathcal W^1_vE, \mathcal W^1_aE, \mathcal E^1E$ such that
\[
\mathcal L_g \mathcal R^1 E + \mathcal E^1E = \mathcal EE + \mathcal W^1_vE+ \mathcal W^1_aE
\]with 
\[\|\mathcal R^1E\|_{2,\beta,\gamma;\dz}+ \|\mathcal W^1_vE\|_{2,\beta,\gamma; \dz}+ \|\mathcal W^1_aE\|_{2,\beta,\gamma;\dz}+ \|\mathcal E^1E\|_{0,\beta,\gamma;\dz} \leq C(\beta,\gamma) \frac 12 \|E\|_{0,\beta,\gamma;\dz}.
\]We continue by induction and produce, for all $k \in \mathbb Z^+$,
\[
\mathcal L_g \mathcal R^k E + \mathcal E^kE = \mathcal E^{k-1}E + \mathcal W^k_vE+ \mathcal W^k_aE \text{ on } M.
\]Set
\[
\varphi:= \sum_{k=0}^\infty \mathcal R^kE, \quad w_v:= \sum_{k=0}^\infty \mathcal W_v^kE, \quad w_a:= \sum_{k=0}^\infty \mathcal W_a^kE. 
\]The estimates imply that all three of the series converge and we have proven the proposition in the first case.

We now move to the general case. First apply \ref{RLambda} to each $\Lambda[p,e,m']$ such that
\begin{align*}
V^{\mathrm{out}}[p,e,m']&:= \mathcal R_\Lambda^{\mathrm{out}}E|_{\Lambda[p,e,m']} \text{ for }m'\text{ odd}, \\
V^{\mathrm{in}}[p,e,m']&:= \mathcal R_\Lambda^{\mathrm{in}}E|_{\Lambda[p,e,m']} \text{ for }m'\text{ even}.
\end{align*}By the proposition, 
\[
V^{\mathrm{out}}[p,e,m']|_{\Cout[p,e,m']} \in \mathcal H^1[\Cout[p,e,m']], V^{\mathrm{out}}[p,e,m']|_{\Cin[p,e,m']} \equiv 0 \text{ for } m' \text{ odd, while}
\]
\[
V^{\mathrm{in}}[p,e,m']|_{\Cin[p,e,m']} \in \mathcal H^1[\Cin[p,e,m']], V^{\mathrm{in}}[p,e,m']|_{\Cout[p,e,m']} \equiv 0 \text{ for } m' \text{ even}.
\]Let
\[
\widetilde E:= \mathbf U\left(\left\{\psi_{S[p]}[d]E, \psi_{S\pen}[d]E\right\}\right) + \mathbf U\left(\left\{0, [\psi_{\Lambda[p,e,m']}[d],\mathcal L_g]V[p,e,m']\right\}\right) \in C^{0,\beta}(M).
\]Note that $\widetilde E \subset \left(\cup_{V(\Gamma)} S_1[p] \bigcup \cup_{\VS(\Gamma)} S_1\pen\right)$ so we can apply the initial argument of the proof with $\widetilde E$ in place of $E$. Thus there exist $(\widetilde \varphi, w_v,w_a) \in C^{2,\beta}(M) \times \mathcal K_V \times \mathcal K_A$ such that $\mathcal L_g \widetilde \varphi = \widetilde E + w_v+w_a$ on $M$. Set
\[
\varphi := \widetilde \varphi + \mathbf U\left( \left\{0, \psi_{\Lambda[p,e,m']}[d] V[p,e,m']\right\}\right).
\]Then by the definition of the cutoff functions, $\mathcal L_g \varphi = E + w_v+w_a$. Moreover, the estimates from \ref{RLambda} and the work done above imply $(\varphi, w_v,w_a)$ satisfy the necessary estimates.
\end{proof}

\section{The Geometric Principle}
\label{GeometricPrinciple} 

Throughout this section, let $\gamma \in (1,2), \beta \in (0,1), \gamma' \in (\gamma,2),$ and $\beta' \in (0,\beta)$. 
Any constant depending on $\bunder,\beta,\gamma,\beta',\gamma',n$ we simply denote as $C$. 
The goal of this section is to prove \ref{prescribexi}. 

\subsection*{Prescribing the substitute kernel}
Throughout this subsection, we consider ``super-extended" central standard regions. For each $p \in V(\Gamma)$, these regions will be determined by immersions of the domain
\begin{align*} S^+[p]:=M[p]&\bigsqcup_{\{e|p=p^+[e]\}}\left(M[e]\cap[a, \Pe]\times \Ssn\right)\\&\bigsqcup_{\{e|p=p^-[e]\}}\left(M[e] \cap [\RH-\Pe,\RH-a] \times \Ssn\right).
\end{align*}
\begin{lemma}\label{unbalancinglemma}
Let $d, \boldsymbol \zeta$ satisfy \eqref{drestriction}, \eqref{zetarestriction} respectively. 
Then, (recall \ref{Vpdefn}, \ref{gammaframe}, \ref{defn:RT}, \ref{defn:H})
\[
\int_{S^+[p]} H_{\mathrm{gluing}}[\dz]\Ntdz\, dg =\Cn_n \sum_{e \in E_p} \td\signep \RRR[e;\dz]\Be_1.
\]
\end{lemma}
\begin{proof}This is an easy calculation involving the force vector. Let $\partial S^+[p] := \bigsqcup_{e \in E_p}\widetilde \Gamma_e$ and let $K_e \subset \Real^{n+1}$ be hypersurfaces such that $\partial K_e =\widetilde \Gamma_e$. Then, (recalling \ref{Lemma:Hdis} and \ref{Hbounds})
\begin{align*}
 n\int_{S^+[p]}H_{\mathrm{gluing}}[\dz]\Ntdz\, dg &=n\int_{S^+[p]}H_{\mathrm{error}}[\dz]\Ntdz\, dg\\
& =\int_{S^+[p]}\sum_{i=1}^{n+1} (\Delta_g x_i)\Be_i -n \int_{S^+[p]} \Ntdz\\
&=\sum_{e \in E_p} \left(\int_{\widetilde\Gamma_e} \sum_{i=1}^{n+1} \nabla_g x_i \cdot \eta_e \Be_i - n \int_{K_e} \nu_e \right) \\
&=\sum_{e \in E_p}\left( \int_{\widetilde\Gamma_e} \eta_e- n \int_{K_e} \nu_e \right).
\end{align*}
Here $\eta_e$ is the conormal to $\widetilde\Gamma_e$, $\nu_e$ is normal to $K_e$ with the appropriate orientation. 
Applying \eqref{forcevec} implies the result.
\end{proof}

\begin{definition}

Given $d,\boldsymbol \zeta $ satisfying \eqref{drestriction}, \eqref{zetarestriction}, let $\hYtdz$ denote the corresponding immersion of $M$. For $\|f\|_{2,\beta,\gamma;\dz}<\infty$, consider the immersion $(\hYtdz)_f:M \to \Real^{n+1}$. Let $\mathcal S[p]:=(\hYtdz)_f(S^+[p])$ and let $\ud[(\hYtdz)_f,\cdot] \in D(\Gamma)$ such that
\begin{equation}\label{greatdest}
\ud[(\hYtdz)_f,p]:=\Cn_n^{-\frac 12} \int_{ S^+[p]}(H_{\mathcal S[p]}-1) N_{\mathcal S[p]} \, dg_{\mathcal S[p]}.
\end{equation}
\end{definition}
\begin{prop}\label{greatdprop}
Given $d,\boldsymbol \zeta $ satisfying \eqref{drestriction}, \eqref{zetarestriction}, let $\hYtdz$ denote the corresponding immersion of $M$. If $\|f\|_{2,\beta,\gamma;\dz}<C\underline C|\utau|$, then for $\ud[(\hYtdz)_f,\cdot]$ defined by \eqref{greatdest},
\[
\left|\ud[(\hYtdz)_f,p] - d[p]\right| \leq C\underline C|\utau|^{1+(n-3+\gamma)/(n-1)}
\]for all $p \in V(\Gamma)$.
\end{prop}
\begin{proof}We will use the notation of \ref{unbalancinglemma}, but the domains $K_e$ will refer now to the parametrizing domain rather than the immersion itself.

More specifically, for $p \in V(\Gamma)$ and each $e \in E_p$, let $K_e$ denote the domain $[0, r_{\td}^{\mathrm{min}}] \times \Ssn$ and let $g_e:= ds^2 + s^2 g_{\Ssn}$ denote the standard polar metric on $K_e$.  Let $\nu_e:=\signep \RRR[e;\dz]\Be_1$. Let $\widetilde \Gamma_e:= \partial K_e$ and note that the metric $g_e$ restricted to $\widetilde\Gamma_e$ takes the form $\sigma_e:=(r_{\td}^{\mathrm{min}})^{n-1} g_{\Ssn}$. Let $\eta_e:= \nu_e$.

Then we may consider each $\nu_e$ as the outward pointing normal of an immersion of $K_e$ into $\R^{n+1}$ such that these immersions are the ``caps" of the immersion of $S^+[p]$ by the map $\hYtdz$. Moreover, each $\eta_e$ represents the conormal of the immersion of $S^+[p]$ by $\hYtdz$ along $\widetilde\Gamma_e$. 

Using \eqref{forcevec} and the calculation in the proof of \ref{unbalancinglemma}, we observe that
\begin{align*}
\Cn_n^{- \frac 12}\int_{S^+[p]}(H_{d,\boldsymbol \zeta} -1)N_{d, \boldsymbol \zeta} dg &=
\Cn_n^{- \frac 12}\sum_{e \in E_p} \left((r_{\td}^{\mathrm{min}})^{n-1} \int_{\widetilde\Gamma_e} \eta_e  dg_{\Ssn} - n \int_{K_e} \nu_e dg_e \right) 
\\&=\CCn \sum_{e \in E_p} \td\signep \RRR[e;\dz]\Be_1\\
&=  \CCn\sum_{e \in E_p} \td\signep \RRR[e;\dz]\Be_1 \\
&:=  d[\hYtdz,p].
\end{align*}
By definition, (recall \ref{defn:RT})
\begin{align*}
d[\hYtdz,p] - d[p]&= \CCn\sum_{e\in E_p} \td \signep \left(\RRR[e;\dz]\Be_1-\RRR[e;\dz]\Be_1[e] + \RRR[e;\dz]\Be_1[e] - \Bv_1[e;\tilde d,0]\right)\\
&= \CCn\sum_{e\in E_p} \td \signep \left(\RRR[e;\dz](\Be_1-\Be_1[e] )+ \Bv_1[e;\tilde d,\tilde \ell]- \Bv_1[e;\tilde d,0]\right).
\end{align*}
Using \eqref{ddifftau} and \ref{zetaframe} we see
\begin{equation}\label{dbzEst}
|d[\hYtdz,p] - d[p]|\leq  C \underline C \,\utau^2
\end{equation}
We will get the full estimate by comparing $d[\hYtdz,p]$ and $\ud[(\hYtdz)_f,p]$. 
Note that the first quantity depends upon $d, \boldsymbol \zeta, \Gamma$ while the second depends again on these and also on $f$.

Again using the proof of \ref{unbalancinglemma},
\[
\ud[(\hYtdz)_f,p] =\Cn_{n}^{-\frac 12}\sum_{e \in E_p}\left( \int_{\widetilde\Gamma_e} \eta_{e,f} d\sigma_{e,f} - n\int_{K_e} \nu_{e,f} dg_{e,f}\right).
\]Here $\eta_{e,f}, \sigma_{e,f}$ represent the modified conormal and metric, respectively, for the immersion of $\widetilde\Gamma_e$ using the map $(\hYtdz)_f$. And $dg_{e,f}$ represents the immersion of $K_e$ with respect to the same map. Note that since $N_{d,\boldsymbol \zeta} \cdot \nu_e = 0$ and $f$ is a normal graph over $\hYtdz$, the normal to the immersion of $K_e$ by the map $(\hYtdz)_f$ is the same as the normal of the immersion by $\hYtdz$. That is, $\nu_{e,f} = \nu_e$.

We compare $\ud[(\hYtdz)_f,p], d[\hYtdz,p]$ by components. For a fixed $e$, 
\begin{align*}
\left|\int_{K_e} \nu_{e,f} dg_{e,f}-\int_{K_e} \nu_e dg_e\right| &=\left| \int_{K_e}\nu_e dg_{e,f}- \int_{K_e}\nu_e dg_e\right|
\\
&\leq \left|\mathrm{Vol}_{g_{e,f}}(K_e) - \mathrm{Vol}_{g_{e}}(K_e)\right|
\end{align*}where $\mathrm{Vol}_g(\Omega)$ represents the volume of $\Omega$ with respect to the metric $g$.

By definition
\[
\mathrm{Vol}_{g_e}(K_e) = (r_{\td}^{\mathrm{min}})^n\omega_{n-1}.
\]On the other hand, we calculate
\[
\mathrm{Vol}_{g_{e,f}}(K_e) \leq \int_0^{r_\td^{\mathrm{min}}+ |f|_{C^0(\Gamma_e)}} \int_{\Ssn}r^{n-1} d\Theta dr \leq (|f|_{C^0(\widetilde\Gamma_e)} + r_{\td}^{\mathrm{min}})^n\omega_{n-1}. 
\]Since $\|f\|_{2,\beta,\gamma;\dz} \leq C\underline C|\utau|$ implies that $|f|_{C^0(\widetilde\Gamma_e)} \leq C \underline C|\utau| {\urin[e;d]}^\gamma$, 
\begin{multline} 
\label{KeuEst} 
\left|\int_{K_e} \nu_{e,f} dg_{e,f}-\int_{K_e} \nu_e dg_e\right| \leq C|f|_{C^0(\widetilde\Gamma_e)} (r_\td^{\mathrm{min}})^{n-1} \le \\ 
\leq C\underline C|\utau|{\urin[e;d]}^\gamma(r_\td^{\mathrm{min}})^{n-1} 
\leq C\underline C|\utau|^{2+ \gamma/(n-1)}.
\end{multline} 
The other estimate is a bit more delicate. We will use the triangle inequality and consider
\[
\left|\int_{\widetilde\Gamma_e} \eta_{e} d\sigma_{e} -\int_{\widetilde\Gamma_e} \eta_{e,f}d\sigma_{e,f} \right| \leq \left|\int_{\widetilde\Gamma_e} \eta_{e,f} d\sigma_{e} -\int_{\widetilde\Gamma_e} \eta_{e,f}d\sigma_{e,f} \right| +\int_{\widetilde\Gamma_e} |\eta_{e,f}-\eta_e| d\sigma_{e} .
\]Note that along $\widetilde\Gamma_e$, $\frac \partial{\partial t} \hYtdz$ is parallel to $\frac \partial{\partial t}N_{d,\boldsymbol \zeta}$. Therefore, $\angle(\eta_{e} , \eta_{e,f})$ is maximized if $f=0$ on $\widetilde\Gamma_e$. In that case, 
\[
\eta_{e,f}\cdot \eta_e = \frac{\frac \partial{\partial t} \hYtdz+ \left(\frac \partial{\partial t} f \right)N_{d,\boldsymbol \zeta}}{\sqrt{\left(r_\td^{\mathrm{min}}\right)^2+ \left|\frac \partial{\partial t} f\right|^2}} \cdot (1,\boldsymbol 0)= \frac{r_\td^{\mathrm{min}}}{\sqrt{\left(r_\td^{\mathrm{min}}\right)^2+ \left|\frac \partial{\partial t} f\right|^2}}
\]
By the definition of the global norm, $|f|_{C^1(\Gamma_e)} \leq C \underline C|\utau|{\urin[e;d]}^{\gamma-1}\leq C \underline C |\utau| \left(r_\td^{\mathrm{min}}\right)^{\gamma-1}$. Thus,
\[
\frac {\left|\frac \partial{\partial t} f\right|}{r_\td^{\mathrm{min}}} \leq C \underline C| \utau|  \left(r_\td^{\mathrm{min}}\right)^{\gamma-2}
\]and therefore
\[
|\eta_{e,f}-\eta_e| = \sqrt 2 \sqrt{1- \eta_{e,f}\cdot \eta_e}  \leq C\underline C|\utau| \left(r_\td^{\mathrm{min}}\right)^{\gamma-2}.
\]
It follows that
\begin{equation}\label{FirstGeuEst}
\int_{\widetilde\Gamma_e} |\eta_{e,f}-\eta_e| d\sigma_{e} \leq  C\underline C|\utau| \left(r_\td^{\mathrm{min}}\right)^{\gamma-2}\mathrm{Vol}_{\sigma_e}(\widetilde\Gamma_e) \leq C \underline C |\utau|  \left(r_\td^{\mathrm{min}}\right)^{n-3+\gamma}\leq C\underline C|\utau|^{2+(\gamma-2)/(n-1)}.
\end{equation}For the last estimate, we note that
\[
\sigma_{e,f} = \sum_{i=1}^n\left(r^2 + f_i^2+ f^2-2rf\right) d\Theta_i^2= \sum_{i=1}^n\left((r-f)^2 + f_i^2\right) d\Theta_i^2
\] where $\sum_{i=1}^n d\Theta_i^2$ represents the standard metric on $\Ssn$. Thus 
\begin{align*}
\int_{\widetilde\Gamma_e} d\sigma_{e,f} &\leq \omega_{n-1}\left((r_\td^{\mathrm{min}}+|f|_{C^0(\widetilde\Gamma_e)})^2 + |f|^2_{C^1(\Gamma_e)}\right)^{(n-1)/2}\\
& \leq \omega_{n-1}\left(r_\td^{\mathrm{min}}\right)^{n-1} + C\underline C |\utau| \left(r_\td^{\mathrm{min}}\right)^{n-3+\gamma}.
\end{align*}Here we used the estimate $(a^2+b^2)^{1/2} \leq a+b$ for $a,b\geq 0$ and the fact that the worst remaining term that appears in the expansion has the form of the second term above.

Applying this estimate reveals that
\begin{align}
\left|\int_{\widetilde\Gamma_e} \eta_{e} d\sigma_{e} -\int_{\widetilde\Gamma_e} \eta_{e}d\sigma_{e,f} \right| &\leq |\mathrm{Vol}_{\sigma_{e,f}}(\widetilde\Gamma_e) - \mathrm{Vol}_{\sigma_e}(\widetilde\Gamma_e)| \notag\\& \leq C\underline C |\utau| \left(r_\td^{\mathrm{min}}\right)^{n-3+\gamma} \notag\\& \leq C\underline C|\utau|^{2+(\gamma-2)/(n-1)}.\label{SecondGeuEst}
\end{align}
Combining the estimates of 
\eqref{dbzEst}, \eqref{KeuEst}, \eqref{FirstGeuEst}, and \eqref{SecondGeuEst} implies the result.
\end{proof}

\begin{prop}\label{subsprescribep}
For $d, \boldsymbol \zeta$ satisfying \eqref{drestriction}, \eqref{zetarestriction} respectively, we have the following:

For each $p \in V(\Gamma)$, there exist $\varphi_{\mathrm{gluing}}[p] \in C^{2,\beta}(\widetilde S[p]), \mu_i'[p]$, and  $\mu_i'\pe$, where $i=0, \dots, n$, such that
\begin{enumerate}
\item $\mathcal L_g \varphi_{\mathrm{gluing}}[p] + H_{\mathrm{gluing}}[\dz] = \sum_{i=0}^n \mu_i'[p]w_i[p] + \sum_{e \in E_p}\sum_{i=0}^n \mu_i'\pe w_i\pe$ on $\widetilde S[p]$.
\item $\varphi_{\mathrm{gluing}}[p] =0$ on $\partial \widetilde S[p]$.
\item $\|\varphi_{\mathrm{gluing}}[p]:C^{2,\beta}(\widetilde S[p],g)\| \leq C|\utau|$. 
\item $\|\varphi_{\mathrm{gluing}}[p]:C^{2,\beta}(\Lambda[p,e,1],\rho_d,g,\ur_d^\gamma)\| \leq C|\utau|$ for all $e \in E_p$.
\item $\left|\mu_i'\pe\right| \leq C|\utau|$.
\item Each of the $\varphi_{\mathrm{gluing}}, \mu_i'[p], \mu_i'\pe$ are all unique by their construction and depend continuously on $\dz$.
\end{enumerate}
\end{prop}
\begin{proof}
Determine $\mu_i'[p]$ such that for $j=0,\dots, n$, 
\[
\int_{\widetilde S[p]} \left(\sum_{i=0}^{n} \mu_i' [p]w_i[p] - H_{\mathrm{gluing}}[\dz]\right) f_j[p] \: dg=0.
\]
The estimates follow immediately from \ref{linearpartp} and the bound on $H_{\mathrm{gluing}}[\dz]$ given in \ref{Hestimates}.
\end{proof}

\subsection*{Prescribing the extended substitute kernel}
\begin{prop}\label{localquadprop}
Let $\hYtdz:M \to \Real^{n+1}$ be the immersion of \ref{immersiondef}.
Let $x \in M$ and $D \subset M$ be a disk of radius $1/10$ in the metric $\rho_d^{-2}(x)g$, centered at $x$. Let $c_1>0$ denote the constant found in \ref{find:c1}.

If $v \in C^{2, \beta}(D,\rho_d^{-2}(x)g)$ satisfies 
\[
\|v:C^{2,\beta}(D,\rho_d^{-2}(x)g)\| \leq {\rho_d(x)}{\epsilon(c_1)}
\]for $\epsilon(c_1)$ given by \ref{quadboundsunscaled}, then
\[
\rho_d(x)\|\left(\hYtdz\right)_v - \hYtdz - v\Ntdz: C^{1,\beta}(D,\rho_d^{-2}(x)g)\| \leq C(c_1) \|v:C^{2,\beta}(D,\rho_d^{-2}(x)g)\|^2
\]
\[
\rho_d^{2}(x)\|H_v-\Htdz-\mathcal L_g v:C^{0,\beta}(D,\rho_d^{-2}(x)g)\| \leq C(c_1)\rho_d^{-1}(x)\|v:C^{2,\beta}(D,\rho_d^{-2}(x)g)\|^2.
\]
\end{prop}

\begin{proof}
We wish to apply \ref{quadboundsunscaled} and to that end, we rescale the target by $\rho_d^{-1}(x)$.

We now proceed with the proof. By \eqref{rhoest}, the conditions of \eqref{quadconditions} are satisfied and the hypothesis of \ref{quadboundsunscaled} is satisfied for the rescaled function $\rho_d^{-1}(x)v$. Under rescaling, $H_{\rho_d^{-1}(x)v} = {\rho_d(x)}H_v$ and $\mathcal L_{(\rho_d^{-1}(x)\hYtdz)^*(\rho_d^{-2}(x)g)} = \rho_d^{2} (x)\mathcal L_{g}$. Thus, \ref{quadboundsunscaled} implies
\[
\|\rho_d(x)\left(H_v - \Htdz- \mathcal L_{g}v\right) :C^{0,\beta}(D,\rho_d^{-2}(x)g)\|\leq C(c_1)\|\rho_d^{-1}(x) v: C^{2, \beta}(D,\rho_d^{-2}(x)g)\|^2.
\] Simplifying implies the second estimate. The first estimate follows directly from the scaling of the target and the function.
\end{proof}

For ease of presentation, we define rotations of the components of the normal vector.
\begin{definition}\label{tildefdefn}
Let $\tilde f_i:\bigsqcup_{e \in E(\Gamma) \cup R(\Gamma)} M[e] \to \Real$ for $i=0, \dots, n$ such that for $x \in M[e]$,
\[
\tilde f_i(x):= {\left(\RRR[e;\dz]^{-1}\Ntdz(x)\right)\cdot \Be_{i+1} }.
\]
\end{definition}Note that $\mathcal L_g \tilde f_i=0$.

For $e \in E(\Gamma)\cup R(\Gamma)$, let $f^+[e]:M[e] \cap\left( [a,a+3] \times \Ssn\right)\to \Real$ and for $e \in E(\Gamma)$ let $f^-[e]: M[e] \cap\left(  [\RH -(a+3),\RH -a] \times \Ssn\right)\to \Real$ such that on these regions 
\begin{align}
\label{one}\UUU[e;\dz]^{-1}\circ\hYtdz\circ D^+ &=\left(Y_0 + \boldsymbol \zeta\ppe\right)_{f^+[e]}\\
\UUU[e;\dz]^{-1}\circ\hYtdz \circ D^-& =\left(Y_0 + \boldsymbol \zeta\pme\right)_{f^-[e]}
\end{align} 
where $D^\pm$ are small perturbations of the identity map.  
We prove estimates for $f^+[e]$ and note that the same estimates hold for $f^-[e]$, 
once we account for appropriate changes to the domain. 
\begin{lemma}\label{quadflemma}
For $f^+[e]$ as above, we have the following:
\begin{enumerate}
\item \label{qf1}$f^+[e]= 0$ on $M[e] \cap \left([a,a+1]\times \Ssn\right)$.
\item \label{qf2}$\|f^+[e]:C^{2,\beta}(M[e] \cap \left([a,a+3] \times \Ssn\right),g)\|\leq C\left| \boldsymbol \zeta \right|$.
\item \label{qf4}For $\tilde f_i$ described above, 
\[\|f^+[e](x) - \sum_{i=0}^{n} \zeta_i\ppe\tilde f_i(x):C^{1,\beta}(M[e] \cap \left([a+2,a+3]\times \Ssn\right), g)\|\leq C\left| \boldsymbol \zeta \right|^2.\]
\item \label{qf3}$\|\mathcal L_{g} f^+[e] -  H_{\mathrm{dislocation}}[\dz]:C^{0,\beta}(M[e] \cap \left([a,a+3] \times \Ssn\right),g)\|\leq C\left| \boldsymbol \zeta \right| \, |\utau|^{\frac 1{n-1}}$.
\end{enumerate}
\end{lemma}

\begin{remark}
Note that we could have stated the previous lemma to fit with the global norm on $S_1[p]$ but since $\rho_d \sim_{C(\bunder)}1$ on $S_1[p]$ the norm bounds given above can be used to bound the global norm.
\end{remark}

\begin{proof}
Items \eqref{qf1} and \eqref{qf2} follow immediately from the definition of $f^+[e]$ and the behavior of the immersion $\hYtdz$ on each of the domains. For items \eqref{qf4} and \eqref{qf3}, we note that item \eqref{qf2} and the uniform estimates on $\rho_d$ allow us to invoke \ref{localquadprop}, which we do with $Y_0 +\boldsymbol \zeta\ppe$ in place of $\hYtdz$. Then item \eqref{qf4} follows from the definition of the functions $\tilde f_i$ and the linear error estimate of \ref{localquadprop} since $\UUU[e;\dz]^{-1}\circ\hYtdz=\RRR[e;\dz]^{-1}N_{\dz} =Y_0$ for $t \in [a+2,a+3]$. Item \eqref{qf3} follows from the quadratic error estimate of \ref{localquadprop} applied to $Y_0 +\boldsymbol \zeta\ppe$ and by recalling \ref{geolimit} to compare $\mathcal L_gf^+[e]$ and $\mathcal L_{Y_0}f^+[e]$.
\end{proof}

\begin{definition}\label{underlinef}Let $S^x[p] \subset M$ such that
\begin{align*}
S^{x}[p]:=M[p]  &\bigsqcup_{\{e|p=p^+[e]\}}\left(M[e] \cap [a,x] \times \Ssn\right)\\& \bigsqcup_{\{e|p=p^-[e]\}} \left(M[e] \cap [\RH -x,\RH -a]\times \Ssn\right)
\end{align*}with the appropriate regions identified as in \eqref{eq:sim}. 
For each $p \in V(\Gamma)$, let 
\[
\underline f[p]: S^{a+3}[p]\to \Real
\]
such that
\begin{equation*}
\underline f[p](x)= \left\{\begin{array}{ll}
0,& \text{if } x \in M[p],\\
f^+[e](x),& \text{if }  p=p^+[e],x \in M[e]\cap [a,a+3] \times \Ssn,\\
f^-[e](x),&\text{if } p=p^-[e],x\in M[e] \cap[\RH -(a+3),\RH -a]\times \Ssn.
\end{array}\right.
\end{equation*}
\end{definition}
The definition of $\underline f[p]$ immediately implies the following corollary.
\begin{corollary}\label{corundf}
\begin{enumerate}
\item $\underline f[p] =0$ on $S^{a+1}[p]$.
\item $\|\mathcal L_g \underline f[p]-H_{\mathrm{dislocation}}[\dz]:C^{0,\beta}(S^{a+3}[p],g)\|\leq C\left| \boldsymbol \zeta \right|\, |\utau|^{\frac 1{n-1}}$.
\item if $p=p^+[e]$, 
\[\|\underline f[p]- \sum_{i=0}^{n} \zeta_i\ppe\tilde f_i(x):C^{1,\beta}(M[e] \cap \left([a+2,a+3]\times \Ssn\right), g)\|\leq C\left| \boldsymbol \zeta \right|^2.\]
\item if $p=p^-[e]$, 
\[\|\underline f[p]- \sum_{i=0}^{n} \zeta_i\ppe\tilde f_i(x):C^{1,\beta}(M[e] \cap \left([\RH-(a+3),\RH-(a+2)]\times \Ssn\right), g)\|\leq C\left| \boldsymbol \zeta \right|^2.
\]
\end{enumerate}
\end{corollary}\noindent
We use these functions to prescribe the dislocation on each central sphere. For convenience, we normalize the functions $\tilde f_i$ on the meridian circle $C_1^{\mathrm{out}}[p,e,0]$. 

\begin{definition}\label{wfdef}For $\pe \in A(\Gamma)$, choose $c'_i[p,e]$ such that for each $i=0,\dots, n$, on $C_1^{\mathrm{out}}[p,e,0]$, (recall \ref{phidef})
\[
c'_i[p,e] \tilde f_i =\phi_{i}.
\]
\end{definition}
\begin{remark}\label{ref:cprime}
Notice that while $c'_i\pe$ depends on $d$, these values are independent of $\boldsymbol \zeta$ since $\tilde f_i$ is independent of $\boldsymbol \zeta$ at $(\bunder+1,\bt)$. 
Moreover, by the asymptotic geometric behavior at $\bunder+1$ (recall \ref{radiuslemma}), 
\[
|c'_i\pe| \sim_{C(\bunder)} 1.
\]
\end{remark}
\begin{assumption}
We now choose a constant $c'\geq 1$, independent of $\dz$ and of $\utau$ but depending on $\bunder$ such that for all $d$ satisfying \eqref{drestriction} and corresponding $c'_i \pe$, 
\begin{equation}\label{cprimedef}
c' \geq \max_{\substack{i=0,\dots,{n},\\ [p,e] \in A(\Gamma)}} |c'_i[p,e]|.
\end{equation}
\end{assumption}

Recalling \eqref{tildecs2} and \ref{D:wi}, the normalization we choose implies that
\begin{equation}\label{eq:vanish}
c'_i[p,e]\tilde f_i - v_i\pe|_{C_1^{\mathrm{out}}[p,e,0]}=0.
\end{equation}This normalization will be convenient for estimating item $(3)$ in the proposition below.
\begin{prop}\label{prescribequad}
Let $\dz$ satisfy \eqref{drestriction}, \eqref{zetarestriction}, respectively. 
For each $p \in V(\Gamma)$ there exist $\phi_{\mathrm{dislocation}}[p]\in C^{2,\beta}(\widetilde S[p])$, $\mu_i''[p]$, and $\mu_i''[p,e]$, where $i=0, \dots, n$, such that
\begin{enumerate}
\item $\mathcal L_g \phi_{\mathrm{dislocation}}[p] + H_{\mathrm{dislocation}}[\dz] = \sum_{i=0}^{n} \left( \mu_i''[p] w_i[p]+ \sum_{e \in E_p} \mu_i''[p,e]w_i[p,e]\right)$ on $\widetilde S[p]$.
\item $\phi_{\mathrm{dislocation}}[p] = 0$ on $\partial \widetilde S[p]$.
\item $|\mu_i''[p]| + | \zeta_i[p,e]/c_i'\pe-\mu_{i}''[p,e]|\leq C\left| \boldsymbol \zeta \right| \, |\utau|^{\frac 1{n-1}}$ 
for all $i = 0, \dots, n$. 
\item $\| \phi_{\mathrm{dislocation}}[p]:C^{2,\beta}(\widetilde S[p], g)\|\leq C\left| \boldsymbol \zeta \right|$. 
\item $\|\phi_{\mathrm{dislocation}}[p]:C^{2,\beta}(\Lambda[p, e, 1], \rho_d,g, \ur_d^\gamma)\|\leq C\left| \boldsymbol \zeta \right| \, |\utau|^{\frac 1{n-1}}$ for all $e \in E_p$.
\item $\phi_{\mathrm{dislocation}}[p], \mu_i''[p], \mu_{i}''[p,e]$ are all unique by their construction and depend continuously on the parameters of the construction.
\end{enumerate}
\end{prop}
\newcommand{\uphipp}[0]{\phi_{\mathrm{dislocation}}'[p] }
\newcommand{\uphip}[0]{\phi_{\mathrm{dislocation}}''[p]}
\begin{proof}
Let $\underline f[p]$ represent the function from \ref{underlinef}.  On $M[e] \cap S[p]$, for $p=p^+[e]$ set
$\check \psi_e:=\psi[a+3,a+2](t)$ and for $p=p^-[e]$ set $\check \psi_e:= \psi[\RH-(a+3), \RH-(a+2)](t)$.  On each $\Lambda[p,e,1]$, find  $\underline V_e$ such that $\mathcal L_g \underline V_e =0$, $\underline V_e = -\sum_{i=0}^n \zeta_i[p,e] \tilde f_i$ on $C^{\mathrm{in}}[p,e,1]$ and $\underline V_e = 0$ on $C^{\mathrm{out}}[p,e,0]$. 
We construct $\uphipp \in C^{2, \beta} (\widetilde S[p])$ in the following way. Let
\[
\uphipp = \left\{ \begin{array}{ll}
\underline f[p]& \text{on } M[p],\\
 \check \psi_e \underline f[p]+(1-\check \psi_e) \sum_{i=0}^n \zeta_i[p,e] \tilde f_i&\text{on } M[e] \cap S[p], \pe \in A(\Gamma),\\
 \sum_{i=0}^n \zeta_i[p,e] \tilde f_i+ (1-\psi_{S[p]}[d])\underline V_e& \text{on } \Lambda[p,e,1].
\end{array}\right.
\]
The construction of $\underline V_e$ and the estimates of \ref{annulardecaylemma} imply that
\[\|\underline V_e:C^{2,\beta}([\bunder,\bunder+2] \times \Ss^{n-1} \cap M[e],g)\|\leq C\left| \boldsymbol \zeta \right| (\urin[e;d])^{n-2}. \]

Noting the estimates provided by \ref{corundf}, we have the following for $\uphipp$:
\begin{enumerate}
\item $\mathcal L_g \uphipp$ is supported on $S_1[p]$ and $\uphipp=0$ on $\partial \widetilde S[p]$.
\item For each $e \in E_p$, $$\|\mathcal L_g \uphipp - H_{\mathrm{dislocation}}[\dz]:C^{0,\beta}(S_1[p] \cap M[e], g)\| \leq C\left| \boldsymbol \zeta \right|(|\utau|^{\frac 1{n-1}}+ (\urin[e;d])^{n-2}).$$
\end{enumerate}

On $C^{\mathrm{out}}_1[p,e,0]$, $\uphipp = \underline V_e + \sum_{i=0}^n \zeta_i[p,e] \tilde f_i$. To modify $\uphipp$ and prescribe the fast decay, we follow the argument of \ref{linearpartp}. In this case, $\tilde f_i \in \mathcal H_1[C_1^{\mathrm{out}}[p,e,0]]$ so we first find $\mathcal R_\partial^{\mathrm{out}}(\underline V_e^\perp|_{C_1^{\mathrm{out}}[p,e,0]})$. Recall that $\underline V_e^\perp=\underline V_e-\underline V_e^T$ where $\underline V_e^T$ denotes the projection of $\underline V_e$ onto $\mathcal H_1[C_1^{\mathrm{out}}[p,e,0]]$. Thus, we find coefficients $\underline \mu_i[p,e]$ such that (recall \ref{D:wi}, \ref{wfdef})
\begin{align*}
\sum_{i=0}^n{ \underline \mu_i[p,e]} v_i[p,e]& = \uphipp - \mathcal R_\partial^{\mathrm{out}}\left(\underline V_e^\perp|_{C_1^{\mathrm{out}}[p,e,0]}\right)\\
&= \underline V_e+ \sum_{i=0}^n \zeta_i[p,e] \tilde f_i- \mathcal R_\partial^{\mathrm{out}}\left(\underline V_e^\perp|_{C_1^{\mathrm{out}}[p,e,0]}\right).
\end{align*}
Since this implies that on $C_1^{\mathrm{out}}[p,e,0]$, 
\[
-\sum_{i=0}^n \zeta_i[p,e] \tilde f_i+\sum_{i=0}^n {\underline \mu_i[p,e]}v_i[p,e]= \underline V_e^T + \underline V_e^\perp - \mathcal R_\partial^{\mathrm{out}}\left(\underline V_e^\perp|_{C_1^{\mathrm{out}}[p,e,0]}\right)
\]we can appeal to the estimates of \ref{linearcor}, exploiting the normalization given in \eqref{eq:vanish}, to conclude that
\[
\left|\frac{\zeta_i[p,e]}{c_i'\pe} - \underline \mu_i[p,e]\right| \leq C\|\underline V_e:C^{2,\beta}([\bunder,\bunder+2] \times \Ss^{n-1} \cap M[e],g)\|\leq C\left| \boldsymbol \zeta \right| (\urin[e;d])^{n-2},
\]
\[
\|\uphipp- \sum_{i=0}^n\underline \mu_{i}[p,e] v_i[p,e]:C^{2,\beta}(\Lambda[p,e,1],\rho_d, g, \ur_d^\gamma)\| \leq C \left| \boldsymbol \zeta \right|( \urin[e;d])^{n-2}.
\] Using \ref{linearpartp} with $E: = \mathcal L_g \uphipp - H_{\mathrm{dislocation}}[\dz]$, let $$(\uphip, w_{\mathrm{dislocation}}):= \mathcal R_{\widetilde S[p]}(E)$$ where
\[
w_{\mathrm{dislocation}} = \sum_i \mu_i''[p] w_i[p] + \sum_{e \in E_p}\sum_{i} \mu_{i}'''[p,e]w_i[p,e].
\]
 Then
\[\mathcal L_g \uphip= \mathcal L_g \uphipp - H_{\mathrm{dislocation}}[\dz]+ w_{\mathrm{dislocation}} \text{ on } \widetilde S[p], \text{ and}\]
\[|\mu_i''[p]|, |\mu_{i}'''[p,e]|\leq C\left| \boldsymbol \zeta \right|  \, |\utau|^{\frac 1{n-1}} 
\]

Set
\[
\phi_{\mathrm{dislocation}}[p] = \uphip - \uphipp + \sum_{e \in E_p}\sum_{i}{\underline \mu_{i}[p,e]}v_i[p,e]
\]and
\[
\mu_{i}''[p,e] = \mu_{i}'''[p,e]+{\underline\mu_{i}[p,e]}.
\]We complete the proof by appealing to all of the estimates above and those of \ref{linearpartp}.
\end{proof}

\subsection*{Prescribing the extended substitute kernel globally}

\newcommand{\bxi}{\boldsymbol \xi}
Choose $d \in D(\Gamma)$ satisfying \eqref{drestriction} such that
\begin{equation}\label{dtaxi}
d[p] = \sum_{i=0}^{n}d_i[p]\Be_{i+1}  
\end{equation}
and choose  $\bxi \in Z(\Gamma)$ such that (recall \ref{cprimedef}) $|\bxi| \leq \underline C |\utau|/c'$ and 
\[
  \bxi\pe:= \sum_{i=0}^n\xi_i \pe  \Be_{i+1}.
\]  Let $\boldsymbol \zeta \in Z(\Gamma)$ such that (recall \ref{wfdef})
\begin{equation}\label{zetaxi}
  \zeta_i\pe =  c_i'\pe \xi_i\pe.
\end{equation}Then $ \boldsymbol \zeta$ satisfies \eqref{zetarestriction}. 

Using this $\dz$, find $\phi_{\mathrm{gluing}}[p],\phi_{\mathrm{dislocation}}[p]$ using \ref{subsprescribep}, \ref{prescribequad} and set
\[
\Phi'_{\dz} = {\bf U}\left( \left\{\psi_{\widetilde S[p]}[d] (\phi_{\mathrm{gluing}}[p]+ \phi_{\mathrm{dislocation}}[p]), 0 \right\}\right).
\] 
Setting $\mu_i[p]:=\mu_i'[p]+\mu_i''[p], \mu_i\pe:=\mu_{i}'\pe+\mu_i''\pe$ where these coefficients come from \ref{subsprescribep}, \ref{prescribequad}, define
\[
\underline w'_v:=\sum_{i=0}^n \sum_{p \in V(\Gamma)}\mu_i[p]w_i[p]  \in \mathcal K_V, \quad \quad \underline w'_a:= \sum_{i=0}^n
\sum_{\pe \in A(\Gamma)} \mu_i\pe w_i\pe  \in \mathcal K_A.
\]
Using \ref{LinearSectionProp}, determine
 $$(\Phi_{\dz}'' , \underline w''_v, \underline w''_a):=\mathcal R_M\left(-\mathcal L_g\Phi_{\dz}'+\underline w'_v+ \underline w'_a - H_{\mathrm{error}}[\dz]\right).$$  Now set 
\[
\Phi_{\dz}:= \Phi_{\dz}'' + \Phi_{\dz}' \: \text{ and } (\underline w_{d})_v:= \underline w''_v + \underline w'_v \: \text{ and } (\underline w_{\boldsymbol \zeta})_a:= \underline w''_a + \underline w'_a.
\]

\begin{prop} 
\label{prescribexi}
For $\dz$ chosen as in \eqref{dtaxi}, \eqref{zetaxi} respectively, the functions $\Phi_{\dz}$, $ (\underline w_{d})_v$, and $(\underline w_{\boldsymbol \zeta})_a$ depend continuously on $\dz$ and satisfy (by decreasing $\maxT$ if necessary):
\begin{enumerate}
\item $\mathcal L_g \Phi_{\dz}  +H_{\mathrm{error}}[\dz]=  (\underline w_{d})_v+(\underline w_{\boldsymbol \zeta})_a$ on $M$ 
(recall \ref{Hbounds}).
\item $\|\Phi_{\dz}\|_{2,\beta,\gamma; \dz} \leq C(|\utau|+\left| \boldsymbol \zeta \right|) \leq C \underline C |\utau|$.
\item 
$\left|(\underline w_{\boldsymbol \zeta})_a- (w_{\boldsymbol \zeta})_a\right|_{A} \leq C  |\utau|,$ 
where 
$(w_{\boldsymbol \zeta})_a:=\sum_{i=0}^n\sum_{\pe \in A(\Gamma)}  { \xi}_i\pe w_i\pe$  
(recall \ref{D:calKp} and \ref{D:wi}).  
\end{enumerate}
\end{prop}

\begin{proof}By construction, item $(1)$ is immediately satisfied. Moreover, the estimates for $\phi_{\mathrm{gluing}}[p]$ and $\phi_{\mathrm{dislocation}}[p]$ imply that
\[
\|\Phi_{\dz}'\|_{2,\beta,\gamma;\dz} \leq C\left(|\utau| + \left| \boldsymbol \zeta \right|\right).
\]
To determine estimates for $\Phi_{\dz}''$ note that for $E :=-\mathcal L_g\Phi_{\dz}'+\underline w'_v+ \underline w'_a - H_{\mathrm{error}}[\dz]$, 
\[
E|_{\widetilde S[p]}=\mathcal L_g\left((1-\psi_{\widetilde S[p]}[d]) (\phi_{\mathrm{gluing}}[p]+ \phi_{\mathrm{dislocation}}[p])\right):=E[p].
\]and
\[
\supp(E) \subset \cup_{\pe \in A(\Gamma)} \left(\Lambda[p,e,1] \cap [\ur_{d}^{-1}(2\urin[e;d]), \Pe-\bunder] \times \Ssn\right).
\]
Using the same strategy that produced the estimate \eqref{iterationest}, we note that
\begin{align*}
\|E\|_{0,\beta,\gamma;\dz} \leq \max_{p\in V(\Gamma)}\|E[p]:C^{0,\beta}(S_1[p,e,1] \cap \Lambda[p,e,1],\rho_d, g ,\ur_{d}^{\gamma-2} )\| \leq C\max_{e \in E_p}\urin[e;d]^{\gamma'-\gamma}\left(|\utau| + \left| \boldsymbol \zeta \right| \right).
\end{align*}Therefore, the estimates from \ref{LinearSectionProp} imply
\[
\|\Phi_{\dz}''\|_{2,\beta,\gamma;\dz} + \|\underline w''_v\|_{2,\beta,\gamma;\dz}+ \|\underline w''_a\|_{2,\beta,\gamma;\dz} \leq  C\max_{e \in E(\Gamma) \cup R(\Gamma)}\urin[e;d]^{\gamma'-\gamma}\left(|\utau| + \left| \boldsymbol \zeta \right|\right).
\] 
Finally, the estimates on $\mu_i[p,e]$ from \ref{subsprescribep}, \ref{prescribequad} imply
$
\left|\underline w_a'- (w_{\boldsymbol \zeta})_a\right|_{A} \leq C( |\utau|+\left| \boldsymbol \zeta \right| |\utau|^{\frac 1{n-1}} ).
$
\end{proof}

\section{The Main Theorem}
\label{MThm} 

\begin{prop} 
\label{globalquadprop} 
There exists $\maxTG>0$ sufficiently small such that for all $0<|\utau|<\maxTG$,  
$\alpha \in \left(0, 1\right)$, and $v \in C^{2, \beta}_{loc}(M)$ such that $\|v\|_{2,\beta,\gamma;\dz} \leq |\utau|^{1-\alpha/2}$,  
we have 
(recall \ref{metricdefn} and \ref{metricequivalence}) 
\[
\|H_v - \Htdz- \mathcal L_{g}v\|_{0,\beta,\gamma;\dz} \leq |\utau|^{\alpha-1}\|v\|_{2,\beta,\gamma;\dz}^2/\Ctilde.
\]
\end{prop}

\begin{proof}
The estimate $\|v\|_{2,\beta,\gamma;\dz} \leq |\utau|^{1-\alpha/2}$ coupled with the definition of $\rho_d$ implies that for every $x \in M$, $v$ satisfies the hypothesis of \ref{localquadprop} on the disk of radius $1/10$ in the metric $\rho_d^{-2}(x)g$, centered at $x$. We refer to it as $D$. The conclusion of the same proposition implies that
\[
\|H_v-\Htdz-\mathcal L_g v:C^{0,\beta}(D,\rho_d,g,f_d\rho_d^{-2})\| \leq C(c_1)f_d(x)\rho_d^{-1}(x)\|v:C^{2,\beta}(D,\rho_d,g,f_d)\|^2,
\]where, for $D \subset S_1[p], S_1^+\penp$, we presume that $f_d\equiv 1$. 
The definition of the global norm and the fact that $\rho_d \sim_{C(\bunder)} 1$ on $S_1[p], 
S_1^+\penp$ implies that it is enough to show that the right hand side of the inequality is bounded above by $|\utau|^{\alpha-1}\|v\|_{2,\beta,\gamma;\dz}^2/\Ctilde$.

For each $D \subset S_1[p]$ or $D \subset S^+_1\penp$, $f_d(x)\rho_d^{-1}(x) \leq C$ and thus
\begin{multline*} 
C(c_1)f_d(x)\rho_d^{-1}(x)\|v:C^{2,\beta}(D,\rho_d,g,f_d)\|^2 \leq 
\\ 
\le C\|v:C^{2,\beta}(D,\rho_d,g,f_d)\|^2 \leq |\utau|^{\alpha-1}\|v:C^{2,\beta}(D,\rho_d,g,f_d)\|^2/\Ctilde 
\end{multline*} 
for sufficiently small $\utau$. 
Since the weighting $\de^{-m/2}>1$, it follows that $\|v:C^{2,\beta}(D,\rho_d,g,f_d)\| < \|v\|_{2,\beta,\gamma;\dz}$ for any $D \subset S^+_1 \penp$.

On each $M[e]$, the definitions imply that $f_d\rho_d^{-1}$ is maximized at $\Pe$. 
By decreasing $\maxTG$ if necessary, depending on $\bunder$,  
\[
\max_{e \in E(\Gamma) \cup R(\Gamma)}f_d(\Pe)\rho_d^{-1}(\Pe)= 
\max_{e \in E(\Gamma) \cup R(\Gamma)} \urin[e;d]^{\gamma}(r_{\td}^{\mathrm{min}})^{-1} \leq C(\bunder)|\utau|^{\frac {\gamma-1}{n-1}} \leq |\utau|^{\alpha-1}/\Ctilde.
 \] 
The result again follows immediately by the definition of the global norms.
 \end{proof}

\begin{theorem}
Let $\Gamma$ be a finite central graph with an associated family $\calF$. 
Then there exist $\underline C, \bunder$ sufficiently large and $\maxTG>0$ sufficiently small so that for all $0<|\utau|<\maxTG$:

There exist $d,\boldsymbol \zeta$ satisfying \eqref{drestriction},\eqref{zetarestriction} and a function $f \in C^{2,\beta}\loc(M)$ such that $(\hYtdz)_f:M \to \Real^{n+1}$ is an immersed surface with CMC equal to $1$ and $\|f\|_{2,\beta,\gamma;\dz} \leq \widetilde C|\utau|$ (recall \ref{Meq}). Moreover, if $\Gamma$ is pre-embedded then $(\hYtdz)_f$ is embedded for $\utau>0$.
\end{theorem}
In the statement, $M$ is the abstract surface based on the graph $\Gamma$ and the parameter $\utau$. $\hYtdz$ is the immersion described in Section \ref{InitialSurface} depending on $\Gamma, \utau, d, \boldsymbol \zeta$. Finally, $(\hYtdz)_f$ is the normal graph over $\hYtdz$ by $f$, as defined in Appendix \ref{quadapp}.

\begin{proof}
Choose $\bunder\gg 1$ as in \ref{ass:b}. 
Recall that $\widetilde C, c'$ (\ref{Meq}, \ref{cprimedef}) and all $C$ appearing in the statement of \ref{prescribexi} depend only on $\bunder$. 
Choose $\underline C$ independent of $\maxTG$ so that $\underline C \geq 4c' C$.  
Choose 
$\maxTG>0$ as in \ref{ass:tgamma6}. 
We again point out that $\maxTG$ does not depend upon the structure of $\Gamma$ but only on the function $\hat \tau$ and on various geometric quantities. 
Moreover, $\bunder$ is independent of $\maxTG$. Fix $\alpha \in \left(0,1\right)$ and $\gamma \in (1,2)$. 
Reduce $\maxTG$ if necessary so that   
$C\underline C+ \widetilde C \leq \maxTG^{-\alpha/2}$  
and $C\underline C \maxTG^{\frac{\gamma- 1}{n-1}} \leq 1$. 
For any $0<|\utau|<\maxTG$, we define $B$ to be the set 
\[
\{u \in C^{2,\beta}\loc(M): \|u\|_{2,\beta,\gamma;0, \boldsymbol 0} \leq |\utau|\} \times 
\{d \in D(\Gamma):\left|d\right| \leq  |\utau|^{1+\frac 1{n-1}}\} \times \{\bxi \in Z(\Gamma):\left|\bxi \right| \leq \underline C|\utau|/c'\}.
\] 
We define $\mathcal J: B \to B$ in the following manner.  
For $(u, d, \bxi) \in B$, define $\boldsymbol \zeta$ by \eqref{zetaxi}.  
Using this $\dz$, find $\Gtdtl \in \calF$ and determine $\hYtdz$ in the manner outlined in Section \ref{InitialSurface}.  
Determine $\Phi_{\dz}$ and $(\underline w_{d})_v, (\underline w_{\boldsymbol \zeta})_a$ by \ref{prescribexi} and define a function $\tilde u = \Phi_{\dz} -u$. Then,
\[
\mathcal L_g \tilde u = (1-\Htdz) + (\underline w_{d})_v +(\underline w_{\boldsymbol \zeta})_a-\mathcal L_g u,
\]
\[
\|\tilde u\|_{2,\beta,\gamma;\dz} \leq C\underline C|\utau| + \widetilde C |\utau|\leq |\utau|^{1-\alpha/2}.
\]Using \ref{LinearSectionProp}, with $H_{\tilde u}$ denoting the mean curvature of the surface $(\hYtdz)_{\tilde u}$, define $(u',w'_v, w'_a)= \mathcal R_{\dz}(H_{\tilde u} - \Htdz-\mathcal L_g \tilde u)$. Then
\[
\mathcal L_g u' = H_{\tilde u} -1+ \mathcal L_g u + w'_v-(\underline w_{d})_v+ w'_a-(\underline w_{\boldsymbol \zeta})_a
\] 
and by \ref{globalquadprop}, 
\begin{equation}\label{eq:CCC}
\|u'\|_{2,\beta,\gamma;{d,\boldsymbol \zeta}} +\left|w'_v\right|_V+\left|w'_a\right|_A \leq 
|\utau|^{\alpha-1}\|\tilde u\|^2_{2,\beta,\gamma;{d,\boldsymbol \zeta}}/\Ctilde  \leq 
|\utau|/\widetilde C.
\end{equation}
By \eqref{Meq}
\begin{equation}\label{wprime}
\|u'\|_{2,\beta,\gamma;{0,\boldsymbol 0}} \leq |\utau|.
\end{equation}

Define $\boldsymbol \mu_a \in Z(\Gamma)$ such that
\[
\boldsymbol \mu_a\pe:= \sum_{i=0}^n \mu_i\pe \Be_{i+1}
\] where the coefficients $\mu_i \pe$ are determined to satisfy
\[\sum_{i=0}^n \mu_i\pe w_i\pe =\left( w_a' + (w_{\boldsymbol \zeta})_a-(\underline w_{\boldsymbol \zeta})_a\right)\big|_{S_1[p] \cap \Lambda[p,e,1]}
\]and the definition of $(w_{\boldsymbol \zeta})_a$ is given in \ref{prescribexi}. The estimates from \ref{prescribexi} and \eqref{eq:CCC} imply that
\begin{equation}\label{muest}
\left|\boldsymbol\mu_a\right| \leq  |\utau|/\widetilde C + C  |\utau| \leq \underline C|\utau|/c'. 
\end{equation}

Define $\boldsymbol \mu_v \in D(\Gamma)$ such that $\boldsymbol \mu_v[p] := \ud[(\hYtdz)_{\tilde u},p]$ (recall \eqref{greatdest}). By \ref{greatdprop}, for each $p \in V(\Gamma)$,
\begin{equation}\label{muestd}
\left|d[p] - \boldsymbol \mu_v[p]\right| \leq C\underline C |\utau|^{1+(n-3+\gamma)/(n-1)} \leq C \underline C |\utau|^{1+ \frac 1{n-1}+ \frac{\gamma-1}{n-1}} \leq |\utau|^{1+\frac 1{n-1}}.
\end{equation}
We use the procedure above to define the map 
\[
\mathcal J(u, d, \bxi) = (u', d - \boldsymbol \mu_v, \boldsymbol \mu_a).
\] Then by \eqref{wprime}, \eqref{muest}, \eqref{muestd}, $\mathcal J(u, d, \bxi) \in B$ and the map $\mathcal J: B \to B$ is well defined. Moreover, for some $\beta' \in (0, \beta)$, $B$ is a compact, convex subset of $C^{2,\beta'}\loc(M) \times D(\Gamma) \times Z(\Gamma)$ and one can easily check that $\mathcal J$ is continuous in the induced topology. 
Thus, Schauder's fixed point theorem \cite[Theorem 11.1]{GiTr}, implies there exists a fixed point $(u', d', \boldsymbol \mu'_a) \in B$. 

By inspection, at a fixed point one has that
\begin{equation}\label{Hparallel}
1-H_{ u'} = w_v'-(\underline w_{d'})_v
\end{equation}and 
\begin{equation} 
\label{diszero}
\ud[(Y_{d',\boldsymbol \zeta})_{ u'},p] \equiv 0 \text{ for all } p \in V(\Gamma).
\end{equation} 

Recall that the function $w_v'-(\underline w_{d'})_v$ is supported on the interior of $\bigsqcup_{p \in V(\Gamma)} S[p]$. For a fixed $p \in V(\Gamma)$, 
\[
\left(w_v'-(\underline w_{d'})_v\right)\big|_{S[p]} = \sum_{i=0}^n \lambda_i w_i[p]
\]for some $\lambda_i \in \Real$. 
Let $N_{d',\boldsymbol \zeta, u'}$ denote the normal to the immersion $(Y_{d',\boldsymbol \zeta})_{u'}(M)$ 
and let $\widetilde F_i:= \Cn_n^{-\frac 12} N_{d',\boldsymbol \zeta, u'} \cdot \Be_{i+1}$. 
The definition of the global norm implies that $\| u':C^{2,\beta}(U[p],g)\| \leq C\underline C|\utau|$ and thus on $U[p]$,
\[
 |N_{d',\boldsymbol \zeta}- N_{d',\boldsymbol \zeta, u'} | \leq C \underline C |\utau|, \quad |\hF_i - \widetilde F_i| \leq C \underline C |\utau|.
\] 
Using then \eqref{hFdefeq} and \ref{wFdefeq} we have 
\[
\int_{U[p]} w_i[p] \widetilde F_i\geq \frac 12, \quad \quad \left|\int_{U[p]}  w_i[p] \widetilde F_j \right|\leq C\underline C|\utau| \text{ for } i \neq j.
\] 
Consider the $(n+1)\times (n+1)$ dimensional matrix $\mathcal M$ where $\mathcal M_{ij} = \int_{U[p]} w_j[p] \widetilde F_i$, $i,j = 0, \dots n$. 
The previous calculations demonstrate that $\mathcal M$ is invertible.  
The definition for $\ud[(Y_{d',\boldsymbol \zeta})_{ u'}, \cdot]$ along with \eqref{Hparallel} and \eqref{diszero} together imply that
\[
\int_{U[p]} \left(\sum_{j=0}^n \lambda_j w_j[p]\right) \widetilde F_i = 0 \text{ for all } i = 0, \dots, n.
\] 
Since $\mathcal M$ is invertible, this implies that
$\lambda_j=0$ for all $j= 0, \dots, n$ and thus $w_v'-(\underline w_{d'})_v\equiv 0$. 

By \eqref{Hparallel}, $1-H_{ u'} \equiv 0$ and thus the immersion $(Y_{d',\boldsymbol \zeta})_{ u'}:M \to \Real^{n+1}$ has mean curvature identically $1$.

Embeddedness follows when $\Gamma$ is pre-embedded and $\utau>0$ as in this case $Y_{d',\boldsymbol \zeta}(M)$ is embedded and $\| u'\|_{2,\beta,\gamma;d',\boldsymbol \zeta} \leq \widetilde C|\utau|$.
\end{proof}

We do not provide an extensive list of examples but instead point out that those provided in \cite[Section 2.2]{BKLD} 
and the finite topology examples in \cite[Section 4]{KapAnn} can easily be modified for the higher dimensional setting. 
For example \cite[Example 4.1]{KapAnn} remains valid and again produces infinitely many topological types with two ends. 
Moreover it is not hard to construct more examples by modifying those graphs to take advantage of the extra dimensions.  
In the embedded case also a finite number of topological types can easily be realized with $k$ ends, 
with the number of the topological types tending to $\infty$ as $k\to\infty$. 
Finally an easy parameter count demonstrates that there are  
$(k-1)(n+1)-\binom{n+1}2+\binom{n+1-k}2$ continuous parameters in these constructions in the absence of symmetry. 
Here the first summand reflects that we have $k-1$ Delaunay ends whose direction and $\tau$ parameter can be arbitrarily assigned, 
and the rest correct for the trivial changes induced by rotations.

\appendix

\section{Delaunay hypersurfaces} 
\label{DelSection}
The Delaunay surfaces are CMC surfaces in $\R^3$ discovered by Delaunay in 1841 \cite{Delaunay}. 
By analogy we call Delaunay (hyper)surfaces the $O(n)$-invariant CMC hypersurfaces in the $(n+1)$-dimensional 
Euclidean space $\R^{n+1}$. 
These are well known to form a one-parameter family. 
We derive them now in order to fix the notation and review their properties. 
We call the parameter of the family $\tau\in\R$, 
and by the $O(n)$ symmetry we can describe them by an immersion 
$Y_\tau: \R\times \Ssn \to \Rn$ 
of the form 
\begin{equation}
\label{DelImm}
Y_\tau(t,\bt) = (k_\tau(t), r_\tau(t) \Theta ) = (k(t), r(t) \Theta ), 
\end{equation}
where $r_\tau=r:\R\to\R_+$ and $k_\tau=k:\R\to\R$ depend on the parameter $\tau$ 
(although sometimes we omit $\tau$ for simplicity), 
and $\Theta$ are the standard coordinates in $\R^n$ restricted to the unit sphere $\Ssn\subset\R^n$. 
In order to determine the two unknown functions $r$ and $k$ 
we need two equations: 
First, we choose to impose the requirement that $Y_\tau$ is conformal. 
Since by \eqref{DelImm} $Y_\tau^*g= ((k')^2 + (r')^2) dt^2 + r^2 g_{\Ss^{n-1}}$,  
this is equivalent to the equation 
\begin{equation}
\label{conformal}
(r')^2+(k')^2=r^2. 
\end{equation} 
The second equation can be provided by the condition $H=1$ which amounts to a second order ODE. 
Instead we use a first integral of this equation as follows.

\subsection*{The force as a parameter}

Observe first that for an $n-1$ chain $\Gamma \subset Y_\tau(\Real \times \Ss^{n-1})$ 
with $\Gamma = \partial K$, 
$K \subset \Real^{n+1}$, the flux is defined as
\begin{equation}
\label{force}
 \mathrm{Force}(\Gamma) = \int_\Gamma \eta - n \int_K \nu.
\end{equation} 
Here $\eta$ is the unit conormal to $\Gamma$ in the surface and $\nu$ is the unit normal to $K$ 
with the appropriate orientation with respect to $\eta$. 
Let $\Gamma=\{Y_\tau(t_0,\bt):\bt\in\Ssn\}$ for a fixed $t_0$. 
We clearly have then in \eqref{force} by using \eqref{conformal} 
\[
\eta=\frac{\partial_t Y_\tau(t_0,\bt)}{\|\partial_t Y_\tau(t_0,\bt)\|} = r^{-1}(t_0)\,(\,k'(t_0) \,,\, r'(t_0)\bt\,), 
\qquad
\nu= (1,\mathbf 0). 
\]
We conclude that the flux is given by 
\begin{equation} 
\label{forcevec}
\mathrm{Force}(\Gamma) = \left( r^{n-2} k' - r^n \right) \omega_{n-1} \Be_1 =: \tau \omega_{n-1} \Be_1, 
\end{equation}
where the second equation is our definition of the parameter $\tau$ 
and $\omega_{n-1}$ denotes the volume of $\Ssn$. 
By combining then \eqref{forcevec} with \eqref{conformal} we have the equations 
\begin{equation} 
\label{eqs}
k'=r^2(1+\tau r^{-n}), 
\qquad 
r^4(1+\tau r^{-n})^2+ (r')^2 = r^2.
\end{equation}

\subsection*{The family of the Delaunay hypersurfaces and their geometry} 
The second equation in \eqref{eqs} easily implies that 
$\tau \in (-\infty, n^{-n}(n-1)^{n-1}]$. 
The boundary value $\tau= n^{-n}(n-1)^{n-1}$ 
clearly corresponds to the (only) solution 
$r \equiv (n-1)/n$ and so in this case the image of $Y_\tau$ 
is 
the $n$-cylinder $\mathbb R \times \mathbb S^{n-1}\!\left(\frac{n-1}{n}\right)$. 
In the case $\tau=0$ \eqref{eqs} has a unique solution up to translations given by 
\begin{equation} 
\label{Y0}
k(t)=\tanh t, \qquad
r(t)=\sech t, \qquad 
Y_0(t,\bt)\, =\, (\, \tanh t\, , \, \sech t\,\bt\,). 
\end{equation} 
$Y_0$ provides clearly a parametrization of 
$\Ssn\setminus\{(\pm1,\mathbf 0)$. 
To study now the case $\tau\ne0$ we introduce a function $w:\R\to\R$ defined by 
\begin{equation}
\label{req}
 r(t) = |\tau|^{1/n} e^{w(t)}. 
\end{equation}
By rewriting then \eqref{eqs} we have the equations 
\begin{equation}\label{weq}
\left\{\begin{array}{l}
(w')^2 = 
1-\left( r + \tau r^{1-n}\right)^2 
= 
1 -\tau^{2/n}\left(e^w\pm e^{(1-n)w}\right)^2,  
\\
w(0)=w_{\max}>0,\:\: w'(0)=0,\end{array}\right.  
\end{equation}
\begin{equation}
\left\{
 \begin{array}{l}
 k' = r^2(1+ \tau r^{-n}) = \tau^{2/n}\left(e^{2w} \pm e^{(2-n)w}\right) 
\\
 k(0)=0,
 \end{array}
\right . \end{equation}
where 
$\pm$ is the sign of $\tau$ and  
we assigned without loss of generality appropriate initial conditions. 
By analyzing the potential in \eqref{weq} it is easy to see that the maximum value $w_{\max}$ of $w$ is always positive as recorded. 

By the nature of the equation and based on the initial conditions we chose we have ensured  uniqueness of the solutions up to trivial changes. 
For 
$\tau \in (-\infty, n^{-n}(n-1)^{n-1})$,  
$w$ is not constant and 
periodic with period we will designate by $2\Pdo$.  

\begin{definition}[Periods of $Y_\tau$] 
\label{dPim} 
We define the domain period of $Y_\tau$ to be the period of $w$, $2\Pdo$,  
and the translational period of $Y_\tau$ to be 
$2\Pim:= k_\tau(2\Pdo)-2$.  
We have therefore 
$$ 
Y_\tau(t+2\Pdo,\bt) \equiv Y_\tau(t,\bt)+ (2+2\Pim,0). 
$$
\end{definition}

Clearly $w$ has a maximum at $t=0$ and a minimum at $t=\Pdo$ and it is an even function about both of those values of $t$.
We have then for any 
$\tau \in (-\infty,0)\cup (0, n^{-n}(n-1)^{n-1}]$ 
an immersion 
$Y_\tau: \Real \times \Ssn \to \Rn$ 
which by construction is rotationally symmetric about the $x_1$ axis
and its embeddedness is determined by the sign of $k'$:  
In fact, one quickly sees that for admissible $\tau>0$ the surface is embedded while for $\tau <0$ it is not. 
The image of $Y_\tau$ is foliated by spheres that sit in hyperplanes orthogonal to the $x_1$ axis, 
with maximum radius at $t=2m\Pdo$ and minimum radius at $t=(2m-1) \Pdo$ for $m \in \mathbb{Z}$.

By an easy calculation 
the Gauss map and first and second fundamental forms of $Y_{\tau}$ are given by 
\begin{equation}\label{FF}
\begin{aligned} 
\nu_\tau(t,\bt) =& \left( w', -(k'/r)\bt \, \right), 
\\
Y^*_\tau g =& ((k')^2 + (r')^2) dt^2 + r^2 g_{\Ss^{n-1}}=r^2(dt^2 + g_{\Ss^{n-1}}),
\\
A =& k'g_{\Ss^{n-1}}+ (k'' w' - k'(w'' + (w')^2))dt^2.
\end{aligned} 
\end{equation}

We simplify now the expression $(k'' w' - k'(w'' + (w')^2))$ so that we may quickly determine essential geometric quantities. 
First observe that for $r_1:=r+\tau r^{1-n}$ and $r_2:= r+(1-n)\tau r^{1-n}$,
\begin{equation}\label{w_derivs}
(w')^2 = 1-r_1^2, \qquad w''=-r_1r_2, \qquad k'=rr_1, \qquad k''=rw'(r_1+r_2).
\end{equation}Thus
\[
k'' w' - k'(w'' + (w')^2)=(rr_1+rr_2)(1-r_1^2)-rr_1(-r_1r_2 + 1-r_1^2)=rr_2.
\]
To determine $H$ observe that
\[
nH =\frac{(n-1)rr_1+rr_2}{r^2} = \frac{r(r+ \tau r^{1-n})(n-1) + r(r+ \tau r^{1-n} - n\tau r^{1-n})}{r^2} = n.
\]
This confirms that $H \equiv 1$ as expected.
For later use, we now determine $|A|^2$. 
A simple calculation gives
\begin{equation}\label{modA}
|A|^2 =r^{-2}\left( (n-1)r_1^2 + r_2^2\right) = n+n(n-1)\tau^2 r^{-2n}=n(1+(n-1)\tau^2 r^{-2n}).
\end{equation}

\subsection*{The limiting behavior as $\tau\to0$} 
We first determine the asymptotic behavior of the maximum and minimum values 
$r_{\tau}^{\max}$ and  $r_\tau^{\min}$ 
of $r_\tau$ which is as in \ref{DelImm} and \ref{req}. 

\begin{lemma}\label{rmaxminlemma}
For $|\tau|\ne0$ small enough in absolute terms we have 
\[
r_{\tau}^{\max}:= r_\tau(0)= 1-\tau+O(\tau^2), \quad \quad r_\tau^{\min}:= r_\tau(\Pdo) = |\tau|^{\frac 1{n-1}} + O(|\tau|^{2/(n-1)}).
\]
\end{lemma}
\begin{proof}
Note that $r_\tau$ has critical points where $w'=0$ and thus at critical points for $r$,
\[
 1= (r + \tau r^{1-n})^2 .
\] Note that $r+\tau r^{1-n} = 1$ for both critical values of $r$ when $\tau>0$ and for $r=r_{\tau}^{\max}$ when $\tau <0$. We set $f^+(r ) = r^n - r^{n-1}+\tau$ and look for zeros of that function in these settings. Since
$r+\tau r^{1-n}=-1$ when $r=r_{\tau}^{\min}$ and $\tau<0$, we look for zeros of the function $f^-(r) = r^n + r^{n-1}+\tau$ in this setting. 

Consider first the case where $\tau>0$. Then $f^+(1)=\tau>0$ and $(f^+)'(r)>0$ for all $r>1-\frac 1n$. Since $\tau -c\tau^2 < \frac 1n$, we will appeal to the intermediate value theorem. An easy calculation shows that $f^+(1-\tau) >0$ so we consider $f^+(1-\tau -c\tau^{1+\alpha})$ for $\alpha>0$. Then,
\begin{align*}
f^+(1-\tau - c\tau^{1+\alpha})&= \tau + (1-\tau -c\tau^{1+\alpha})^{n-1}(-\tau - c\tau^{1+\alpha})\\
&= \tau + \left(1-(n-1)\tau - c(n-1)\tau^{1+\alpha} + \binom{n-1}{2}\tau^2 + O(\tau^{2+\alpha})\right)(-\tau - c\tau^{1+\alpha})\\
&=\tau -\tau - c\tau^{1+\alpha}+ (n-1)\tau^2 +O(\tau^{2+\alpha}).
\end{align*}Thus, $f^+(r)<0$ if $\alpha\geq1$ and $c$ is sufficiently large and it follows that $r_{\tau}^{\max} = 1-\tau + O(\tau^2)$. 

To determine $r_{\tau}^{\min}$, we observe that $f^+(0)=\tau>0$, $(f^+)'<0$ near $0$, and $f^+(\tau^{\frac 1{n-1}})>0$. So we consider $r= \tau^{\frac 1{n-1}}(1+c\tau^\alpha)$ for $\alpha>0$. Then
\begin{align*}
f^+(r) &= \tau + (\tau^{\frac 1{n-1}}(1+c\tau^\alpha))^{n-1}(\tau^{\frac 1{n-1}}(1+c\tau^\alpha)-1)\\
&=\tau +( \tau+c(n-1)\tau^{1+\alpha} + O(\tau^{2\alpha+1}))(\tau^{\frac 1{n-1}} + c \tau^{\frac 1{n-1}+\alpha}-1)\\
&=\tau -\tau -c(n-1) \tau^{1+\alpha} + (1+c)\tau^{n/(n-1)}+O(\tau^{2+\alpha})
\end{align*}To make the final line negative, we need $\alpha \geq \frac 1{n-1}$ and $c$ sufficiently large. Therefore, $r_{\tau}^{\min}$ has the claimed asymptotics.

For $\tau<0$, we find the zero of $f^+(r)$ near $r=1$ and the zero of $f^-(r)$ near $r=0$. We leave the rest of the details to the reader, as they are similar to those outlined previously.
\end{proof}

Let $\rho_\tau:[-k(2\Pdo), k(2\Pdo)]\to \Real$ be the function
such that $\rho_\tau(x_1)=r_\tau(k(t_1))$ where $k(t_1)=x_1$. Then one can describe this fundamental piece of the surface as the rotation of $\rho_\tau$ about the $x_1$-axis and an easy computation 
of the mean curvature implies that
\begin{equation}\label{rhoode}
 nH = \frac{n-1}{\rho\sqrt{1+\rho'^2}} - \frac{\rho''}{(1+\rho'^2)^{\frac{3}{2}}}.
\end{equation}
Note that $\rho_\tau(0)= 1+O(\tau)$ and $\rho_\tau'(0) =0$. Thus, \eqref{rhoode} implies $\rho_\tau$ satisfies a second order ODE -- for
$\rho \neq 0$ -- with these initial conditions. We can immediately conclude the following lemma:
\begin{lemma}\label{radiuslemma}
 Let $\rho_0(x_1) = \sqrt{1-x_1^2}$ defined on $[-1,1]$. For any $\delt\in(0,1)$, there exists $\tau_\delt>0$ such that for $0<|\tau|<\tau_\delt$,
$\rho_\tau$ restricted to the interval $[-1+ \delt, 1- \delt]$ depends smoothly on $\tau$ and 
\[
 \|\rho_0-\rho_\tau:C^k([-1+ \delt, 1- \delt])\| \leq C(\delt,k)|\tau|.
\]
\end{lemma}
\begin{remark}\label{remark:btrelation}
Given the definition of the norms, it is straightforward to see that the previous lemma implies that, given any $b \gg 1$ there exists $\maxT>0$ such that for all $0<|\tau|\leq \maxT$,
\[
\|u:C^{k,\beta}([-b,b] \times \Ssn, g_\tau)\| \sim_{C(b,k,\beta)} \|u:C^{k, \beta}([-b,b]\times \Ssn,g_0)\|.
\]
\end{remark}
We also consider the geometry of the necks. To that end, we denote the conformal parametrization of the catenoid $Y_C:\Real \times \Ssn \to \Real^{n+1}$:
\[
Y_C(t,\Theta):=(k_C(t), r_C(t)\Theta)
\]where $(k_C')^2+(r_C')^2 =r_C^2$. Note that since $H \equiv 0$, the force vector is determined only by $\mathrm{Force}(\Gamma):=\int_\Gamma \eta$. Observe that at $r_C=1$, $\mathrm{Force}(\Gamma)= \omega_{n-1}\Be_1$. For any $\Gamma$, an explicit calculation implies that 
\[
\mathrm{Force}(\Gamma) = \frac 1{r_C}\int_\Gamma (k_C', r_C' \Theta) r_C^{n-1} dg_{\Ssn}=  {k_C'}r_C^{n-2} \omega_{n-1}\Be_1.
\]Therefore $k_C' = r_C^{2-n}$ and setting $r_C(t) = e^{w_C(t)}$, we have that $(w_C')^2= 1- r^{2-2n}$.

Denote the unit normal 
\[
\nu_C(t,\Theta) = (w_C', -(k_C'/r_C) \Theta).
\]Of particular interest will be the Jacobi field given by the dilation vector field. We denote it
\begin{equation}
f_C(t,\Theta) := Y_C \cdot \nu_C = w_C'k_C - k_C'.
\end{equation}

We determine estimates for $k_C$ by observing first that
\[
\frac{dk_C}{dr_C} = \frac 1{\sqrt{r_C^{2n-2}-1}}.
\]Using the above, the asymptotics of the immersion (as $r_C \to \infty$) are then given by a Gamma function. That is
\[
k_C(r_C) = \frac{\sqrt \pi \Gamma(\frac 12 \frac{3n-4}{n-1})}{(n-2)\Gamma(\frac 12 \frac{2n-3}{n-1})} - \frac 1{n-2}r_C^{2-n} +O(r_C^{4-3n})
\]where (using Maple) 
\begin{equation}\label{D:Tn}
T_n:=\int_1^\infty \frac {dr}{\sqrt{r^{2n-2}-1}} 
=\frac{\sqrt \pi \Gamma(\frac 12 \frac{3n-4}{n-1})}{(n-2)\Gamma(\frac 12 \frac{2n-3}{n-1})} . 
\end{equation} 
Therefore, for large $r$, 
\begin{align}\label{Dilation_Cat}
\notag f_C(r)&= - r^{2-n}+\sqrt{1-r^{2-2n}}\left(T_n - \frac{1}{n-2}r^{2-n} +O(r^{4-3n})\right)\\
& =T_n - \frac {n-1}{n-2} r^{2-n} + O(r^{2-2n}).
\end{align}

\begin{lemma}\label{Cat_lemma}
Given any $b \gg 1$, there exists $\maxT>0$ such that the following holds: 

For $0<|\tau|<\maxT$, define the immersion 
\[
Y_\tau^C(t,\Theta):=\left\{\begin{array}{ll}\frac 1{r_\tau^{\min}}Y_\tau(t+\Pdo,\Theta) & \text{if } \tau>0\\
\frac 1{r_\tau^{\min}}Y_\tau(-t+\Pdo,\Theta) & \text{if } \tau<0.
\end{array}\right. 
\]Then
\[
\|Y_\tau^C - Y_C:C^{k}([-b,b] \times \Ssn, Y_C^*g_{\Real^{n+1}})\| \leq C(k,b) |\tau|^{\frac 1{n-1}}.
\]
\end{lemma} An analogous result can be found in \cite[equation 6.32]{HaskKapCAGpq}. 
\begin{proof}
For $R>0$ and $|\beta|$ small, 
consider the immersion given by $Y_\beta(t,\Theta) = (k_\beta(t),  r_\beta(t)\Theta)$ 
of a domain in $\Real \times \Ssn$, where $Y_\beta(t,\Theta)$ is determined by  
\begin{equation*}
\left\{ \begin{array}{l}
(r_\beta')^2 + (k_\beta')^2 = r_\beta^2\\
 r_\beta^{n-2}k_\beta' = 1+\beta(r_\beta^n -1)\\
r_\beta(0)=1, \quad k_\beta(0)=0\end{array}\right.
\end{equation*} 
The family of immersions induced by this system varies smoothly in $\beta$. 
When $\beta=0$, the system of ODEs gives the conformal embedding of the unit catenoid. 
Therefore, for any $R>0$ there exists $\epsilon>0$ such that for all $\beta \in [-\epsilon, \epsilon]$, $Y_\beta$ is a smooth immersion on $[-2R,2R]\times \Ssn$. 

We now consider the immersion for $\beta \neq 0$, calculating the mean curvature using the same method as done after \eqref{FF}. In this case,
\[
A_\beta = k_\beta' g_{\Ssn} +\left(k_\beta'' \left(\frac{r_\beta'}{r_\beta}\right) -k_\beta'\left(\left(\frac{r_\beta'}{r_\beta}\right)'+\left( \frac{r_\beta'}{r_\beta}\right)^2\right)\right)dt^2
\]and with $r_{1,\beta}:= \beta r_\beta + (1-\beta)r_\beta^{1-n}, r_{2,\beta}:= \beta r_\beta + (1-\beta)(1-n)r^{1-n}_\beta$, we can simplify the $dt^2$ term exactly as before. Thus,
\[
nH= \frac{(n-1)r_\beta r_{1,\beta} + r_\beta r_{2,\beta}}{r_\beta^2}= n\beta.
\]It follows that $Y_\beta$ is a rotationally symmetric CMC immersion with mean curvature equal to $\beta$ and $r_\beta(0)=1$. Therefore, the immersion described by $|\beta| Y_\beta$ is a rotationally symmetric CMC surface with mean curvature $\pm 1$, where the sign is determined by the sign of $\beta$. Moreover, the radius of the meridian circle $|\beta|Y_\beta(0,\Theta)$ equals $|\beta|$. Thus $|\beta|=r_\tau^{\min}$. For $\beta>0$ the proof is complete as it follows immediately that $Y_\beta(t,\Theta)= \frac 1 \beta Y_\tau(t+\Pdo,\Theta)$ on the domain of interest. For $\beta<0$, we need only observe that mean curvature $-1$ corresponds to a change of direction for the unit normal. The change of direction corresponds to the parameter change $t \mapsto -t$, giving a sign change on $k_\beta'$, in the definition of $Y_\tau^C$.
\end{proof}

\subsection*{The change of parameter calculations}

\begin{lemma}\label{periodasymptotics} We have the following asymptotics as $\tau \to 0$:
\begin{equation*}
\lim_{\tau \to 0}\frac{\Pdo}{\log |\tau|} = -\frac1{n-1}, \qquad \lim_{\tau \to 0}|\tau| {\frac{d \Pdo}{d \tau} }=- \frac 1{n-1}.\end{equation*}
\end{lemma}

\begin{proof}First observe that we now consider the setting where $\tau>0$ as there are only minor changes necessary when $\tau<0$ that produce the same asymptotics.
 
We begin with the formulation
\[
\Pdo = \int^{r(\Pdo)}_{r(0)} \frac{dt}{dr}dr = \int^{r_\tau^{\max}}_{r_\tau^{\min}}\frac{dr}{r\sqrt{1-(r + \tau r^{1-n})^2}}.
\]

Given the estimates for $r_\tau^{\max}, r_\tau^{\min}$, one can easily check that for any $\epsilon>0$ there exists $0<\delta\ll 1$ such that for all $r \in (r_{\tau}^{\mathrm{min}}/\delta, r_\tau^{\max}\delta)$, the function $(r+\tau r^{1-n}) \in (-\epsilon,\epsilon)$. 
Fixing an $\epsilon>0$ we choose an appropriate $\delta>0$ and subdivide the integral into three parts so that
\begin{equation}\label{expand}
\begin{aligned} 
\Pdo = I_+ + I_\Lambda + I_-,  
\qquad \text{where} \quad 
I_+ :=& \int^{r_\tau^{\max}}_{r_\tau^{\max}\delta}\frac{dr}{r\sqrt{1-(r + \tau r^{1-n})^2}}, 
\\
I_\Lambda := \int^{r_\tau^{\max}\delta}_{r_\tau^{\min}/\delta}\frac{dr}{r\sqrt{1-(r + \tau r^{1-n})^2}}, 
\qquad 
I_- :=&  \int^{r_\tau^{\min}/\delta}_{r_\tau^{\min}}\frac{dr}{r\sqrt{1-(r + \tau r^{1-n})^2}}.
\end{aligned} 
\end{equation}
To estimate the contribution of $I_+$ in \eqref{expand}, consider 
the change of variables $u = -1+r/r_\tau^{\max}$. Then
\[
I_+ =\int_{\delta-1}^0 \frac{du}{(u+1)\sqrt{1-(r_\tau^{\max}(u+1) + \tau (r_\tau^{\max}(u+1))^{1-n})^2}}.
\] 
We expand 
\[
r_\tau^{\max}(u+1) + \tau (r_\tau^{\max}(u+1))^{1-n} = r_\tau^{\max} + \tau (r_\tau^{\max})^{1-n} + r_\tau^{\max}u + \tau (r_\tau^{\max})^{1-n}((1+u)^{1-n}-1).
\]Thus
\begin{align*}
\left(r_\tau^{\max}(u+1) + \tau (r_\tau^{\max}(u+1))^{1-n}\right)^2 =& 1 +  ( r_\tau^{\max}) ^2u^2 + 2(r_\tau^{\max} )^2 u +
2\tau(r_\tau^{\max} )^{2-n}u(1+u)^{1-n} \\ &+ \tau^2(r_\tau^{\max} )^{2-2n}((1+u)^{2-2n}-1).
\end{align*}
The square root term thus simplifies to
\[
f_\tau^{\max}(u):=r_\tau^{\max} \sqrt{-(u^2+2u+2\tau (r_\tau^{\max} )^{-n}u(1+u)^{1-n} + \tau^2(r_\tau^{\max} )^{-2n}((1+u)^{2-2n}-1))}.
\]Let
\[
F_\tau^{\max}(u):= \frac 1{(u+1)f_\tau^{\max}(u)}.
\]
When $\tau$ is sufficiently small, there exists $C_\delta>0$ independent of $\tau$ such that for $u\in [\delta-1,0)$ 
\[
|F_\tau^{\max}(u)| \leq \frac {C_\delta}{\sqrt{|u|}} \text{ and } \left|\frac {dF_\tau^{\max}}{d\tau}(u)\right| \leq \frac {C_\delta }{\sqrt{|u|}} .
\] 
Then by the dominated convergence theorem,
\begin{equation}\label{Contribution1}
\left|\frac d{d\tau} I_+  \right| 
\leq \int_{\delta-1}^0 \left|\frac{dF_\tau^{\max}}{d\tau}(u) \right|du\leq C_\delta.
\end{equation}

To estimate the contribution from the second integral in \eqref{expand}, recall that we have chosen $\delta$ so that $(r + \tau r^{1-n}) \in (-\epsilon, \epsilon)$. 
Therefore,
\[
I_\Lambda 
= (1+O(\epsilon))  \int^{r_\tau^{\max}\delta}_{r_\tau^{\min}/\delta}\frac{dr}{r}.
\] 
By inspection 
\[
 \int^{r_\tau^{\max}\delta}_{r_\tau^{\min}/\delta}\frac{dr}{r}= - \log |\tau|^{\frac 1{n-1}}+ 2 \log \delta + O(|\tau|^{\frac 1{n-1}}).
 \]Note also that
 \[
\left| \frac d{d\tau} \frac{1}{r\sqrt{1-(r + \tau r^{1-n})^2}}\right|=\left| \frac{r^{1-n}(r+\tau r^{1-n})}{r(1-(r + \tau r^{1-n})^2)^{3/2}}\right| \leq \epsilon r^{-n}.
 \]We calculate the derivative
 \begin{equation} 
\label{Contribution2}
\frac d{d \tau}  I_\Lambda 
=(1+O(\epsilon)) 
\left( \frac{\frac d{d \tau} r_\tau^{\max}}{r_\tau^{\max}} \right.\left.-  \frac{\frac d{d \tau} r_\tau^{\min}}{r_\tau^{\min}} 
+\epsilon \int^{r_\tau^{\max}\delta}_{r_\tau^{\min}/\delta} r^{-n} dr\right) 
= - \frac 1{n-1}|\tau|^{-1}\left(1 + O(\epsilon)\right). 
\end{equation}

Finally, to determine the contribution made by the third integral in \eqref{expand}, make the change of variables $u = -1+r/r_\tau^{\min}$ so that
\[
I_- = \int_{0}^{-1+1/\delta} \frac {du}{(u+1)\sqrt{1-(r_\tau^{\min}(u+1) + \tau (r_\tau^{\min}(u+1))^{1-n})^2}}.
\] 
We observe that the square root can be simplified to
\[
f_\tau^{\min}(u)= |\tau| (r_\tau^{\min})^{1-n} \sqrt{1-(1+u)^{2-2n}- 2 \tau^{-1}(r_\tau^{\min})^nu(1+u) -|\tau|^{-2}(r_\tau^{\min})^{2n}(u^2+2u)}
\] Let $F_\tau^{\min}:= \frac 1{(u+1)f_\tau^{\min}(u)}$. 
To determine the derivative of $F_\tau^{\min}$, recall that $r_\tau^{\min} + \tau (r_\tau^{\min})^{1-n}= \pm 1$. Therefore
\[
\frac d{d\tau} r_\tau^{\min} = \frac{(r_\tau^{\min})^{1-n}}{\tau(n-1)(r_\tau^{\min})^{-n}-1}= \frac{r_\tau^{\min} \tau^{-1}}{n-1}\left( 1+O(|\tau|^{\frac 1{n-1}})\right)
\]and thus, using $\pm, \mp$ to distinguish the two cases for the sign of $\tau$, 
\begin{align*}
\frac d{d\tau}(|\tau|^{-1} (r_\tau^{\min})^{n-1})& =  \mp|\tau|^{-2} (r_\tau^{\min})^{n-1}+ |\tau|^{-1}(n-1) (r_\tau^{\min})^{n-2}\frac{\pm r_\tau^{\min} |\tau|^{-1}}{n-1}\left( 1+O(|\tau|^{\frac 1{n-1}})\right)\\
&=  |\tau|^{-2}(r_\tau^{\min})^{n-1}O(|\tau|^{\frac 1{n-1}})\\
&= |\tau|^{-1}O(|\tau|^{\frac 1{n-1}}).
\end{align*}Using similar techniques, we compute
\[
|\tau|^{-1}(r_\tau^{\min})^{n-1}\frac d{d\tau}(\tau^{-1}(r_\tau^{\min})^{n})= |\tau|^{-1}\left( \frac 1{n-1}+ O(|\tau|^{\frac 1{n-1}})\right),
\]and 
\[
|\tau|^{-1}(r_\tau^{\min})^{n-1}\frac d{d\tau}(|\tau|^{-2}(r_\tau^{\min})^{2n})=|\tau|^{-1}\left(\frac{n+1}{n-1}+O(|\tau|^{\frac 2{n-1}})\right).
\]
Therefore
\[
\frac {dF_\tau^{\min}}{d\tau} = |\tau|^{-1}O(|\tau|^{\frac 1{n-1}})\left( \frac{|\tau| (r_\tau^{\min})^{1-n}}{(u+1)f_\tau^{\min}(u)}+ \frac{2u\frac{2}{n-1} + \frac{u^2+2u}{u+1}\frac{n+2}{n-1}O(|\tau|^{\frac 1{n-1}})}{(f_\tau^{\min}(u))^3}\right)
\]
For $0 \leq u \leq \frac 12$, there exists $C>0$, independent of $\tau$ and of $\delta$, so that
\[
|F_\tau^{\min}(u) | \leq \frac C{\sqrt{u}} \text{ and } |\tau|\left|\frac {dF_\tau^{\min}}{d\tau}\right| \leq \frac {O(|\tau|^{\frac 1{n-1}})}{\sqrt{u}}.
\]For $\frac 12 \leq u \leq \frac 1\delta -1$, there exists $C>0$, independent of $\tau$ and $\delta$, so that
\[
|F_\tau^{\min}(u) | \leq  C \text{ and }  |\tau|\left|\frac {dF_\tau^{\min}}{d\tau}\right| \leq  Cu\, {O(|\tau|^{\frac 1{n-1}})}.
\]Therefore,
\begin{equation}\label{Contribution3}
\left| \frac d{d\tau} \int_0^{\frac 1\delta -1} F_\tau^{\min} \right| \leq -C\frac{1}{\delta^2|\tau|} O(|\tau|^{\frac 1{n-1}}).
\end{equation}
Combining \eqref{Contribution1}, \eqref{Contribution2}, and \eqref{Contribution3}, we observe that
\[
\frac{d\Pdo}{d\tau} = |\tau|^{-1} \left( - \frac 1{n-1} + C\delta^{-2}O(|\tau|^{\frac 1{n-1}}) + O(\epsilon)+C_\delta|\tau| \right).
\]
Noting that $\epsilon>0$ can be chosen arbitrarily small, taking $\tau \to 0$ implies the result.
\end{proof}

\begin{lemma} 
Let $Y_\tau$ be a Delaunay immersion with axis along the $x_1$-axis. 
For $\phi \in C^\infty(M)$, let $(Y_{\tau})_\phi$ be a CMC graph over $Y_\tau$ possessing the same axis of symmetry as $Y_\tau$. The force through the meridian $\{t\} \times \Ssn$ of $(Y_\tau)_\phi$ is given by
\begin{equation}\label{linearizedforce}
\mathrm{Force}(\Gamma)= \omega_{n-1} \left(\tau + r^{n-2}\left(\phi' \hf_0 - \phi \hf_0'\right) + h.o.t.\right) \Be_1
\end{equation} 
Here h.o.t. stands for terms that are quadratic and higher in $\phi$ and its derivatives 
and 
$\hf_0:=  \nu_\tau \cdot \Be_1 = \frac{\ovr'}{\ovr}$. 
\end{lemma}

\begin{proof}
Recall that \[
\left(Y_{\tau}\right)_\phi(x):= Y_{\tau}(x) + \phi(x)N_{Y_{\tau}}(x)
\]for $\phi \in C^\infty(M)$. Let $\Gamma \subset \left(Y_\tau\right)_\phi$ be an $n-1$ chain with a fixed value for $x_1$ and let $K$ be the $n$ sphere in this hyperplane such that $\partial K=\Gamma$. The radius of $\Gamma$ is $r- \phi r_1$. To compute $\mathrm{Force}(\Gamma)$, observe that for $r_2:= r_1'/w'$,
\begin{align*}
\partial_t((Y_{\tau})_\phi)&=\left(\ovk'+\phi'\hf_0+ \phi \hf_0', \left(\ovr'-\frac{\ovk'\phi'}{\ovr} + \frac{\ovk'\hf_0-\ovk''}{\ovr}\phi\right)\bt \right)\\
&=\left(\ovk'+\phi'\hf_0+ \phi \hf_0', \left(\ovr\hf_0 -\ovr_1\phi'-\hf_0\ovr_2\phi\right)\bt\right),
\end{align*}
\[
|\partial_t((Y_\tau)_\phi)|^2= \ovr^2- 2\phi \ovr\ovr_2+h.o.t.
\]
Thus
\begin{align*}
\mathrm{Force}(\Gamma) &= \int_\Gamma \eta - n\int_K N_K\\
&= \left(\left(\ovr_1+\frac{\phi'\hf_0}\ovr + \frac{\phi \hf_0'}\ovr +\frac{\phi \ovr_2\ovr_1}\ovr+ h.o.t.\right) (\ovr-\phi \ovr_1)^{n-1} -(\ovr-\phi \ovr_1)^n\right)\omega_{n-1}\Be_1\\
&=(\ovr- \phi \ovr_1)^{n-1}\left(\ovr + \tau \ovr^{1-n} +\frac{\phi'\hf_0}\ovr + \frac{\phi \hf_0'}\ovr +\frac{\phi \ovr_2\ovr_1}{\ovr} - \ovr + \phi \ovr_1+h.o.t.\right) \omega_{n-1}\Be_1\\
&=\left(\tau +\ovr^{n-2}\left(\phi'\hf_0 - \phi \hf_0' \right)+ h.o.t\right) \omega_{n-1} \Be_1
\end{align*}
\end{proof}

\begin{lemma}
We have the following asymptotics as $\tau \to 0$ (recall \ref{D:Tn}):
\begin{equation}\label{Pim_est}
\lim_{\tau \to 0}{|\tau|^{\frac 1{1-n}}}\Pim = {T_n}; \quad \quad \lim_{\tau \to 0}{ |\tau|^{\frac{n-2}{n-1}}} \frac{d\Pim}{d\tau}= \frac{T_n}{n-1}.
\end{equation} 
\end{lemma}

\begin{proof} 
We first define motions on the domain cylinders and the ambient $\Real^{n+1}$. 
Let $T_{x}, R_x:\Real \times \Ss^{n-1} \to \Real \times \Ss^{n-1}$ such that $T_{x}(t,\Theta) = (t+x, \Theta)$ and $R_x(t,\Theta) = (2x-t,\Theta)$. 
Thus, $T_x$ is a translation by $x$ and $R_x$ a reflection about $(x,\Theta)$ on the cylindrical domain. 
Let $\hat T_{y} ,\hat R_y: \Real^{n+1} \to \Real^{n+1}$ such that  $\hat T_{y}(p):= p+y\Be_1$ and $\hat R_y(\sum_{i=1}^{n+1}a_i\Be_i):= (2y -a_1)\Be_1+ \sum_{i=2}^{n+1}a_i\Be_i$. Thus $\hat T_{y}$ is a translation by $y\Be_1$ and $\hat R_y$ is a reflection of the first component about $y \Be_1$.
Consider an admissible $\tau$, a $\sigma$ near $\tau$,  
and three immersions 
$X:= Y_\tau$, $Y:= Y_\sigma$ and $Z:= \hat T_{  \Pim-\hat{\mathbf p}_\sigma}\circ Y \circ T_{\underline{ \mathbf p}_\sigma-\Pdo}$ 
defined 
on $[-2\Pdo, 2\Pdo] \times \Ssn$ by 
\begin{align*}
\hat R_0 \circ X &= X \circ R_0; &\hat R_0 \circ Y &= Y \circ R_0; \\
\hat R_{1+\Pim} \circ X &= X \circ R_{\Pdo}; &\hat R_{1 + \Pim} \circ Z &= Z \circ R_{\Pdo} .
\end{align*}

For $\sigma$ sufficiently close to $\tau$ there exist smooth functions $\phi_\sigma, \psi_\sigma$ and smooth diffeomorphisms $D_Y, D_Z$, of the identity map on $[-2\Pdo,2\Pdo] \times \Ssn$, such that
\[
Y= X_{\phi_\sigma} \circ D_Y \qquad Z = X_{\psi_\sigma} \circ D_Z.
\]By the symmetries, $\phi_\sigma'(0)=0, \psi_\sigma'(\Pdo)=0$. Also note that $\phi_\sigma, \psi_\sigma, D_Y, D_Z$ all depend on and are smooth in $\sigma$. We now calculation the linearization of the normal part of each of these variations. Expanding the immersion for $Z$, we observe that
\[
Z =  X_{\psi_\sigma} \circ D_Z = X \circ D_Z + (\psi_\sigma \circ D_Z) (\nu_X \circ D_Z).
\]Note that $D_Z: \Real^2 \times \Ssn \to \Real^2 \times \Ssn$ where $(\sigma, t, \Theta) \mapsto (\sigma, T_1(\sigma,t,\Theta), T_2(\sigma,t,\Theta))$ so 
\[
\frac \partial{\partial \sigma}(\psi_\sigma \circ D_Z) = \frac{\partial \psi_\sigma}{\partial \sigma} + \frac{\partial \psi_\sigma}{\partial T_1}\frac{\partial T_1}{\partial \sigma}+\frac{\partial \psi_\sigma}{\partial T_2}\frac{\partial T_2}{\partial \sigma}.
\]Evaluating this at $\sigma=\tau$, note that $\psi_\tau \equiv 0$ so $\frac{\partial \psi_\sigma}{\partial T_1}=\frac{\partial \psi_\sigma}{\partial T_2}=0$ for $\sigma=\tau$ and thus
\[
\frac \partial{\partial \sigma}\bigg|_{\sigma=\tau}(\psi_\sigma \circ D_Z) =\frac{\partial \psi_\sigma}{\partial \sigma}\bigg|_{\sigma=\tau}. 
\]Further, observe that both $\frac \partial {\partial \sigma}(X \circ D_Z), \frac \partial {\partial \sigma}(\nu_X \circ D_Z)$ are in the tangent space of $X$ when $\sigma=\tau$ and thus
\[
\frac {\partial Z}{\partial \sigma}\bigg|_{\sigma = \tau}   \cdot \nu_X= \frac{\partial \psi_\sigma}{\partial \sigma}\bigg|_{\sigma=\tau}.
\] Performing a similar calculation for $Y$, we define
\[
\hat \phi_\sigma:= \frac {\partial\phi_\sigma}{\partial\sigma}\bigg|_{\sigma = \tau}= \frac {\partial Y}{\partial \sigma}\bigg|_{\sigma=\tau} \cdot \nu_X; \qquad \hat \psi_\sigma:= \frac {\partial\psi_\sigma}{\partial\sigma}\bigg|_{\sigma = \tau}=\frac {\partial Z}{\partial \sigma}\bigg|_{\sigma=\tau} \cdot \nu_X.
\]

We now determine an equation for $\frac{d\hat {\mathbf p}_\sigma}{d\sigma}\bigg|_{\sigma=\tau}$ using the second immersion of $Z$. Since $Y$, and thus $Z$, depends smoothly on $\sigma$, we rewrite the immersion as
\[
Z(\sigma,t,\Theta)=(  \Pim-\hat{\mathbf p}_\sigma ) \Be_1 + Y \circ T_{\underline{ \mathbf p}_\sigma-\Pdo}(\sigma,t,\Theta).
\]
Since
\[
 Y\circ T_{\underline{ \mathbf p}_\sigma-\Pdo}(\sigma,t,\Theta)=Y(\sigma,\underline{ \mathbf p}_\sigma-\Pdo+ t,\Theta):= Y(T_1(\sigma,t,\Theta),T_2(\sigma,t,\Theta),T_3(\sigma,t,\Theta)),
\]we calculate
\[
\frac{\partial }{\partial \sigma} Y\circ T_{\underline{ \mathbf p}_\sigma-\Pdo} = \frac{\partial Y}{\partial T_1}\frac{\partial T_1}{\partial \sigma} + \frac{\partial Y}{\partial T_2}\frac{\partial T_2}{\partial \sigma} +\frac{\partial Y}{\partial T_3}\frac{\partial T_3}{\partial \sigma}. 
\]The second two terms above are in the tangent space of $X$ when $\sigma = \tau$. Moreover, $\frac{\partial Y}{\partial T_1}\frac{\partial T_1}{\partial \sigma} = \frac{\partial Y}{\partial \sigma}$. Therefore, comparing the normal components of the two linearizations of $Z$ we observe that
\begin{equation}\label{First_hat_p}
 \hat \psi_\sigma( t) = \hat \phi_\sigma( t) -\frac{d\hat {\mathbf p}_\sigma}{d\sigma}\bigg|_{\sigma=\tau}\Be_1 \cdot \nu_X= \hat \phi_\sigma( t) -\frac{d\hat {\mathbf p}_\sigma}{d\sigma}\bigg|_{\sigma=\tau}\hf_0( t).
\end{equation}

Consider the force calculation of \eqref{linearizedforce} where $\phi$ is replaced by either $\hat \phi_\sigma$ or $\hat \psi_\sigma$. 
Note that the functions $r, w$ are independent of $\sigma$. 
If we differentiate with respect to $\sigma$ and evaluate at $\sigma=\tau$ we observe that
\[
r^{n-2}(\hat \phi'_\sigma \hf_0 - \hat \phi_\sigma \hf_0') =1 \text{ and } r^{n-2}(\hat \psi'_\sigma \hf_0 - \hat \psi_\sigma \hf_0') =1.
\]Therefore, for any region $[x,y]$ where $\hf_0(t) \neq 0$, we determine that
\[
\left( \frac{\hat \phi}{\hf_0}\right)' = \left( \frac{\hat \psi}{\hf_0}\right)' = \frac{r^{2-n}}{\hf_0^2}
\]and thus
\[
\frac{\hat \phi_\sigma}{\hf_0}(t) = \frac{\hat \phi_\sigma}{\hf_0}(x)+ \int_x^t \frac{r^{2-n}(s)}{\hf_0^2(s)} ds, \quad \quad \frac{\hat \psi_\sigma}{\hf_0}(t) = \frac{\hat \psi_\sigma}{\hf_0}(y)- \int_t^y \frac{r^{2-n}(s)}{\hf_0^2(s)} ds.
\]
Therefore, substituting into \eqref{First_hat_p}, for any region $[x,y]$ where $\hf_0(t) \neq 0$,
\begin{equation}\label{Second_hat_p}
\frac{d\hat {\mathbf p}_\sigma}{d\sigma}\bigg|_{\sigma=\tau} = \int_x^y  \frac{r^{2-n}(s)}{\hf_0^2(s)} ds +  \frac{\hat \phi_\sigma}{\hf_0}(x)- \frac{\hat \psi_\sigma}{\hf_0}(y).
\end{equation}By a change of variables, we rewrite (recalling $\hf_0<0$ on the domain of integration), 
\[
\int_x^y  \frac{r^{2-n}(s)}{\hf_0^2(s)} ds= \int^{r(x)}_{r(y)} \frac {r^{2-n}}{r |\hf_0|^3(r)}dr= \int^{r(x)}_{r(y)} \frac {r^{1-n}}{|\hf_0|^3(r)}dr.
\]We now choose $x,y$ so as to clearly estimate all of the terms in \eqref{Second_hat_p}.

Recall that $\hf_0 = -\sqrt{1-(r+\tau r^{1-n})^2}$. Therefore, as in the proof of the asymptotics for $\Pdo$, for any $\epsilon>0$, we can choose $0<\delta \ll 1$ such that $|\hf_0|^{-3}(r) \in [1-\epsilon,1]$ for all $r \in [r_\tau^{\min}/\delta,r_\tau^{\max}\delta]$. Choose, $0<x<y<\Pdo$ so that $r(y)= r_\tau^{\min}/\delta, r(x) = r_\tau^{\max}\delta$. Then 
\[
 \frac{1-\epsilon}{n-2}\left(\frac{r_\tau^{\min}}{\delta}\right)^{(2-n)} + O(1) \leq \int^{r_\tau^{\max}\delta}_{r_\tau^{\min}/\delta} \frac {r^{1-n}}{|\hf_0|^3(r)}dr \leq \frac {1}{n-2}\left(\frac{r_\tau^{\min}}{\delta}\right)^{(2-n)} + O(1).
\]

Observe that the construction of $\hat \phi$ implies ${\hat \phi}' (0)=0$, $|\hat \phi(0)|\leq C$ independent of $\tau$. Moreover, $\hat \phi(t)$ satisfies $\hat \phi' \hf_0 - \hat \phi \hf_0'= r^{2-n}$ and on $[0, x]$ the coefficients of this ODE are uniformly bounded independent of $\tau$. 
Thus, for $\tau$ sufficiently small, $\hat \phi(x)/\hf_0(x)=O(1)$.

Finally, we consider the value of $\hat \psi_\sigma(y)/\hf_0(y)$.  Note that for any $R>0$, Lemma \ref{Cat_lemma} implies that for $\tau$ sufficiently small, on $[\Pdo -R,\Pdo+R] \times \Ssn$, $\hat \psi_\sigma$ behaves like a multiple of the dilation Jacobi field on the unit catenoid. 
Indeed, 
\[
\hat \psi_\sigma(t) = cf_C(t-\Pdo)(1+O(|\tau|^{\frac 1{n-1}})) 
\] 
for some constant $c$. 
By calculation, $\hat \psi_\sigma(\Pdo) = \pm\frac 1{n-1}|\tau|^{\frac{2-n}{n-1}}(1+O(|\tau|^{\frac 1{n-1}}) \,)$. 
The sign on this term is positive for $\tau>0$ since $\psi_\sigma(\Pdo) \approx r_\tau^{\min} - r_\sigma^{\min}$ as the normal points inward. 
For $\tau<0$, the normal points outward and $\psi_\sigma$ and thus $\hat \psi_\sigma$ changes sign. 

Since $f_C(0)=1$, we observe that $c =  -\frac 1{n-1}|\tau|^{\frac{2-n}{n-1}}$. (The sign is negative since when $\tau >0$ the normals for $Y_\tau,Y_C$ agree but when $\tau <0$ the normals point in opposite directions). By substitution, using \eqref{Dilation_Cat}, 
\[
f_C(y-\Pdo) = T_n(1+O(\delta^{2n-2})) - \frac{n-1}{n-2}\delta^{n-2} 
\]It follows that 
\begin{align*}
\hat \psi_\sigma(y)&= -\frac 1{n-1}|\tau|^{\frac{2-n}{n-1}}T_n(1+O(\delta^{2n-2}))+ \frac 1{n-2}|\tau|^{\frac{2-n}{n-1}}\delta^{n-2} 
\end{align*}
Inserting the estimates into \eqref{Second_hat_p}, we observe that
\[
\frac{d\hat {\mathbf p}_\sigma}{d\sigma}\bigg|_{\sigma=\tau} = \frac 1{n-1}|\tau|^{\frac{2-n}{n-1}}T_n(1+o(\delta))+ O(1).
\]
\end{proof}

\section{Quadratic Estimates}
\label{quadapp}
For completeness, we include here a proposition we will need. The proposition is analogous to the ones in the appendices of \cite{KapYang,HaskKap}. We have adapted it here for our purposes, though the proof is identical to that in \cite{KapYang}. 
Let $X:D \to U$ be an immersion of a disk of radius $1/10$ in $\Real^{n}$ into an open cube $U\subset \Real^{n+1}$ equipped with a metric $g$. Assume $\dist_g(X(D), \partial U)>1$ and there exists $c_1>0$ such that:
\begin{equation}\label{quadconditions}
\|\partial X: C^{2,\beta}(D,g_0)\| \leq c_1, \:\:\:\|g_{ij}, g^{ij}:C^{4, \beta}(U,g_0)\|\leq c_1, \:\:\:g_0 \leq c_1X^*g,
\end{equation}where here $\partial X$ represents the partial derivatives of the coordinates of $X$, $g^{ij}$ are the components of the inverse of the metric $g$, and $g_0$ denotes the standard Euclidean metric on $D$ or $U$ respectively. We note that \eqref{quadconditions} can be arranged by an appropriate magnification of the target, which we will exploit in order to make use of the following proposition. 

Let $\nu:D \to \Real^{n+1}$ be the unit normal for the immersion $X$ in the $g$ metric. Given a function $\phi:D \to \Real$ which is sufficiently small, we define $X_\phi:D \to U$ by
\begin{equation}\label{xphi}
X_\phi(p):= \exp_{X(p)}(\phi(p)\nu(p))
\end{equation}where here $\exp$ is the exponential map with respect to the $g$ metric. Then the following holds:
\begin{prop}\label{quadboundsunscaled}
There exists a constant $\epsilon(c_1)>0$ such that if $X$ is an immersion satisfying \eqref{quadconditions} and the function $\phi:D \to \Real$ satisfies
\[
\|\phi:C^{2,\beta}(D,g_0)\| \leq \epsilon(c_1)
\]then $X_\phi:D \to U$ is a well defined immersion by \eqref{xphi} and satisfies
\[
\|X_\phi - X - \phi \nu:C^{1,\beta}(D,g_0)\|\leq C(c_1)\|\phi:C^{2,\beta}(D,g_0)\|^2
\]and
\[
\|H_\phi - H - \mathcal L_{X^*g} \phi:C^{0, \beta}(D,g_0)\|\leq C(c_1)\|\phi: C^{2, \beta}(D,g_0)\|^2.
\]Here $H=tr_g A$ is the mean curvature of $X$, defined with respect to the metric $X^*g$ where $A$ is the second fundamental form, $H_\phi$ the mean curvature of $X_\phi$, and $\mathcal L_{X^*g}:= \Delta_g + |A|^2$.
\end{prop}

\begin{proof}
That the linear terms are as stated is well known and follows by a straightforward calculation we omit. The nonlinear terms are given by expressions of monomials consisting of contractions of derivatives of $X, g_{ij}, g^{ij}$, the exponential map, and $\phi$. This implies both the existence results and the estimate on the nonlinearity.
\end{proof}

\section{An easy result for flat annuli}
\label{annuli}
Let $\Lambda:= [\ssin, \sout] \times \Ssn$ and $g_A:= ds^2 + s^2 g_{\Ssn}$. Let $\Cin:=\{\ssin\} \times \Ssn$ and $\Cout:= \{\sout\}\times \Ssn$. The following result is well known, though we could not find a reference. We sketch the steps of the proof and leave a few technicalities to the reader.

\begin{prop}\label{flatannuluslinear}Given $\beta \in(0,1)$ and $\gamma \in (1,2)$ there exist linear maps $\mathcal{R}^{\mathrm{out}}_A,\mathcal R^{\mathrm{in}}_{A}:C^{0,\beta}( \Lambda,g_A) \to C^{2, \beta}( \Lambda,g_A)$ such that if $E \in C^{0, \beta}( \Lambda, g_A)$ then either one of the following can occur:
\begin{enumerate}[(i)]
\item if $V^{\mathrm{out}} = \mathcal{R}^{\mathrm{out}}_A(E)$ then 
\begin{itemize} 
\item $\mathcal L_{g_A} V^{\mathrm{out}}=E$ on $\Lambda$.
\item $V^{\mathrm{out}}|_{\Cout} \in \mathcal H_1[\Cout]$ and vanishes on $\Cin$.
\item $\|V^{\mathrm{out}}:C^{2,\beta}( \Lambda,s,g_A, s^{\gamma}) \|\leq C(\beta, \gamma)\|E:C^{0,\beta}( \Lambda,s,g_A,s^{\gamma-2}) \|$.
\end{itemize}
 \item  if $V^{\mathrm{in}} = \mathcal{R}^{\mathrm{in}}_A(E)$ then 
\begin{itemize}
\item $\mathcal L_{g_A} V^{\mathrm{in}}=E$ on $\Lambda$.
\item $V^{\mathrm{in}}|_{\Cin} \in \mathcal H_1[\Cin]$ and vanishes on $\Cout$.
\item $\|V^{\mathrm{in}}:C^{2,\beta}( \Lambda,s,g_A, s^{2-n-\gamma}) \|\leq C(\beta, \gamma)\|E:C^{0,\beta}( \Lambda,s,g_A, s^{-n-\gamma}) \|$.
\end{itemize}
\end{enumerate}
\end{prop} 
Recall that $\mathcal H_1[\Cin],\mathcal H_1[\Cout], \phi_i$ are defined in \ref{LowHdef}.
\begin{proof}First observe that $\mathcal L_{g_A} = \Delta_{g_A} = \partial_{ss} + \frac{n-1}{s} \partial_s + \frac{1}{s^2}\Delta_{\Ssn}$. Recall \ref{phidef}. 
Let $E:= E_0 + E_1 + E_{2}$ where $E_0(s,\bt) = E_0(s)$, $E_1=\sum_{i=1}^nE_i(s)\phi_i(\bt)$, and $E_{2}= E-E_0 -E_1$. 

Let $r_k = 2^{-k}$ and let $A_k := B_{r_k} \backslash B_{r_{k+1}} \subset \Real^{n}$ and $\tilde A_k:= A_{k-1} \cup A_k \cup A_{k+1}$. Let $\{\psi_k\}_{k \in \mathbb N}$ be a partition of unity of $B_1$ such that $\psi_k \equiv 1$ on $A_k$ and $\psi_k =0$ on $\Real^n \backslash \tilde A_k$. Finally, let $E^k:= \psi_k E$ and $E^k_i = \psi_k E_i$ for $i=1,2,3$. 

We begin by considering $E_2$. Let $u_k$ satisfy
\begin{equation*}
\left\{ \begin{array}{ll}
\mathcal L_{g_A} u_k = E^k_2 & \text{ in }  \Lambda,\\
u_k = 0 & \text{ on } \partial  \Lambda
\end{array} \right.
\end{equation*} Then since, for a fixed $c \in \Real$, $\mathcal L_{c^2g_A} u_k = c^{-2}E^k_2$, by applying De Giorgi-Nash-Moser techniques and then Schauder theory we determine
\[
\|u_k:C^{2,\alpha}(\tilde A_k, s, g_A)\| \leq C(b) r_k^2\|E_k:C^{0,\alpha}(\tilde A_k, s, g_A)\|.
\]As $\mathcal L_{g_A} u_k = 0$ on $ \Lambda \backslash \tilde A_k$ we observe that on each ``tail'', the worst decay for $u_k$ comes from the terms of the form $a_k^{\mathrm{in}, \mathrm{out}} s^2 + b_k^{\mathrm{in}, \mathrm{out}}s^{-n}$. The Dirichlet conditions imply that 
\[
a_k^{\mathrm{in}}=-b_k^{\mathrm{in}}(\ssin)^{-n-2} \text{   and   }a_k^{\mathrm{out}}=- b_k^{\mathrm{out}}(\sout)^{-n-2}.
\]Thus, referring to each tail as $u_k^{\mathrm{in}, \mathrm{out}}$, we note that
\[
u_k^{\mathrm{in}} = b_k^{\mathrm{in}}\left(s^{-n} - (\ssin)^{-n-2}s^2\right) \qquad \text{and} \qquad u_k^{\mathrm{out}}= b_k^{\mathrm{out}}\left(s^{-n} - (\sout)^{-n-2}s^2\right).
\]Moreover, the estimates imply
\[
|b_\ell^{\mathrm{in}}| \leq  {C\ssin^{n+2}}\frac{\|E_\ell:C^{0,\alpha}(\tilde A_\ell, s,g_A)\|}{1-r_\ell^{-n-2}\ssin^{n+2}},
\]
\[
|b_\ell^{\mathrm{out}}| \leq  Cr_\ell^{n+2}\frac{\|E_\ell:C^{0,\alpha}(\tilde A_\ell, s,g_A)\|}{1- r_\ell^{(n+2)}/\sout^{(n+2)}}.
\]Let $u_2:= \sum_\ell u_\ell$. Then $\mathcal L_{g_A} u_2 = E_2$. Finally, we consider the estimates. On any fixed dyadic $A_k$, note that
\[
u_2 = \sum_{m < k} u_m^{\mathrm{in}} + \sum_{m \geq k} u_m^{\mathrm{out}}.
\]Thus,
\begin{align*}
\|u_2:C^{2,\beta}(A_k, s,g_A,s^{2-n-\gamma})\| &\leq Cr_k^{n+\gamma-2} \left(\sum_{r_m > r_k}\|E_m:(A_m,s,g_A)\|\frac{r_k^{-n}\ssin^{n+2} + r_k^2}{1-r_m^{-n-2}\ssin^{n+2}}\right.\\
&\indent+\left.\sum_{r_m \leq r_k}\|E_m:(A_m,s,g_A)\|r_m^{n+2}\frac{r_k^{-n}+ \sout^{-n-2}r_k^2}{1-(r_m/\sout)^{n+2}}\right)\\
&\leq C\|E:C^{0,\beta}( \Lambda, s,g_A, s^{-n-\gamma})\|\cdot\\
&\indent \left(\sum_{r_m>r_k}\frac{\left(\frac{r_k}{r_m}\right)^\gamma \frac{\ssin^{n+2}}{r_k^2r_m^{n}}+\left(\frac{r_k}{r_m}\right)^{n+\gamma}}{1-r_m^{-n-2}\ssin^{n+2}}+ \sum_{r_m\leq r_k}\frac{\left(\frac{r_m}{r_k}\right)^{2-\gamma}+ \left(\frac{r_k}{\sout}\right)^{n+2}} {1-(r_m/\sout)^{n+2}}\right).
\end{align*}The estimate with decay $s^\gamma$ works similarly.

We now consider the low harmonic terms. 
Let $u_{0,\ell}^{\mathrm{out}}$ denote the solution to the initial value problem $\partial_{ss}f + \frac{n-1}{s} \partial sf = E^{\ell}_0$ where $f(2^{-\ell})=f'(2^{-\ell})=0$ At $s=2^{-\ell -1}$, choose $u_{0, \ell}^{\mathrm{out}}(s)= \partial_su_{0,\ell}^{\mathrm{out}}(s)=0$. Then $u_{0, \ell}^{\mathrm{out}} = a_0 + b_0 s^{2-n}$  on $ \Lambda \backslash B_{\ell}$ where $|a_0| \leq C2^{-2\ell} \|E_0^\ell:C^{0,\alpha}( \Lambda, s,g_A)\|$, $|b_0| \leq C2^{-n\ell}\|E_0^\ell:C^{0,\alpha}( \Lambda,s, g_A)\|$. Now consider $u_0^{\mathrm{out}}:=\sum_{\ell = 1}^\infty u_{0,\ell}^{\mathrm{out}}$. First, $\mathcal L_{g_A} u_0^{\mathrm{out}} = E_0$ on $ \Lambda$ and $u_0^{\mathrm{out}}= 0$ on $\Cin$, $u_0^{\mathrm{out}}|_{\Cout} \in \mathcal H_0[\Cout]$. Moreover, consider any dyadic annulus $\tilde A_k \subset  \Lambda$. Then on $A_k$, $u_0^{\mathrm{out}} = \sum_{r_\ell < r_k} u_{0, \ell}^{\mathrm{out}}$. We estimate
\begin{align*}
\|u_{0, \ell}^{\mathrm{out}}:(A_k, s,g_A, s^\gamma)\|&\leq C \frac{r_\ell^2}{r_k^{\gamma}}\|E_\ell:C^{0,\beta}(A_\ell,s,g_A)\|\\
&\leq C \frac{r_\ell^2r_k^{\gamma-2}}{r_k^\gamma}\|E:C^{0, \beta}( \Lambda, s,g_A, s^{\gamma-2})\| \\
&\leq C \left(\frac{r_\ell}{r_k}\right)^2\|E:C^{0, \beta}( \Lambda,s, g_A, s^{\gamma-2})\|.
\end{align*}

For $u_0^{\mathrm{in}}$ we perform the identical construction but now prescribe data so that each $u_{0, \ell}^{\mathrm{in}}$ so that $u_{0, \ell}^{\mathrm{in}} =0$ on $\Real^n \backslash B_{2^{-\ell+1}}$. 

On any $A_k$, $u_0^{\mathrm{in}} = \sum_{r_\ell  >r_k} u_{0, \ell}^{\mathrm{in}}.$ Now
\begin{align*}
\|u_{0, \ell}^{\mathrm{in}}:(A_k, s,g_A, s^{2-n-\gamma})\| &\leq C\frac{r_\ell^nr_k^{2-n}}{r_k^{2-n-\gamma}}\|E_\ell:C^{0,\beta}(A_\ell,s,g_A)\|\\
& \leq C \frac{r_\ell^{n-n-\gamma}}{r_k^{-\ell}}\|E:C^{0, \beta}( \Lambda, s,g_A, s^{-n-\gamma})\|\\
& \leq C \left(\frac{r_k}{r_\ell}\right)^\gamma\|E:C^{0, \beta}( \Lambda,s, g_A, s^{-n-\gamma})\| .
\end{align*}
For $u_1^{\mathrm{out}, \mathrm{in}}$ we apply similar techniques as for $u_0$ to produce the necessary estimates. Taken together, these imply the result for $V^{\mathrm{in}}:= u_0^{\mathrm{in}}+u_1^{\mathrm{in}} + u_2$ and $V^{\mathrm{out}} := u_0^{\mathrm{out}} + u_1^{\mathrm{out}} + u_2$. 
\end{proof}

\bibliographystyle{amsplain}
\bibliography{Main_Biblio}
\end{document}